\documentclass{amsart}
\usepackage[latin1]{inputenc}
\usepackage{amsthm} %
\usepackage{amsmath}
\usepackage{amssymb}
\usepackage{mathrsfs}
\usepackage{enumitem}
\usepackage{mathtools}
\usepackage{multirow,multicol,makecell}
\usepackage{tikz}
\usetikzlibrary{shapes.geometric}
\usetikzlibrary{arrows}
\newcommand{\midarrow}{\tikz \draw[-triangle 90] (0,0) -- +(.05,0);}

\usepackage[T1]{fontenc}

\usepackage[hidelinks]{hyperref}
\usepackage{graphicx, caption}
\usepackage{float}
\usepackage{wrapfig}
\usepackage{adjustbox}
\usepackage{subcaption}

\theoremstyle{plain}
\newtheorem{theorem}{Theorem}[section]  
\newtheorem{proposition}[theorem]{Proposition} 
\newtheorem{lem}[theorem]{Lemma} 
\newtheorem{cor}[theorem]{Corollary}

\theoremstyle{definition} 
\newtheorem{conjecture}[theorem]{Conjecture}

\DeclareMathOperator{\re}{Re}

\newcommand{\etalchar}[1]{$^{#1}$}

\newcommand{\C}{\mathbb{C}}
\newcommand{\N}{\mathbb{N}}
\newcommand{\R}{\mathbb{R}}
\newcommand{\Z}{\mathbb{Z}}
\newcommand{\F}{\mathbb{F}}

\renewcommand{\mod}[1]{\,(\mathrm{mod}\,#1)}
\renewcommand{\pmod}[1]{\,(\mathrm{mod}\,#1)}

\usepackage{color}
\definecolor{pink}{rgb}{1,.2,.6}
\definecolor{orange}{rgb}{0.7,0.3,0}
\definecolor{blue}{rgb}{.2,.6,.75}
\definecolor{green}{rgb}{.4,.7,.4}
\definecolor{purple}{RGB}{127,0,255}

\newcommand{\sos}{sums of two squares }

\newcommand{\kommentar}[1]{}

\calclayout

\title[Bias for consecutive sums of two squares]{Lemke Oliver and Soundararajan bias \\  for consecutive sums of two squares}

\author{Chantal David, Lucile Devin, Jungbae Nam and Jeremy Schlitt}

\address{Department of Mathematics \& Statistics, Concordia University, 1455 de Maisonneuve Blvd. West, Montr\'{e}al, Qu\'{e}bec, H3G 1M8, CANADA}
\email{chantal.david@concordia.ca} 

\address{Department of Mathematical Sciences, Chalmers University of Technology and the University of Gothenburg, SE-412 96 Gothenburg, Sweden}
\email{devin@chalmers.se}

\address{Department of Mathematics \& Statistics, Concordia University, 1455 de Maisonneuve Blvd. West, Montr\'{e}al, Qu\'{e}bec, H3G 1M8, CANADA}
\email{jungbae.nam@concordia.ca} 

\address{Department of Mathematics \& Statistics, Concordia University, 1455 de Maisonneuve Blvd. West, Montr\'{e}al, Qu\'{e}bec, H3G 1M8, CANADA}
\email{jeremy.schlitt@mail.concordia.ca}


\begin{document}

	\begin{abstract}
		 In a surprising recent work, Lemke Oliver and Soundararajan noticed how experimental data exhibits erratic distributions for consecutive pairs of primes in arithmetic progressions, and proposed a heuristic model based on the Hardy--Littlewood conjectures containing a large secondary term, which fits the data very well. In this paper, we study  consecutive pairs of sums of squares in arithmetic progressions, and develop a similar heuristic model based on the Hardy--Littlewood conjecture for sums of squares, which also explain the biases in the experimental data.
		 In the process, we prove several results related to averages of the Hardy--Littlewood constant in the context of sums of two squares.
	\end{abstract}

\vspace*{.01 in}
\maketitle
\vspace{.08 in}

\section{Introduction}\label{section intro}

We study in this paper the distribution of consecutive sums of two squares in arithmetic progressions. Our work is inspired by a recent paper of Lemke Oliver and Soundararajan \cite{LOS} who proposed a heuristic model based on the Hardy--Littlewood conjecture for the distribution of consecutive primes in arithmetic progressions.

Roughly speaking, it is expected that numbers described by reasonable multiplicative constraints should be well-distributed, in short intervals and in arithmetic progressions. The case of prime numbers is of course well-studied, and this philosophy was also tested for numbers expressible as sums of two squares, as well as square-free numbers\footnote{In the case of square-free numbers, the Hardy--Littlewood conjecture is a theorem \cite{Mir}, and the analogue of \cite{LOS} has been proved recently by Mennema \cite{Men}.}. 
Gallagher \cite{Ga1976} proved that the distribution of primes of size up to $x$ in intervals of size $\log{x}$ has a Poisson spacing distribution, assuming some explicit form of the Hardy--Littlewood conjecture. This was generalized to sums of two squares by Freiberg, Kurlberg and Rosenzweig \cite{FKR}, for intervals of size $\sqrt{\log{x}}/K$, which is the correct analogue to Gallagher's result in view of~\eqref{landau-intro}. For primes in larger intervals, Montgomery and Soundararajan \cite{MS} showed that the spacings exhibit a normal distribution around the mean, assuming again some explicit form of the Hardy--Littlewood conjecture. We prove in this paper a weaker version of their results (Theorem~\ref{Prop bound for D with large sets}) for the case of sums of two squares which is needed to study the distribution of successive sums of two squares in arithmetic progressions.
We speculate that the full analogue of their results can be obtained for sums of two squares, but we did not pursue it as Theorem~\ref{Prop bound for D with large sets}
is sufficient for our application.
Some unexpected irregularities in the distribution of primes in short intervals were discovered by Maier \cite{Ma85}, and it was shown by Balog and Wooley \cite{BW} that sums of two squares exhibit the same irregularities. Sums of two squares in short intervals were also studied over function fields of a finite field $\F_q$, where many results which are inaccessible over number fields can be proven when the size of the finite field $\F_q$ grows \cite{BYW16, BBF18, BF19, GR20}.

 We first fix some notations. We denote by $$\mathbf{E}= \lbrace a^2 + b^2 : a,b \in \Z\rbrace = \lbrace E_n : n\in \N\rbrace$$ the set of sums of two squares (enumerated in increasing order), such that $E_n$ is the $n$th number that can be written as a sum of two squares. Let $\mathbf{1}_{\mathbf{E}}$ the indicator function of this set.  By a classical result of Landau,
\begin{align} \label{landau-intro}
\sum_{n \leq x}  \mathbf{1}_{\mathbf{E}}(n)  \sim K \frac{x}{\sqrt{\log{x}}},\end{align}
where $K$ is the constant defined by~\eqref{LR-constant}. 
The distribution of sum of two squares in arithmetic progressions exhibit different behavior depending on the modulus $q$ of the progression, and we restrict in this paper to the case where $q$  is a prime number such that $q \equiv 1 \mod 4$. In that case, the sums of squares are equidistributed in all the residue classes $a \mod q$, including the class $a \equiv 0 \mod q$ (see Theorem~\ref{thm-SOS-AP-1}), but unlike the case of the primes, there is a large secondary term depending if  the residue class $a \equiv 0 \mod q$ or not 
(see Theorem~\ref{thm-SOS-AP}).

We consider in this paper the following question, which was studied by Lemke Oliver and Soundararajan for primes \cite{LOS}.
Fix a prime number $q \equiv 1 \mod 4$, and integers $a,b$. 
What is the distribution of 
$$N(x;q,(a,b))  := \# \{ E_n \leq x \;:\; E_n \equiv a \mod q, \; E_{n+1} \equiv b \mod q \} \text{ ?}$$ 
Using a model based on randomness, we expect successive sums of two squares to be well-distributed in arithmetic progressions, and each of the $q^2$ pairs of classes $(a,b)$ to contain the same proportion (asymptotically) of sums of two squares, with possibly a bias towards the pairs $(a,b)$ where $ab \equiv 0 \mod q$ in view of Theorem~\ref{thm-SOS-AP}. However, the numerical data of Table~\ref{table:actual-pairs-sts-5} (for $q=5$ and $x=10^{12}$) shows a lot of fluctuation, and in particular an unexpected large bias against the classes $(a,a)$ including $(0,0)$.
Interestingly, this bias goes in the opposite direction of the bias for sums of squares in arithmetic progressions: there are ``more'' \sos congruent to $0 \mod q$, but there are ``less'' consecutive \sos congruent to $(0,0) \mod q$.

\begin{table}[htp!]
	\begin{tabular}{cc|cccc|ccccc}
		$a$ & $b$ & $N (10^{12};5,(a,b))$ &   & $a$ & $b$ & $N (10^{12};5,(a,b))$ &   & $a$ & \multicolumn{1}{l|}{$b$} & $N (10^{12};5,(a,b))$ \\ \cline{1-3} \cline{5-7} \cline{9-11} 
		0 & 0 &   4 108 407 474   &  & 2 & 0 &  8 049 996 586     &  &  4 & \multicolumn{1}{l|}{0} & 7 155 732 959\\
		& 1 &   7 153 121 164  &  &   & 1 &  5 516 037 772  &  &    & \multicolumn{1}{l|}{1} & 5 356 545 210 \\
		& 2 &   5 604 312 560 &  &   & 2 &   3 754 593 831 &  &  & \multicolumn{1}{l|}{2} & 7 730 855 281 \\
		& 3 &   8 054 714 831 &  &   & 3 &  6 837 553 372  &  &    & \multicolumn{1}{l|}{3} & 5 497 266 920 \\
		& 4 &   5 780 373 060 &   &   & 4 &    5 350 735 550  &  &    & \multicolumn{1}{l|}{4} & 3 768 530 444\\
		& & & & & & & & & & \\
		1 & 0 &   5 777 315 850    &  & 3 & 0 &     5 609 476 219 &   &   &     &                      \\
		& 1 &   3 765 205 659  &  &   & 1 &      7 718 021 263   &   &   &      &                    \\
		& 2 &   6 870 009 299   &  &   & 2 &    5 549 146 140    &  &   &    &                       \\
		& 3 &   5 354 226 097  &  &   & 3 &      3 765 159 558  &  &   &   &                         \\
		& 4 &   7 742 174 162   &  &   & 4 &      6 867 117 598  &  &   &     &                        
	\end{tabular}
	\caption{$N(x; q, (a,b))$ for $q = 5$ and $x = 10^{12}$. The average of $N(x; q, (a,b))$ is 5 949 465 154.}
	\label{table:actual-pairs-sts-5}
\end{table}

Estimates for the consecutive sums of squares (or consecutive primes) in arithmetic progressions is a very difficult question, and  few results are known. For consecutive primes in arithmetic progressions, it was conjectured by Chowla that there are infinitively many primes $p_n$ such that $p_{n+i-1} \equiv a \mod q$ for $1 \leq i \leq r$, for any $(a,q)=1$ and $r \geq 2$. This was proven by Shiu \cite{Sh00}. Recent progress in sieve theory have led to a new proof of Shiu's result \cite{BFT15}, and Maynard has proven that the number of such primes is $\gg \pi(x)$ \cite{Ma16}. It would be interesting to see if those recent progresses could be applied to get lower bounds for the number of successive sums of two squares $E_n$ such that $E_{n+i-1} \equiv a \mod q$ for $1 \leq i \leq r$, for any $a$ and $r \geq 2$, but this question was not addressed yet in the literature.

We propose in this paper a heuristic model predicting an asymptotic for $N(x; q,(a,b))$, based on the heuristic of Lemke Oliver and Soundararajan \cite{LOS} for the case of primes, and exhibiting a similar bias.

\begin{conjecture} \label{main-conjecture-intro} Fix a prime $q \equiv 1 \mod 4$, and $J \geq 1$. Then, for any $a\in \N$, we have
\begin{align*}
		N(x;q,(a,a)) &= \frac{K}{q^2} \frac{ x}{ \sqrt{\log x}} \bigg( 1 - \frac{\sqrt{2} \phi(q)}{\pi}\frac{\sqrt{\log \log{x}}}{\sqrt{\log{x}}} + \frac{1}{\sqrt{\log{x}}} \sum_{j = 1}^J C_j (\log\log x)^{\frac{1}{2}-j} \bigg)\\
		&+ O \left( \frac{x}{\log{x}(\log{\log x})^{J+ \frac{1}{2}}} \right), 
\end{align*}
for some explicit constants $C_j$ depending only on $q$.
For $a,b\in \N$ with $a \not\equiv b \mod q$, we have
\begin{align*}
	N(x;q,(a,b)) &= \frac{K}{q^2} \frac{ x}{ \sqrt{\log x}} \bigg( 1 + \frac{\sqrt{2}}{\pi} \frac{\sqrt{\log \log{x}}}{\sqrt{\log{x}}} + \frac{C_{a,b}}{\sqrt{\log{x}}} - 
	\frac{1}{\phi(q)\sqrt{\log{x}}} \sum_{j = 1}^J C_j (\log\log x)^{\frac{1}{2}-j} \bigg) \\ &+ O \left( \frac{x}{\log{x}(\log{\log x})^{J+ \frac{1}{2}}} \right), 
\end{align*}
with
\begin{equation}\label{D_a,b-intro}
C_{a,b} := \frac{1}{2K}\frac{q}{\phi(q)}\sum_{\chi\neq \chi_0}\overline{\chi}(b-a)C_{q,\chi},
\end{equation}
where the sum is over the non-principal Dirichlet characters modulo $q$ and $C_{q, \chi}$ is defined in~\eqref{def-C}. The value of $C_1$  is given in 
Conjecture~\ref{main-conjecture-J1}. 
\end{conjecture}

Our heuristic model leading  to Conjecture~\ref{main-conjecture-intro} follows very closely \cite{LOS}, and as such it is based 
on the Hardy--Littlewood conjectures for sums of squares, which are stated in Section~\ref{section-HL}.  Our exposition for that section, and many of the results used for the properties of the (conjectural) Hardy--Littlewood constants for sums of squares follow from \cite{FKR}. 
Fix $k \geq 1$ and $\lbrace d_1, \dots, d_k\rbrace \subseteq\Z$. We denote $ \mathfrak{S}(\lbrace d_1, \dots, d_k \rbrace)$  the Hardy-Littlewood constants for $k$-tuples of \sos defined in Section~\ref{section-HL}.
As the results of \cite{LOS}, our conjecture follows from an average of the Hardy--Littlewood constants, which is one of the main results of our paper. 
\begin{theorem}\label{Intro prop2.1}  Let $q \equiv 1 \mod 4$ be a prime. 
	For each Dirichlet character $\chi \neq \chi_0 \mod q$, let $C_{q, \chi}$ be defined by~\eqref{def-C}.
	Then, for any $J \geq 1$, and $v \neq 0 \mod q$, we have
	\begin{align*}
		\sum_{h \geq 1}  \mathfrak{S} (\lbrace 0,h \rbrace) e^{-h/H} &= 
		 H - \frac{2}{K \pi} \sqrt{\log\: H} + \sum_{j=1}^J c(j) \, (\log{H})^{1/2-j} + O \left( (\log{H})^{-1/2-J}  \right)\\
		\sum_{\substack {h \geq 1\\h \equiv 0 \mod q}}  \mathfrak{S} (\lbrace 0,h \rbrace) e^{-h/H} &= 
		\frac{H}{q} - \frac{2}{K\pi} \sqrt{\log \: H} + \sum_{j=1}^J c_0(j)\, (\log{H})^{1/2-j} + O \left( (\log{H})^{-1/2-J}  \right)\\
		\sum_{\substack {h \geq 1\\h \equiv v \mod q}}  \mathfrak{S} (\lbrace 0,h \rbrace) e^{-h/H} &= 
		\frac{H}{q} 
		+ {\frac{1}{2 K^2 \phi(q)}} \sum_{\substack{\chi \mod q \\\chi\neq\chi_0}} \bar{\chi}(v) C_{q,\chi} + \sum_{j=1}^J c_1(j)\, (\log{H})^{1/2-j} + O \left( (\log{H})^{-1/2-J}  \right).
	\end{align*}
	The constants $c(j)$ are absolute while the constants $c_0(j), c_1(j)$ depend only on $q$, they can all be explicitly computed, in particular the values for $j=1$ are given in~\eqref{Value c} and~\eqref{Values c_1 c_0}.
	Moreover they satisfy the relation 
		\begin{align} \label{relation-between-c} c_0(j) + \phi(q)c_1(j) = c(j), \hspace{0.2cm} j\geq 1. \end{align}
\end{theorem}

By using Theorem~\ref{Prop bound for D with large sets}, which is the analogue of the work \cite{MS} for sums of two squares, we need only to compute a weighted average of the constants $\mathfrak{S}( \left\{ 0, h \right\})$ associated to 2-tuples, 
while \cite{FKR} compute a more general average of the constants $\mathfrak{S}( \left\{ h_1, \dots, h_k \right\} )$
associated to $k$-tuples. 
Since the Hardy--Littlewood constants  $\mathfrak{S}( \left\{ 0, h \right\} )$ can be described explicitly with  a simple formula from the  work of Connors and Keating \cite{CK}, this allows us to get a very precise result exhibiting a small secondary term which gives the bias. A similar average of the constants  $\mathfrak{S}( \left\{ 0, h \right\})$ was computed by 
Smilansky \cite{smilansky}, and we also use some of his results. Moreover, the techniques developed in this paper yield a more precise form of the averages considered in  \cite{smilansky} and \cite{FKR}.

\begin{proposition}\label{intro-FKR-improved}
	Assume the Generalized Riemann Hypothesis.
	For $\varepsilon >0$ and  $k\geq 2$,
	we have
	\begin{align*}
		\sum_{\substack{1\leq d_1, \dots, d_k \leq H \\ \text{distinct}}} \mathfrak{S}(\lbrace d_1, \dots, d_k \rbrace)
		= H^k +  \frac{{k(k-1) H^{k-1}}}{\pi K^2}   \int_{1/2+\varepsilon}^1 \frac{F'(\sigma) H^{\sigma -1} + F(\sigma) H^{\sigma-1} \log{H}}{|\sigma-1|^{1/2}} d\sigma  + O_{k,\varepsilon}(H^{k - \frac32 + \varepsilon}),
	\end{align*}
	where $F(s)=\zeta(s-1) M(s-1) \left[ (s-1) \zeta(s) \right]^{1/2} s^{-1}$, with $M(s)$ as defined by~\eqref{def-M}.
\end{proposition}

Finally, the heuristic leading to Conjecture~\ref{main-conjecture-intro} can be generalized to predict an asymptotic for $r$ successive \sos in arithmetic progressions. 

\begin{conjecture} \label{general-conjecture-intro} Fix a prime $q \equiv 1 \mod 4$, $r \geq 2$ and  ${\bf a} = (a_1, \dots, a_r) \in \N^r$. Let
$$
N(x;q,{\bf a})  := \# \{ E_n \leq x \;:\; E_{n+i-1} \equiv a_i \mod q \}.
$$ We have
	\begin{align*}
		N(x;q,\mathbf{a})& = \frac{x}{q^{r}}\frac{K}{\sqrt{\log x}}\Bigg( 1 + C_{-1}(\mathbf{a})\frac{(\log\log x)^{\frac12}}{(\log x)^{\frac12}} + \frac{C_{0}(\mathbf{a})}{(\log x)^{\frac12}} + \frac{C_{1}(\mathbf{a})}{(\log\log x)^{\frac12}(\log x)^{\frac12}} \Bigg) \\ & \quad \quad+ O\big(x(\log\log x)^{-\frac32}(\log x)^{-1} \big),
	\end{align*}
where
\begin{align*}
	C_{-1}(\mathbf{a}) &=  \frac{q\sqrt{2} }{\pi}\sum_{i=1}^{r-1} \big(\tfrac{1}{q} -\delta(a_{i+1}\equiv a_i)  \big)\\
	C_{0}(\mathbf{a}) &=  \sum_{\substack{ 1\leq i \leq r-1 \\ a_i \not\equiv a_{i+1} \pmod{q} }}C_{a_i,a_{i+1}} \\
	C_{1}(\mathbf{a}) &=  
-\frac{qC_1}{\phi(q)}\sum_{i=1}^{r-1} \big(\tfrac{1}{q} -\delta(a_{i+1}\equiv a_i)  \big)
	+ \frac{q\sqrt{2}}{\sqrt{\pi}}  \sum_{k=1}^{r-2}  \sum_{i=1}^{r-1-k} \frac{\frac{1}{q} -\delta(a_{i+k+1}\equiv a_i)}{k}  ,
\end{align*}
and the constants $C_{a_i,a_{i+1}}$ are defined by~\eqref{D_a,b-intro}, and the value of $C_1$ is given in 
Conjecture~\ref{main-conjecture-J1}. \end{conjecture}

The structure of the paper is as follows: we review in Section~\ref{section-SOS-AP} the basic properties of sums of two squares, including the secondary terms for the counting function of \sos in arithmetic progressions, which surprisingly we did not find in the literature. We discuss the Hardy--Littlewood conjectures for \sos in Section~\ref{section-HL}.
 We present the heuristic model leading to Conjecture~\ref{main-conjecture-intro} in Section~\ref{section-heuristic-1},
 following Lemke Oliver and Soundararajan \cite{LOS}; in particular, we explain how the heuristic reduces Conjecture~\ref{main-conjecture-intro} to an average of Hardy--Littlewood constants (Theorem~\ref{Intro prop2.1}), which we prove in  Section~\ref{proof-prop-section} using the Selberg--Delange method.
 
 We prove Theorem~\ref{Prop bound for D with large sets}, which is an analogue of the main result of Montgomery and Soundararajan \cite{MS} mentioned above used to justify our heuristic, in Section~\ref{Section proof prop bound for D with large sets}. We then use this result in Section~\ref{section-improved-averages} to prove Proposition~\ref{intro-FKR-improved}
which improve the average results of \cite{FKR} and \cite{smilansky}. 
Finally, we explain how to deduce  Conjecture~\ref{general-conjecture-intro} from our heuristic in Section~\ref{section-general-conj},  and 
 we present numerical data in Section~\ref{section-numerical}.

\section*{Acknowledgements}
We thank the organizers of the MOBIUS ANT who invited us to present a preliminary version of the results contained in this paper in a friendly environment with subsequent discussions that led to several improvements in our exposition.
We are particularly grateful to Dimitris Koukoulopoulos for constructive discussions, and for pointing out some relevant material from
his book \cite{koukoulopoulos}, and the post of Lucia on Mathoverflow \cite{lucia}.

Our results and conjectures are supported by several numerical data, that were computed using SageMath~\cite{Sage}, PARI/GP~\cite{PARI2} and Mathematica \cite{Mathematica}.
The authors thank the Centre de Recherches Math\'{e}matiques (CRM) in the Universit\'{e} de Montr\'{e}al for offering their clusters for some of the numerical computations.

The first author was supported by  a NSERC Discovery Grant and a FQRNT Team Grant; the second author was supported by the grant KAW 2019.0517 from the Knut and Alice Wallenberg Foundation; the fourth author was supported by a Concordia CUSRA.

\section{Sums of two squares in arithmetic progressions}
\label{section-SOS-AP}

	By a classical result of Landau \cite{Lau}, we have
	\begin{equation} \label{Landau}
	\sum_{\substack{n \leq x}} \mathbf{1}_{\mathbf{E}}(n) \sim {K}\frac{x}{\sqrt{\log x}},
	\end{equation}
	where \begin{align} \label{LR-constant}
	K = \frac{1}{\sqrt{2}}  \prod_{p\equiv 3 \pmod{4}} (1-p^{-2}) ^{-\frac12} \end{align} is the Landau--Ramanujan constant.
	We remark that, unlike the prime number theorem, the asymptotic above gives only the main term, and there is no simple integral similar to ${\text{li}}(x)$ which approximates well the number of \sos up to $x$. This is caused by the fact that the generating series for \sos has an essential singularity at $s=1$, its contribution is evaluated by the Selberg--Delange method which gives~\eqref{Landau}. It is possible to iterate the Selberg--Delange method to write, for any $J \geq 1$,
	\begin{equation} \label{refined-Landau}
	\sum_{\substack{n \leq x}} \mathbf{1}_{\mathbf{E}}(n) = {K}x \left( \sum_{j=0}^J \frac{c_j}{(\log x)^{1/2+j}}  \right) +O \left( x(\log{x})^{-3/2-J} \right).
	\end{equation}
	Explicit values for the constants $c_j$ can be found in the literature for $c_0 = 1$ \cite{Lau}, $c_1 = 0.581948659\cdots$ \cite{Sta28, Sta29, Sha} and up to $c_{15}$ in \cite{EG18}. 
	It is possible to get an expression for the number of \sos smaller than $x$ with a better error term,
	but one looses the simplicity of the formula above as a sum of descending powers of $\log$. We state this result in the next theorem, that we will prove in Section~\ref{section-improved-averages}. A similar expression for the number of \sos exhibiting squareroot cancellation under the GRH can be found in \cite[Theorem B.1]{GR20}, inspired by the work of \cite{Ram76}.

	\begin{theorem}\label{thm-S0S-integral} Let $0 < \varepsilon < 1/2$. There exists a constant $c > 0$ such that
		\begin{align*}
			\sum_{n \leq x} \mathbf{1}_{\mathbf{E}}(n) = \frac{1}{\pi} \int_{1/2+\varepsilon}^1  G(\sigma) \frac{x^\sigma}{\sigma |\sigma -1|^{1/2}} \; d \sigma 
			+ O \left( x \exp{ \left( - c \sqrt{\log{x}} \right)} \right),
		\end{align*}
		where $G(s) = (\zeta(s)(s-1))^{1/2} L(s, \chi_4)^{1/2} (1-2^{-s})^{-1/2} \prod_{p \equiv 3 \mod 4} \left( 1 - p^{-2s} \right)^{{-1/2}}$ and $\chi_4$ is the non-trivial Dirichlet character modulo 4, so
		$G(s)$ is an analytic function for $\re(s)>1/2 + \varepsilon.$
		If we assume the Riemann Hypothesis for $\zeta(s)$ and $L(s, \chi_4)$, we can replace the error term by $O \left( x^{1/2+\varepsilon} \right).$
	\end{theorem}	
	
	Even if it is more precise (see Table~\ref{table:S0S-integral}), this formula gives somehow less insight on the behaviour of the secondary terms and we come back to the Selberg--Delange method 
	when separating the \sos into congruence classes.
	
	\begin{table}[htb!]
		\captionsetup{width=.58\linewidth}
		\begin{tabular}{c|rrrr|ccc}
			\hline 
			\makecell{$x$} & \makecell{Actual} & \makecell{\eqref{Landau}} & \makecell{\eqref{refined-Landau}} & \makecell{Theorem~\ref{thm-S0S-integral}} & 
			\makecell{\eqref{Landau}} & \makecell{\eqref{refined-Landau}} & \makecell{Theorem~\ref{thm-S0S-integral}} \\
			\hline
			\hline
			$10^9$  & 173 229 059 & 167 877 068 & 172 591 375 & 173 226 354 & 1.0319 & 1.0037 &  1.00001562\\
			$10^{10}$  & 1 637 624 157 & 1 592 621 708 & 1 632 873 166 & 1 637 616 416 &  1.0283 & 1.0029 & 1.00000473\\
			$10^{11}$ & 15 570 512 745 & 15 185 052 177 & 15 533 945 443 &  15 570 488 969 & 1.0254 & 1.0024 & 1.00000153\\
			$10^{12}$  & 148 736 628 859 & 145 385 805 874 & 148 447 838 016 &  148 736 563 568 & 1.0230 & 1.0019 & 1.00000044\\
			\hline
		\end{tabular}
		\caption{Comparison of the experimental data for the number of \sos up to $x$ with the asymptotic of~\eqref{Landau},  the asymptotic of~\eqref{refined-Landau} with the first two terms and  the integral of Theorem~\ref{thm-S0S-integral}. The three rightmost columns are the percentage errors. Notice that the error for the integral approximation of Theorem~\ref{thm-S0S-integral} agrees with the error term under the Riemann Hypothesis.}
		\label{table:S0S-integral}
	\end{table}
	
	Let us now consider the distribution of  sums of two squares in arithmetic progression modulo $q$. For $a\in \N$, following the notations introduced in Section~\ref{section intro}, let us denote $$
	N(x;q,a)  := \# \{ E_n \leq x \;:\; E_{n} \equiv a \mod q \}.$$ The case $q\equiv 1 \pmod{4}$ is a prime is particularly simple, and we restrict to that case.
We refer the reader to \cite[Satz~1]{Rie} (see also \cite[Lemma~2.1]{BW}) for the general case.
	\begin{theorem}\cite[Satz~1]{Rie} \label{thm-SOS-AP-1}
	Let $q\equiv 1 \pmod{4}$ be a prime. Then, for $a\in \Z/q\Z$,
	\begin{equation*}
	N(x; q,a) := \sum_{\substack{n \leq x \\ n \equiv a \pmod{q}}} \mathbf{1}_{\mathbf{E}}(n) \sim \frac{K}{q}\frac{x}{\sqrt{\log x}}.
	\end{equation*}
	\end{theorem}

If ones compares the above theorem with experimental data for $N(x; q, a)$ as shown in Table~\ref{table:S0S_mod_5}, there is a discrepancy, and the experimental data shows an excess for 
	$a \equiv 0 \pmod{q}$ compared to the other classes modulo $q$. This is caused by secondary terms that depends on the class $a$, which do not seem to appear in the literature, and we compute the first such term in Theorem~\ref{thm-SOS-AP} below. The proof uses the Selberg--Delange method  which evaluates the contribution of essential singularities by using Hankel's formula, replacing Cauchy's residue theorem for this case. We state below the version of the method needed for the proof of Theorem~\ref{thm-SOS-AP}, and we refer the reader to \cite[Chapter~II.5]{tenenbaum} and \cite[Chapter~13]{koukoulopoulos}, and to Section~\ref{proof-prop-section} for more details. 
	
	\begin{theorem} \cite[Theorem~13.2]{koukoulopoulos} \label{LSD} Let $f(n)$ be a multiplicative function with generating function $F(s) = \sum_{n \geq 1} f(n) n^{-s}$. Suppose there exists $\kappa \in \C$ be such that for $x$ large enough
	$$\sum_{p \leq x} f(p) \log{p} = \kappa x + O_A \left( x/(\log{x})^A \right),$$
	{for each fixed $A>0$,} and such that $|f(n)| \leq \tau_k(n)$ for some $k \in \N$, where $\tau_k$ is the $k$-th divisor function.
	For $j \geq 0$, let $\widetilde{c_j}$  be the Taylor coefficients about 1 of the function $(s-1)^\kappa F(s)/s$. Then, for any $J  \in \N$, and $x$ large enough, we have
		$$
	\sum_{n \leq x} f(n) = x \sum_{j=0}^{J} \widetilde{c_j} \frac{(\log{x})^{\kappa-j-1}}{\Gamma(\kappa-j)} + O \left( \frac{x}{(\log{x})^{J+2-\re(\kappa)}}\right).
	$$
	\end{theorem}

	\begin{theorem} \label{thm-SOS-AP}
	Let $q\equiv 1 \pmod{4}$ be a prime, and let $K$ and $c_1$ be as defined above. Then, 
		\begin{equation*} 
	\sum_{\substack{n \leq x \\ n \equiv a \pmod{q}}} \mathbf{1}_{\mathbf{E}}(n) = \frac{K}{q}x\sum_{j = 0}^J\frac{c_{j,a}}{(\log x)^{1/2+j}} + O\left( \frac{x}{(\log x)^{J+3/2}}\right),
	\end{equation*}
where
	\begin{align} \label{constant-AP}
	c_{0,a} = c_0 = 1 \quad \text{ and } \quad
	c_{1,a} :=
	\begin{dcases}
	c_1 + \frac{\log{q}}{2} & \text{ if } a \equiv 0 \pmod{q}\\
	c_1 - \frac{\log{q}}{2(q-1)} & \text{ otherwise.}
	\end{dcases}.
	\end{align}
	\end{theorem}
We refer the reader to Table~\ref{table:S0S_mod_5} for the comparison between the numerical data and Theorem~\ref{thm-SOS-AP}.

\begin{table}[htb!]
	\captionsetup{width=.58\linewidth}
	\begin{tabular}{cc|c|cc}
		\hline 
		\makecell{$q$} & \makecell{$a$} & \makecell{$N(x; q, a)$} & \makecell{Main term}
		& \makecell{Main and secondary terms}\\
		\hline
		\hline
		& $0$ & 30 700 929 089 & 29 077 161 174 & 30 536 403 581 \\
		& $1$ & 29 508 931 067 &   & 29 477 858 608     \\
		$5$ & $2$ & 29 508 917 111 & &  \\
		& $3$ & 29 508 920 778 &   &     \\
		& $4$ & 29 508 930 814 &   &    \\
		\hline
	\end{tabular}
	\caption{Comparison of the experimental data for $N(x; q, a)$ and the asymptotic of Theorem~\ref{thm-SOS-AP} using only the main term, or the main term and the first secondary term for $q = 5$ and $x = 10^{12}$.
		The average of $N(x; q, a)$ is $\approx$ 29 747 325 771.}
	\label{table:S0S_mod_5}
\end{table}

		\begin{proof}
	 Let $F(s) := \sum_{n \ge 1}\mathbf{1}_{\mathbf{E}}(n)n^{-s}$  be the generating series for sums of two squares. Using the well-known fact that $n$ is a sum of two squares if and only if $v_p(n)$ is even for all primes $p \equiv 3 \mod 4$, it is easy to see that for $\re(s) > 1$, 
	 \begin{align*}
F^2(s) &= \prod_{p \not\equiv 3 \mod 4} \left( 1 - \frac{1}{p^s} \right)^{-2} \prod_{p \equiv 3 \mod 4}  \left( 1 - \frac{1}{p^{2s}} \right)^{-2} \\
&= \zeta(s)L(s, \chi_4)\left(1 - \frac{1}{2^{s}}\right)^{-1}\prod_{p \equiv 3 (4)}\left(1 - \frac{1}{p^{2s}}\right)^{-1} \end{align*}
	 where $\chi_4$ is the non-principal Dirichlet character modulo 4.
	 Landau ~\cite{Lau} also showed 
	that in a neighborhood of $s = 1$, 
\[
\frac{F(s)}{s^2} = \sum_{\ell \ge 0}i a_\ell (1-s)^{\ell-1/2},
\]
with
$a_0 = K\sqrt{\pi}$ and 
$a_1= a_0(2c_1 + 1)$ \cite{Sha}. Applying Theorem~\ref{LSD} with $\kappa=1/2$, we get~\eqref{refined-Landau}, using the values $a_0, a_1$ to get explicit values for the first two Taylor coefficients of $(s-1)^{1/2} F(s)/s$.

To introduce the congruence condition, we write for $a \not\equiv 0 \mod q$,
\begin{align} \label{ortho-rel}
N(x; q, a) 
= \frac{1}{q-1}\sum_{\chi \pmod{q}}\overline{\chi}(a)\sum_{n \le x}\chi(n)\mathbf{1}_{\mathbf{E}}(n), 
\end{align}
and we denote the generating function of $f_\chi(n) = \chi(n)\mathbf{1}_{\mathbf{E}}(n)$ by $F_\chi(s) := \sum_{n \ge 1}\chi(n)\mathbf{1}_{\mathbf{E}}(n)n^{-s}.$ 
For $\chi_0$ the principal character modulo $q$ and $\chi \neq \chi_0$, we have for $\re(s) > 1$, \begin{equation}  \begin{split} \label{Fchi-Fchi0}
F_\chi^2(s) &= L(s, \chi)L(s, \chi_4\chi)\left(1 - \frac{\chi(2)}{2^{s}}\right)^{-1}\prod_{p \equiv 3 (4)}\left(1 - \frac{\chi^2(p)}{p^{2s}}\right)^{-1} \\
F_{\chi_0}^2(s) &= \left(1 - \frac{1}{q^{s}}\right)^{2}F^2(s).
\end{split}\end{equation}
For $\chi \neq \chi_0$, since $F_{\chi}(s)$ is analytic for $\operatorname{Re}(s) > 1/2$, we have for any $\varepsilon > 0$ that
$$\sum_{n \le x}\chi(n)\mathbf{1}_{\mathbf{E}}(n) = O \left( x^{1/2+\varepsilon} \right),$$
and the theorem will follow by evaluating $\sum_{n \le x}\chi_0(n)\mathbf{1}_{\mathbf{E}}(n)$ with the Selberg--Delange method. 
Let $\widetilde{b_j}$ be the Taylor coefficients of $(s-1)^{1/2} F_{\chi_0}(s)/s$ around $s=1$, and $\widetilde{c_j}$ are the Taylor coefficients of $(s-1)^{1/2} F(s)/s$ around $s=1$.
From~\eqref{Fchi-Fchi0}, it is easy to compute 
\begin{align*}
\widetilde{b_0} &= (1-q^{-1}) \, \widetilde{c_0} = (1-q^{-1}) \,  K \sqrt{\pi} \\
\widetilde{b_1} &=  (1-q^{-1}) \, \widetilde{c_1} + \frac{\log q}{q} \, \widetilde{c_0} = K \sqrt{\pi} \left( \frac{\log{q}}{q} - 2 c_1 (1-q^{-1}) \right).
\end{align*}
Applying Theorem~\ref{LSD} with $\kappa=1/2$, to estimate the sum $\sum_{n \le x}\chi_0(n)\mathbf{1}_{\mathbf{E}}(n),$
and replacing in~\eqref{ortho-rel}, we get 
the statement of the theorem when $a \not\equiv 0 \mod q$, with
$$\frac{K}{q}c_{0,a} = \frac{\widetilde{b_0}}{(q-1) \Gamma(1/2)}, \;\;\;\; \frac{K}{q}c_{1,a} = \frac{\widetilde{b_1}}{(q-1) \Gamma(-1/2)}.$$
For $a \equiv 0 \pmod{q}$, we use the above and ~\eqref{refined-Landau} to obtain
\begin{align*}
N(x; q, 0) &= \sum_{\substack{n \leq x}} \mathbf{1}_{\mathbf{E}}(n) - \sum_{a \not\equiv 0 \mod q} N(x; q, a) \\ &= 
\frac{K}{q}x\left(\frac{1}{(\log x)^{1/2}} + \left(c_1 + \frac{\log{q}}{2}\right)\frac{1}{(\log x)^{3/2}} +
\sum_{j = 2}^J\frac{c_j - (q-1)c_{j,1}}{(\log x)^{1/2+j}} \right) + O\left(x(\log x)^{-3/2-J}\right), 
\end{align*}
which completes the proof. 
\end{proof}

	 \section{Hardy--Littlewood conjectures in arithmetic progressions for sum of two squares}
\label{section-HL}

We state in this section the analogue of the Hardy--Littlewood prime $k$-tuple conjectures for the case of sums of two squares, following \cite{FKR}. We also state new bounds on the average of the Hardy--Littlewood constant in this context that are useful in our heuristic for Conjecture~\ref{main-conjecture-intro}, but are also interesting in themselves as they are related to the distribution of gaps between sums of two squares.

For $k \geq 1$, let $\mathcal{H} = \left\{ h_1, \dots, h_k \right\} \subseteq \Z$, and
	\begin{align*} 
	R_k(\mathcal{H}; x) :=  \frac{1}{x} \sum_{\substack{n \leq x}}  \mathbf{1}_{\mathbf{E}}(n+h_1 ) \dots \mathbf{1}_{\mathbf{E}}(n+h_k).
	\end{align*}
In the case $\mathcal{H} = \{ 0 \}$, we have
	$$
	R_1(x) :=  R_1(\{ 0 \}; x) = \frac{1}{x} \sum_{\substack{n \leq x}} \mathbf{1}_{\mathbf{E}}(n) \sim \frac{K}{\sqrt{\log x}}.
	$$
	The philosophy of the Hardy--Littlewood conjecture is that the events $\mathbf{1}_{\mathbf{E}}(n+h_i)$ are ``independent'', and the probability that $n+h_i$ are simultaneously \sos for $1 \leq i \leq k$ is the product of the probabilities, which is (ignoring the small differences between $\log{n}$ or $\log{n+h_i}$)
	$$
	\left( \frac{K}{\sqrt{\log{n}}} \right)^k.
	$$
	Of course, the events are not really independent, so we adjust by considering the probabilities that $n+h_i$ are \sos modulo $p$ versus the probably that $k$ independent integers are \sos modulo $p$. To do so,  for each prime $p$, we define
	\begin{align*} 
		\delta_{\mathcal{H}}(p) = \lim_{\alpha\rightarrow\infty} \frac{\# \lbrace 0 \leq a < p^{\alpha} : \forall h \in \mathcal{H}, a + h \equiv \square + \square \pmod{p^{\alpha}} \rbrace}{p^{\alpha}}.
	\end{align*}
	Since  $\delta_{\mathcal{H}}(p) = 1$ for $p \equiv 1 \mod 4$ (see e.g.\ \cite[Proposition~5.1]{FKR}),
	we define the singular series for $\mathcal{H} = \left\{ h_1, \dots, h_k \right\}$ by
	\begin{align}\label{def singular series product}
			\mathfrak{S}( \mathcal{H} ) := \prod_{p \not\equiv 1 \mod 4} \frac{\delta_{\mathcal{H}}(p)}{(\delta_{ \{ 0 \}}(p))^k}.
	\end{align}
	It is proven in \cite{FKR} that the limit defining $\delta_{\mathcal{H}}(p)$ exists, and the the Euler product converges to a non-zero limit provided that $\delta_{\mathcal{H}}(p) > 0$ for all $p \not\equiv 1 \mod 4$.

	\begin{conjecture}\cite[Conjecture~1.1]{FKR} \label{conj-HL}
	Fix $k \geq 1$, and $\mathcal{H} = \left\{ h_1, \dots, h_k \right\} \subseteq \Z$. If $\mathfrak{S}( \mathcal{H} )  > 0$, then 
	$$
	R_k(\mathcal{H}; x) \sim \mathfrak{S}( \mathcal{H} ) \left( R_1(x) \right)^k \sim \mathfrak{S}( \mathcal{H} ) \left( \frac{K}{\sqrt{\log{x}}} \right)^k
	$$
	\end{conjecture}
	
	This conjecture is still open, but it is known that   $\sum_{\substack{n}}  \mathbf{1}_{\mathbf{E}}(n+h_1 ) \dots \mathbf{1}_{\mathbf{E}}(n+h_k)$ is infinite
	for $k=2,3$ by the work of Hooley \cite{Hooley1,Hooley2}.

	It is not straightforward to give a simple formula for the singular series $\mathfrak{S}( \mathcal{H} )$ for a given set $\mathcal{H}$ (see Section~\ref{Section proof prop bound for D with large sets}), except the trivial cases $\mathfrak{S}( \emptyset)=
	\mathfrak{S}( \{ h \} )=1.$ For $\mathcal{H} = \left\{ 0, h \right\}$, Connors and Keating \cite{CK} computed
\begin{equation}\label{connors-keating-normalized}
\mathfrak{S}(\lbrace 0, h\rbrace) =  \frac{1}{2K^2} W_2(h) \prod_{\substack{p\equiv 3 \pmod{4}\\ p\mid h}} \frac{1 - p^{-v_p(h) -1}}{1-p^{-1}},
\end{equation} 
where
		$$W_2(h) = \begin{cases}
		1 \text{ if } 2\nmid h \\
		2 - 3\cdot2^{-v_2(h)} \text{ otherwise,}
		\end{cases}$$ and $v_p$ is the $p$-adic valuation.
		
		Notice that it means that $\mathfrak{S}( \mathcal{H} ) > 0$ when $k=2$. This can also be proven for $k=3$, but for general~$k$, we can find sets~$\mathcal{H}$ such that $\mathfrak{S}( \mathcal{H} ) = 0$. 
		It is easy to see that  $\sum_{\substack{n}}  \mathbf{1}_{\mathbf{E}}(n+h_1 ) \dots \mathbf{1}_{\mathbf{E}}(n+h_k)$ is finite 
when $\mathfrak{S}( \mathcal{H} ) = 0$.

	We now state a slight generalization of the Hardy--Littlewood conjecture where $n$ is restricted to an arithmetic progression modulo $q$.

	\begin{conjecture} \label{SOS-HL} (Hardy--Littlewood for \sos in arithmetic progressions) Fix $k \geq 1$, and $\mathcal{H} = \left\{ h_1, \dots, h_k \right\} \subseteq \Z$. 
	Let $q\equiv 1 \pmod{4}$ be a prime, and $a \in  \Z$.
	If $\mathfrak{S}( \mathcal{H} )  > 0$, then 
	\begin{eqnarray*}
	R_k(\mathcal{H}; x, q, a) &:=&  \frac{1}{x} \sum_{\substack{n \leq x\\ n \equiv a \mod q}}  \mathbf{1}_{\mathbf{E}}(n+h_1 ) \dots \mathbf{1}_{\mathbf{E}}(n+h_k)\\
	&\sim& \frac{\mathfrak{S}( \mathcal{H} )}{q}  \left( \frac{K}{\sqrt{\log{x}}} \right)^k.
	\end{eqnarray*}
	\end{conjecture}
	
We remark that unlike the generalized Hardy--Littlewood conjecture of \cite{LOS}, we do not need to adjust the local factors at the prime numbers dividing $q$ in  $\mathfrak{S}( \mathcal{H} )$ since we fixed $q$ to be prime with $q \equiv 1 \mod 4$, and this prime does not appear in the Euler product~\eqref{def singular series product} defining 
 $\mathfrak{S}( \mathcal{H} )$.

 In Conjectures~\ref{conj-HL} and~\ref{SOS-HL}, we used $K/\sqrt{\log{n}}$ for the probability that $n$ is a sum of two squares. As the secondary term for this probability depends on the residue class modulo $q$ from Theorem~\ref{thm-SOS-AP}, we get more precise results by using this second term to refine the probability in Conjecture~\ref{SOS-HL}. We state that in the conjecture below, 
	and we used it to illustrate the fit with the numerical data in Table~\ref{table:CK_5_a_h=1-first},
but not in the rest of the paper while getting in the heuristic model leading to  Conjecture~\ref{main-conjecture-intro} and Conjecture
\ref{general-conjecture-intro}  (as those secondary terms would be smaller than some error terms occuring in the heuristic).

	\begin{conjecture} \label{SOS-HL-refined} (Refined Hardy--Littlewood in arithmetic progressions) Fix $k \geq 1$, and $\mathcal{H} = \left\{ h_1, \dots, h_k \right\} \subseteq \Z$. 
	Let $q\equiv 1 \pmod{4}$ be a prime, and $a \in  \Z$.
	If $\mathfrak{S}( \mathcal{H} )  > 0$, then 
	\begin{eqnarray*}
	R_k(\mathcal{H}; x, q, a) 
	&\sim&  \frac{ \mathfrak{S}( \mathcal{H} )}{q} K^k \left( \frac{1}{(\log{x})^{k/2}} + \frac{1}{(\log{x})^{k/2+1}}\sum_{h \in \mathcal{H}}c_{1, h { +a}}  + O \left(\frac{1}{(\log x)^{k/2+2}}\right)\right),
	\end{eqnarray*}
	where $c_{1,h}$ is defined by ~\eqref{constant-AP}. 
	\end{conjecture}

\begin{table}[htb!]
\begin{tabular}{cc|ccc|cc}
\hline 
\makecell{$a$} & \makecell{$h$} & \makecell{$x R_k(\mathcal{H}; x, q, a) $} & \makecell{Main term} & \makecell{Main and secondary term} & \makecell{$\text{Err}_1$} & \makecell{$\text{Err}_2$}\\
\hline
\hline
$0$ & $1$ & 3 906 419 030 & 3 619 120 683 & 3 850 620 130 & 1.0794 & 1.0145 \\
$1$ & $1$ & 3 751 339 794 & 3 619 120 683 & 3 718 867 172 & 1.0365 & 1.0087 \\
$1$ & $2$ & 1 925 818 092 & 1 809 560 341 & 1 859 433 586 & 1.0642 & 1.0357 \\
$0$ & $5$ & 4 062 607 000 & 3 619 120 682 & 3 982 373 088 & 1.1225 & 1.0201 \\
\hline
\end{tabular}
\caption{Numerical data versus Conjecture~\ref{SOS-HL-refined} for $\mathcal{H}= \left\{ 0, h \right\}$, $x=10^{12}$, $q = 5$. The third column shows the numerical data, the 4-th and 5-th columns show the product of $x$ and the prediction of Conjecture~\ref{SOS-HL-refined} with the main term, and with the main and first secondary term respectively. The last two columns show their percentage errors, respectively.}
\label{table:CK_5_a_h=1-first}
\end{table}

Finally, we need an equivalent form of Conjecture~\ref{SOS-HL}, inspired by the work of Montgomery and Soundararajan \cite{MS} for the case of primes, namely
\begin{equation} \label{def-S0q} 
   \frac{1}{x}  \sum_{\substack{n \leq x \\ n \equiv a \mod q}}\prod_{h\in\mathcal{H}}\bigg(  \mathbf{1}_{\mathbf{E}}(n+h)
    -\frac{K}{\sqrt{\log n}}\bigg) \sim  \frac{\mathfrak{S}_{0}(\mathcal{H})}{q} \bigg(\frac{K}{\sqrt{\log x}}\bigg)^{|\mathcal{H}|} .
\end{equation}
Assuming that Conjecture~\ref{SOS-HL} holds, we get relations between the constants ${\mathfrak{S}_{0}(\mathcal{H})}$ and ${\mathfrak{S}(\mathcal{H})}$, and 
it is easy to see that
\begin{eqnarray*}
	\mathfrak{S}_{0}(\emptyset) &=&1 \\
	\mathfrak{S}_{0}(\left\{ h\right\} ) &=&0 \\
	\mathfrak{S}_{0}(\left\{ h_1, h_2 \right\} ) &=&  \mathfrak{S}(\left\{ h_1, h_2 \right\} ) - 1,
\end{eqnarray*}
and that for a general set $\mathcal{H}$,
\begin{equation}\label{altdef-S0q}
	\mathfrak{S}_{0}(\mathcal{H})  = \sum_{\mathcal{T}\subseteq\mathcal{H}} (-1)^{\lvert \mathcal{H} \smallsetminus \mathcal{T} \rvert}\mathfrak{S}(\mathcal{T}).
\end{equation}
Mirroring \cite{MS}, we prove in Section~\ref{Section proof prop bound for D with large sets} the following result, which is critical to justify our heuristic.

 \begin{theorem}\label{Prop bound for D with large sets}
	Let $\mathfrak{S}_0(\mathcal{H})$ the constants defined by~\eqref{altdef-S0q}. Then, for any $k \geq 1$ and $\varepsilon >0$, we have
	\begin{align*}
		\sum_{\substack{ \mathcal{H} \subseteq [1,h] \\ \lvert\mathcal{H}\rvert = k}}\mathfrak{S}_0(\mathcal{H}) \ll_{k,\varepsilon} h^{\frac{k}{2}+\varepsilon}.
	\end{align*}
\end{theorem}
Note that our result is weaker than the result of Montgomery and Soundararajan who computed an aymptotic for the average of Theorem~\ref{Prop bound for D with large sets} in the case of primes \cite[Theorem~2]{MS}. We did not pursue that as Theorem~\ref{Prop bound for D with large sets} is sufficient for our application.

 	\section{Heuristic for the conjecture}  \label{section-heuristic-1}
	
	We now develop a heuristic leading to Conjecture~\ref{main-conjecture-intro} following \cite{LOS}.
	Let $q \equiv 1 \pmod{4}$ be a prime, and we recall that $$N(x;q,(a,b)) = \# \lbrace E_{n} \leq x : E_n \equiv a \pmod{q}, E_{n+1} \equiv b \pmod{q}\rbrace.$$
	We first write
	\begin{align} \label{counting-consecutive}
	N(x; q,(a,b)) = \sum_{\substack{n \leq x \\ n \equiv a \pmod{q}}} \sum_{\substack{h>0 \\ h \equiv b-a \pmod{q}}} \mathbf{1}_{\mathbf{E}}(n) \mathbf{1}_{\mathbf{E}}(n+h) \prod_{t = 1}^{h-1} (1 - \mathbf{1}_{\mathbf{E}}(n+t)).
	\end{align}
We introduce the notation $$\widetilde{ \mathbf{1}}_{\mathbf{E}}(n)= \mathbf{1}_{\mathbf{E}}(n)-\frac{K}{\sqrt{\log n}},$$
and for each fixed $h$ in~\eqref{counting-consecutive}, we study the sum
\begin{eqnarray*}
S_h :=  \sum_{\substack{n \leq x \\ n \equiv a \mod q}}  
 \bigg(\frac{K}{\sqrt{\log n}}+\widetilde{ \mathbf{1}}_{\mathbf{E}}(n)\bigg)\bigg(\frac{K}{\sqrt{\log{(n+h)}}}+
\widetilde{ \mathbf{1}}_{\mathbf{E}}(n+h) 
 \bigg) 
  \prod_{0<t<h}\bigg( 1 -\frac{K}{\sqrt{\log{(n+t)}}}- \widetilde{ \mathbf{1}}_{\mathbf{E}}(n+t) \bigg).
\end{eqnarray*}
If we ignore the small differences among $\sqrt{\log{n}}$, $\sqrt{\log(n+h)}$, and $\sqrt{\log{(n+t)}}$ and we expand out the product, we get
\begin{align*} 
S_h=    \sum_{\mathcal{A} \subset \lbrace 0,h \rbrace} \sum_{\mathcal{T} \subset [ 1,h-1 ]} (-1)^{|\mathcal{T}|}\sum_{\substack{n \leq x \\ n \equiv a \mod q}} \bigg(\frac{K}{\sqrt{\log n}}\bigg)^{2-|\mathcal{A}|}\prod_{\substack{t\in[1,h-1] \\ t \not\in \mathcal{T}}} \bigg( 1 - \frac{K}{\sqrt{\log n}} \bigg) \prod_{\substack{t\in \mathcal{A \cup T}}} \widetilde{ \mathbf{1}}_{\mathbf{E}}(n+t)
\end{align*}
Finally,  denoting
$$\alpha(n) = 1-\frac{K}{\sqrt{\log n}},$$ 
and using~\eqref{def-S0q}, we conjecture that
\begin{align*}
  S_h &=  \sum_{\mathcal{A} \subset \lbrace 0,h \rbrace} \sum_{\mathcal{T} \subset [ 1,h-1 ]} (-1)^{|\mathcal{T}|}\sum_{\substack{n \leq x \\ n \equiv a \mod q}} \bigg(\frac{K}{\sqrt{\log(n)}}\bigg)^{2-|\mathcal{A}|} \alpha(n)^{{h -1 - | \mathcal{T}|}}
    \prod_{\substack{t\in \mathcal{A \cup T}}}\widetilde{\mathbf{1}}_{\mathbf{E}}(n+t)\\
&\sim \frac{x}{q} \sum_{\mathcal{A} \subset \lbrace 0,h \rbrace} \sum_{\mathcal{T} \subset [ 1,h-1 ]} (-1)^{|\mathcal{T}|}\mathfrak{S}_{0} (\mathcal{A} \cup \mathcal{T})  \bigg(\frac{K}{\sqrt{\log x}}\bigg)^{2+|\mathcal{T}|} \alpha(x)^{h -1-\lvert\mathcal{T}\rvert}.  \end{align*}
We emphasize that this is a heuristic argument: in obtaining this expression for $S_h$, we have not paid attention to the error terms in~\eqref{def-S0q}, in particular on the
  dependency on the size of the sets $\mathcal{A} \cup \mathcal{T}$ and on $h$.

Summing $S_h$ over all $h \equiv b-a \mod q$, this gives the conjectural estimate
\begin{equation} \label{total-count-D}
  N(x;q,(a,b)) \sim \frac{x}{q}  {\alpha(x)^{-1}}  \bigg( \frac{K}{\sqrt{\log x}} \bigg)^{2}  \mathcal{D}(a,b;x),
\end{equation}
where
\begin{equation} \label{first-D}
  \mathcal{D}(a,b; x) = \sum_{\substack{h>0 \\ h \equiv b-a \mod q}} \sum_{\mathcal{A} \subset \lbrace 0 , h \rbrace} \sum_{\mathcal{T}\subset [1,h-1]}(-1)^{|\mathcal{T}|} \mathfrak{S}_{0} (\mathcal{A} \cup \mathcal{T}) \bigg( \frac{K}{\alpha(x) \sqrt{\log x}} \bigg)^{|\mathcal{T}|} \alpha(x)^{{h}}.
\end{equation}
In order to evaluate~\eqref{first-D}, we will use the following notations. Let
\begin{align} \label{changeofvariables}
	\alpha(x)^h = \left( 1 - \frac{K}{\sqrt{\log{x}}} \right)^h = e^{-h/H} \iff H = - \frac{1}{\log{\alpha(x)}},\end{align}  which implies that
\begin{eqnarray*}
	H &=& \frac{\sqrt{\log{x}}}{K} - \frac12 + O \left( (\log x)^{-1/2} \right)\\
	\log{H} &=& \frac{1}{2} \log{\log{x}} - \log{K} + O \left((\log x)^{-1/2} \right).
\end{eqnarray*}

\subsection{Discarding the singular series involving larger sets}
We approximate  $\mathcal{D}(a,b;x) $ by discarding all the singular series where $\mathcal{A} \cup \mathcal{T}$ has more than 2 elements, which is justified by 
Theorem~\ref{Prop bound for D with large sets}. We separate in 3 cases, depending on the possible choices for the set $\mathcal{A} \subseteq \{ 0, h \}$.
We use the notation defined in~\eqref{changeofvariables} for $H$, and the bound $\sum\limits_{\substack{ h>0 \\h \equiv v \mod q}} \alpha(x)^h h^{\ell} \ll_{\ell} H^{\ell+1}$ for any $\ell \geq 0$, and $v \in \Z$.

If $\mathcal{A} = \emptyset$, then
for $k\geq3$, we deduce from Theorem~\ref{Prop bound for D with large sets} that
\begin{align*}
	\sum_{\substack{h>0 \\ h \equiv b-a \mod q}} & \sum_{\substack{\mathcal{T}\subset [1,h-1] \\ \lvert\mathcal{T}\rvert =k}} \mathfrak{S}_{0} (\mathcal{T}) \bigg( \frac{K}{\alpha(x) \sqrt{\log x}} \bigg)^{k} \alpha(x)^{{h}}
	\ll_{k} 
	\bigg( \frac{K}{\alpha(x) \sqrt{\log x}} \bigg)^{k} \sum_{\substack{h>0 \\ h \equiv b-a \mod q}} h^{\frac{k}{2}+\varepsilon}  \alpha(x)^{{h}} \\
	&\ll_{k} \bigg( \frac{K}{\alpha(x) \sqrt{\log x}} \bigg)^{k} H^{1  +\frac{k}{2}+\varepsilon} 
	\ll_{k} (\log x)^{-\frac{k}{4} + \frac{1}{2}  + \varepsilon}.
\end{align*}

If $\mathcal{A} = \{ h \}$ and $\lvert\mathcal{A} \cup \mathcal{T}\rvert \geq 3$, we have for 
$k \geq 2$
\begin{align*}
	\sum_{\substack{h>0 \\ h \equiv b-a \mod q}}& \sum_{\substack{\mathcal{T}\subset [1,h-1] \\ \lvert\mathcal{T}\rvert =k}} \mathfrak{S}_{0} (\mathcal{T}\cup\lbrace h \rbrace) \bigg( \frac{K}{\alpha(x) \sqrt{\log x}} \bigg)^{k} \alpha(x)^{h} \\
	&\approx_{k} \frac{1}{q}
	\bigg( \frac{K}{\alpha(x) \sqrt{\log x}} \bigg)^{k} \sum_{\substack{\mathcal{D}\subset [1,H] \\ \lvert\mathcal{D}\rvert =k+1}} \mathfrak{S}_{0} (\mathcal{D}) 
	\ll_{k} \frac{1}{q}\bigg( \frac{K}{\alpha(x) \sqrt{\log x}} \bigg)^{k} H^{\frac{k+1}{2} + \varepsilon} 
	\ll_{k} \frac{1}{q} (\log x)^{-\frac{k}{4} + \frac{1}{4}  + \varepsilon},
\end{align*}
where we are approximating the sum over $h$ and ${\mathcal{T}}$ of the first line as the sum over all $\mathcal{D}$ of size $k+1$ contained in $[1,H]$, which we then bound by Theorem~\ref{Prop bound for D with large sets}.
We obtain the same bound for $\mathcal{A} = \{ 0 \}$ using the fact that $\mathfrak{S}_{0}$ is invariant by translation.

Finally, in the case $\mathcal{A} = \lbrace 0,h\rbrace$, we introduce an extra average. Since  $\mathfrak{S}_{0}$ is translation invariant, we have
$$ \sum_{s\geq 1} \mathfrak{S}_{0} ( \lbrace s, t_1+s, \dots, t_k+s, h+s \rbrace)e^{-{s}/{H}}
= \mathfrak{S}_{0} ( \lbrace 0, t_1, \dots, t_k, h \rbrace)\sum_{s\geq 1} e^{-{s}/{H}} \approx \mathfrak{S}_{0} ( \lbrace 0,  t_1, \dots, t_k, h \rbrace) H,$$ and using this, we get for 
 $k\geq 1$
\begin{align*}
	\bigg( \frac{K}{\alpha(x) \sqrt{\log x}} \bigg)^{k}&\sum_{\substack{h>0 \\ h \equiv b-a \mod q}} \sum_{\substack{\mathcal{T}\subset [1,h-1] \\ \lvert\mathcal{T}\rvert =k}} \mathfrak{S}_{0} (\mathcal{T} \cup \lbrace 0, h \rbrace) \alpha(x)^{h} \\
	&\approx \frac{1}{q H} 	\bigg( \frac{K}{\alpha(x) \sqrt{\log x}} \bigg)^{k} \sum_{s\geq 1}\sum_{h\geq 1} \sum_{0< t_1 <\dots < t_k <h} \mathfrak{S}_{0} ( \lbrace s, t_1+s, \dots, t_k+s, h+s \rbrace)e^{-{(s+h)}/{H}} \\
	&\approx \frac{1}{q H} \bigg( \frac{K}{\alpha(x) \sqrt{\log x}} \bigg)^{k}\sum_{0<s< t'_1 <\dots < t'_k <h'<2H} \mathfrak{S}_{0} ( \lbrace s, t'_1, \dots, t'_k, h' \rbrace) \\ 
	&\ll_k \frac{1}{q}\bigg( \frac{K}{\alpha(x) \sqrt{\log x}} \bigg)^{k} (2H)^{-1 + \frac{k+2}{2} + \varepsilon} \ll_{k} \frac{1}{q} (\log x)^{-\frac{k}{4}  + \varepsilon}.
\end{align*}

Discarding all the singular series where $\mathcal{A} \cup \mathcal{T}$ has more than 2 elements from~\eqref{first-D}, and working again heuristically by ignoring the dependence on $\lvert \mathcal{A} \cup \mathcal{T}\rvert$ in the error terms, we are led to the model

\begin{equation*} 
  \mathcal{D}(a,b; x) = (\mathcal{D}_0+\mathcal{D}_1+\mathcal{D}_2)(a,b;x) + O_{\varepsilon}((\log x)^{-\frac{1}{4} + \varepsilon}),
\end{equation*}
 where
\begin{equation*}
  \mathcal{D}_0(a,b;x) =  \sum_{\substack{h>0 \\ h \equiv b-a \mod q}} \bigg(1+ \mathfrak{S}_{0} (\lbrace 0,h \rbrace) \bigg) \alpha(x)^{{h}}
\end{equation*}
\begin{equation*}
  \mathcal{D}_1(a,b;x) = - \bigg( \frac{K}{\alpha(x) \sqrt{\log x}} \bigg) \sum_{\substack{h>0 \\ h \equiv b-a \mod q}} \sum_{t \in [1,h-1]} \bigg(\mathfrak{S}_{0} (\lbrace 0,t \rbrace) + \mathfrak{S}_{0} (\lbrace t,h \rbrace)\bigg) \alpha(x)^{{h}}
\end{equation*}
\begin{equation*}
  \mathcal{D}_2(a,b;x) =  \bigg( \frac{K}{\alpha(x) \sqrt{\log x}} \bigg)^2 \sum_{\substack{h>0 \\ h \equiv b-a \mod q}} \sum_{1 \leq t_1 < t_2 < h} \bigg(\mathfrak{S}_{0} (\lbrace t_1,t_2 \rbrace)\bigg) \alpha(x)^{{h}}.
\end{equation*}
Replacing in~\eqref{total-count-D}, we then conjecture that up to 
error term of order $x(\log x)^{-\frac{5}{4} + \varepsilon}$, we have
\begin{equation} \label{total-count-D012}
  N(x;q,(a,b)) \sim \frac{x}{q}   {\alpha(x)^{-1}}  \bigg( \frac{K}{\sqrt{\log x}} \bigg)^{2}  (\mathcal{D}_0 + \mathcal{D}_1+  \mathcal{D}_2)(a,b;x) .
\end{equation}

\subsection{Evaluation of the sums of singular series involving sets of size $2$}\label{section-heuristic-2}

In order to evaluate~\eqref{total-count-D012}, we first evaluate the simple exponential sums.
We will use the notation
$$
f(v; q) := \begin{cases} - \frac{1}{2} & v = 0  \\  \frac{q - 2v}{2q} & 1 \leq v \leq q-1 \end{cases}
$$
which gives
\begin{eqnarray*} 
E(H) &:=& 	
\sum_{h>0 } e^{-h/H} = H - \frac{1}{2} + O(H^{-1}) 
= \frac{\sqrt{\log{x}}}{K}- 1  
+ O \left( ( \log{x})^{-1/2} \right) \\
E(q,v;H) &:=&
\sum_{\substack{h>0 \\ h \equiv v \mod q}} e^{-h/H} = \frac{H}{q} + f(v; q) + O(H^{-1}) 
= \frac{\sqrt{\log{x}}}{Kq} + f(v; q)  - \frac{1}{2q}
+ O \left( ( \log{x})^{-1/2} \right) .
\end{eqnarray*}
Let 
\begin{equation}  \label{eq def S(q,v;H)-S0(q,v;H)} \begin{split}
S(q,v; H) &:= \sum_{\substack{h \geq 1 \\ h \equiv v \mod q}} \mathfrak{S}(\{ 0,h \})e^{-h/H} 
\\
S_0(q,v;H) &:= \sum_{\substack {h \geq 1\\h \equiv v \mod q}}  \mathfrak{S}_{0} (\lbrace 0,h \rbrace) e^{-h/H} = 
\sum_{\substack {h \geq 1 \\h \equiv v \mod q}}  (\mathfrak{S}(\lbrace 0,h \rbrace) -1)e^{-h/H}. 
\end{split} \end{equation}
and 
\begin{align*} S(H) &:= \sum_{h \geq 1} \mathfrak{S}(\{ 0,h \})e^{-h/H} = \sum_{v \mod q} S(q,v; H) \\
	S_0(H) &:= \sum_{h \geq 1}  \mathfrak{S}_{0} (\lbrace 0,h \rbrace) e^{-h/H} 
	= \sum_{v \mod q} S_0(q,v; H). 
	\end{align*}
We then have
\begin{equation} \label{S0H} \begin{split}
S_0(q,v; H) &= S(q,v; H) - \frac{H}{q} - f(v;q) + O(H^{-1})\\
S_0(H) &= S(H) - H + \frac12 + O(H^{-1}).
\end{split} \end{equation}

Using Theorem~\ref{Intro prop2.1}, we evaluate $\mathcal{D}_{0}(a,b;x)$, $\mathcal{D}_{1}(a,b;x)$ and $\mathcal{D}_{2}(a,b;x)$.
\begin{proposition}\label{Prop estim D012} Let $q \equiv 1 \mod 4$ be a prime. For $j \geq 1$, let $c(j)$ be the constants from Theorem~\ref{Intro prop2.1}. 
 Then, 
  \begin{align*} 
  		&\mathcal{D}_{0}(a,b;x) + \mathcal{D}_{1}(a,b;x) +\mathcal{D}_{2}(a,b;x) \\
  		 &= S(q, b-a; H) + 
		 \frac{2}{q K \pi}(\log H)^{1/2} - \frac1{2q} -  \frac{1}{q} \sum_{j=1}^J c(j) (\log{H})^{1/2-j} + O \left( (\log{H})^{-1/2-J} + \frac{\sqrt{\log{H}}}{\sqrt{\log{x}}} \right)
  	\end{align*} 
where we use the change of variables ~\eqref{changeofvariables}. We remark that the error term $(\log{H})^{-1/2-J}$ is the largest one, for any value of $J$.
	\end{proposition}

\begin{proof}
First, notice that $\mathcal{D}_{0}(a,b;x) = S(q, b-a; H)$. 
For $\mathcal{D}_{1}(a,b;x)$, we first compute
  \begin{eqnarray} \label{formula sum D1}  \sum_{\substack {{h \geq 2}\\ h \equiv b-a \mod q}} 
  \sum_{1 \leq t \leq h-1}  \mathfrak{S}_{0} (\lbrace 0,t \rbrace) e^{-h/H} &=&
 \sum_{t \geq 1}  \mathfrak{S}_{0} (\lbrace 0,t \rbrace) e^{-t/H} \sum_{\substack {h \geq 1 \\h \equiv b-a-t \mod q}} e^{-h/H} \\
 &=& \left( \frac{H}{q} + O(1) \right) S_0(H), \nonumber
  \end{eqnarray}
 and \begin{eqnarray*}
  - \bigg( \frac{K}{\alpha(x) \sqrt{\log x}} \bigg) \sum_{\substack {h > 0\\ h \equiv b-a \mod q}} \sum_{1 \leq t \leq h-1}  \mathfrak{S}_{0} (\lbrace 0,t \rbrace) e^{-h/H} 
  = \left( -\frac{1}{q} + O\left(  \frac{1}{\sqrt{\log{x}}} \right) \right)  {S}_{0} (H).
  \end{eqnarray*}
  We get a similar estimate for the second sum in  $\mathcal{D}_{1}(a,b;y)$ involving $\mathfrak{S}_{0} (\lbrace t, h \rbrace)$ by making a change of variable to replace it by $\mathfrak{S}_{0} (\lbrace 0, r \rbrace)$ with $r=h-t$, which gives
  \begin{eqnarray*} 
  \mathcal{D}_{1}(a,b;x) 
  &=& {\left( -\frac{2}{q} + O\left( (\log{x})^{-1/2} \right) \right)S_0(H)} .
  \end{eqnarray*} 
 Similarly, for  $\mathcal{D}_{2}(a,b; x)$,  we first compute
    \begin{align}
 & \sum_{\substack {h \geq 3\\ h \equiv b-a \mod q}} \sum_{1 \leq t_1<t_2 < h}  \mathfrak{S}_{0} (\lbrace t_1,t_2 \rbrace) e^{-h/H} \nonumber \\
  &=
   \sum_{1 \leq t_1<t_2 }  \mathfrak{S}_{0} (\lbrace 0, t_2-t_1 \rbrace) \sum_{\substack{h \equiv b-a \mod q \\ h \ge t_2 +1}} e^{-h/H}
  = \sum_{r \ge 1}  \mathfrak{S}_{0} (\lbrace 0, r \rbrace) \sum_{\substack{t_2 \ge r+1}} e^{-t_2/H} \sum_{\substack{h' \geq 1\\ h' \equiv b-a - t_2 \mod q}} e^{-h'/H}  \nonumber \\
    &= \sum_{r \ge 1}  \mathfrak{S}_{0} (\lbrace 0, r \rbrace) e^{-r/H} \sum_{\substack{t_2' \geq 1}} e^{-t_2'/H} \sum_{\substack{h' \geq 1\\ h' \equiv b-a - t_2' {-r} \mod q}} e^{-h'/H}
    = { \left(\frac{H^2}{q} + O(H) \right)  S_0(H),}  \label{formula sum D2}
  \end{align}
  and replacing in the definition of $\mathcal{D}_{2}(a,b;x)$, we have
  \begin{equation*} 
  \mathcal{D}_{2}(a,b;x) 
   = \left( \frac{1}{q} + O \left( (\log x)^{-1/2} \right) \right) S_0(H).
  \end{equation*}
  Using Theorem~\ref{Intro prop2.1} and~\eqref{S0H} to evaluate $S_0(H)$, this completes the proof.
 	\end{proof}
	
One can be more precise regarding the dependence on the congruence classes by separating the sum over~$t$ in~\eqref{formula sum D1} and the sum over~$r$ in~\eqref{formula sum D2} in congruence classes modulo $q$. In particular, the following refinement of Proposition~\ref{Prop estim D012} will be used for numerical testing. The proof follows directly from the proof of Proposition
\ref{Prop estim D012}, and we omit it.
	\begin{proposition} \label{Prop D012 numerical} Let $q \equiv 1 \mod 4$ be a prime. Then
  \begin{align*}
  &\mathcal{D}_{0}(a,b;x) + \mathcal{D}_{1}(a,b;x)  + \mathcal{D}_{2}(a,b;x) \\
  &= E(q, b-a; H) + S_0(q, b-a; H) - 2 \frac{K}{\alpha(x) \sqrt{\log{x}}} \sum_{c \mod q} S_0(q, b-a-c; H) E(q,c;H) \\ &+ \left( \frac{K}{\alpha(x) \sqrt{\log{x}}} \right)^2 \sum_{c, d \mod q} S_0(q, b-a-c-d; H) E(q,c; H) E(q,d; H). 
\end{align*}
	\end{proposition}
	
\subsection{Completing the heuristic}	
	
We now deduce Conjecture~\ref{main-conjecture-intro}, by replacing Theorem~\ref{Intro prop2.1} and Proposition~\ref{Prop estim D012} in~\eqref{total-count-D012}. If $a \equiv b \mod q$, we have
 \begin{align} \nonumber
  N(x;q,(a,a)) 
=& \frac{x K^2}{q \log x}\left(1+ \frac{K}{\sqrt{\log x}} + O \left( \frac{1}{\log{x}} \right) \right) \bigg[ \frac{\sqrt{\log x}}{Kq}  - \frac{2(q-1)}{qK\pi}(\log H)^{1/2}  - \frac1{q} \\ \nonumber
& \qquad  + \sum_{j=1}^J \left(c_0(j) - \frac{c(j)}{q} \right) \;(\log H)^{1/2-j} + O \left( (\log{H})^{-J-1/2} \right) 
\bigg]\\
\label{other-one}
=& \frac{K x}{q^2 \sqrt{\log x}}   \bigg[ 1 +  \frac{1}{\sqrt{\log x}}  \sum_{j=0}^J b_0(j) \;(\log H)^{1/2-j} + O \left( \frac{1}{\sqrt{\log{x}} (\log{H})^{J+1/2}} \right) \bigg]
\end{align}
where $b_0(0)= -2(q-1)/\pi$ and $b_0(j) = K(q c_0(j) - c(j))$ for $j \geq 1$.

If $a \not\equiv b \mod q$, we have
\begin{align} \nonumber
	N(x;q,(a,b)) 
	=& \frac{x K^2}{q \log x}\left(1+ \frac{K}{\sqrt{\log x}} + O \left( \frac{1}{\log{x}} \right) \right) \bigg[ \frac{\sqrt{\log x}}{Kq}  + \frac{2}{qK\pi}(\log H)^{1/2}  - \frac1{q}  
	\\ \nonumber &+ {\frac{1}{2 K^2 \phi(q)}} \sum_{\substack{\chi \mod q \\\chi\neq\chi_0}} \chi(v)^{-1} C_{q,\chi} + \sum_{j=1}^J  \left(c_1(j) - \frac{c(j)}{q} \right) \; (\log H)^{1/2-j}
	+ O \left( (\log{H})^{-J-1/2} \right)  \bigg] \\
	\label{case a b distinct}
	=&  \frac{K x}{q^2 \sqrt{\log x}}   \bigg[1 + C_{a,b} + \frac{1}{\sqrt{\log x}}  \sum_{j=0}^J b_1(j) \;(\log H)^{1/2-j} + O \left( (\log{H})^{-J-1/2} \right) 
	 \bigg] \\ \nonumber
	=&  \frac{K x}{q^2 \sqrt{\log x}}   \bigg[1 + C_{a,b} - \frac{1}{\phi(q)}  \frac{1}{\sqrt{\log x}}  \sum_{j=0}^J b_0(j) \;(\log H)^{1/2-j} + O \left( (\log{H})^{-J-1/2} \right) \bigg], 
\end{align}
where $C_{a,b} = {\frac{q}{2 K \phi(q)}} \sum_{\substack{\chi \mod q \\\chi\neq\chi_0}} \chi(v)^{-1} C_{q,\chi} $, 
$b_1(0) = {2}/{\pi}$ and $b_1(j) = K(qc_1(j)-c(j))$ for $j\geq 1$. 
For the last line, we used~\eqref{relation-between-c} which gives $b_1(j)= -\frac{b_0(j)}{\phi(q)}$, for $j \geq 0$.

To deduce Conjecture~\ref{main-conjecture-intro} and obtain the explicit expressions for the constants~$C_j$ for~$0 \leq j \leq J$, from the above expressions~\eqref{other-one} and~\eqref{case a b distinct},
we approximate $(\log{H})^{1/2-j}$ for $0 \leq j \leq J$, where $H$ is given by~\eqref{changeofvariables}. 
We illustrate the process below for $J=1$. 

Using the approximations
\begin{align*}
(\log{H})^{1/2} &= \frac{1}{\sqrt{2}} \sqrt{\log{\log{x}}} - \frac{\log{K}}{\sqrt{2}}\frac{1}{\sqrt{\log{\log{x}}}} +  O \left((\log\log x)^{-3/2} \right),\\
(\log{H})^{-1/2} &= \frac{\sqrt{2}}{ \sqrt{\log{\log{x}}}}   + O \left((\log\log x)^{-3/2} \right),
\end{align*}
we obtain
\begin{align} \label{last-one}
 \sum_{j=0}^1 b_0(j) \;(\log H)^{1/2-j} &= -\frac{\sqrt{2} (q-1)}{\pi} \sqrt{\log{\log{x}}} + \left( \frac{\sqrt{2} (q-1) \log{K}}{ \pi} + \sqrt{2} b_0(1) \right) (\log \log x)^{-1/2} \\
 & \nonumber + O
 \left( (\log \log x)^{-3/2} \right)
 \end{align}
 Using  the values of $c(1)$ and $c_0(1)$ given by~\eqref{Value c} and~\eqref{Values c_1 c_0}, we have 
 \begin{align*}
 b_0(1) &=  K(q c_0(1) - c(1)) = \frac{\phi(q)}{\pi}\left(\frac{\omega + \gamma}{2} + \frac{q}{\phi(q)}\log{q}\right), 
 \end{align*}
 where $\gamma$ is the Euler-Mascheroni constant and $\omega$ is defined in Lemma \ref{lemma-Z}.
 Replacing in~\eqref{last-one}  and then in~\eqref{other-one}, we get Conjecture~\ref{main-conjecture-J1} below which is the special case of Conjecture~\ref{main-conjecture-intro} for $J=1$. The case $a\not\equiv b \pmod{q}$ follows from multiplying the corresponding term by $\frac{-1}{\phi(q)}$ in~\eqref{case a b distinct}.
The general case of Conjecture~\ref{main-conjecture-intro} follows similarly by using approximations for $(\log{H})^{-1/2-j}$ as above for $1 \leq j \leq J$.

\begin{conjecture} \label{main-conjecture-J1} Fix $q \equiv 1 \mod 4$. Then,
\begin{align*}
		N(x,q,(a,a)) &\sim \frac{x}{q^2} \frac{ K}{ \sqrt{\log x}} \bigg( 1 - \frac{\sqrt{2} \phi(q)}{\pi}\frac{(\log \log x)^{1/2}}{(\log x)^{1/2}} 
		 + \frac{C_1}{(\log x)^{1/2} (\log\log x)^{1/2}}  \bigg) 
\end{align*}
up to an error term of $O \left( \displaystyle \frac{x}{\log{x}(\log{\log x})^{3/2}} \right),$ and 
with
\begin{align*}
C_1 
&= \frac{\sqrt{2} \phi(q)}{\pi} \left( \log{K} + \frac{\omega + \gamma}{2} \right) + \frac{\sqrt{2} q \log{q}}{\pi},
\end{align*}
where $\gamma$ is the Euler-Mascheroni constant and $\omega$ is defined in Lemma \ref{lemma-Z}.

For $a \neq b \mod q$, 
\begin{align*}
	N(x,q,(a,b)) &= \frac{x}{q^2} \frac{ K}{ \sqrt{\log x}} \bigg( 1 + \frac{\sqrt{2}}{\pi} \frac{\sqrt{\log \log{x}}}{\sqrt{\log{x}}} + \frac{C_{a,b} }{\sqrt{\log{x}}} - \frac{C_1}{\phi(q)(\log x)^{1/2} (\log\log x)^{1/2}} \bigg)
\end{align*}
up to an error term of $O \left( \displaystyle \frac{x}{\log{x}(\log{\log x})^{3/2}} \right),$ and 
with 
\begin{equation*} 
C_{a,b} := \frac{1}{2K}\frac{q}{\phi(q)}\sum_{\chi\neq \chi_0}\overline{\chi}(b-a)C_{q,\chi}
\end{equation*}
where the sum is over the non-principal Dirichlet characters modulo $q$ and $C_{q, \chi}$ is defined in~\eqref{def-C}. 
\end{conjecture}

\subsection{Another formulation of Theorem~\ref{Intro prop2.1}}

Finally, as we did in Theorem~\ref{thm-S0S-integral} for the counting function of the number of \sos up to $x$, we now state a different version of Theorem~\ref{Intro prop2.1} with a very good error term by using an integral form for the main term. We prove this Proposition in Section~\ref{section-improved-averages}, and we use it for numerical testing in Section~\ref{section-numerical}.

\begin{proposition} \label{prop-lucia-2}
	Fix $\varepsilon > 0$. There exists $c>0$ such that for $v \not\equiv 0 \mod q$
\begin{align*} 
S(q, v, H) 
&= \frac{H}{q} + \frac{1}{2K^2 \phi(q)} \sum_{\chi \neq \chi_0} \chi(v)^{-1} C_{q, \chi} + \frac{1}{2  \pi K^2 \phi(q)}\int_{1/2+\varepsilon}^1 \frac{F_{\chi_0}(\sigma) H^{\sigma-1}}{|\sigma-1|^{1/2}} d\sigma + O \left( \exp(-c \sqrt{\log{H}} ) \right), 
\end{align*}
and
\begin{align*}
S(q, 0, H) 
	&= \frac{H}{q} + 
	\frac{1}{ \pi  K^2}\int_{1/2+\varepsilon}^1 \frac{F'(\sigma) H^{\sigma-1} + F(\sigma) H^{\sigma-1} (\log{H} - A_q(\sigma)/2)}{|\sigma-1|^{1/2}} d\sigma + O \left( \exp(-c \sqrt{\log{H}} ) \right)\\	
	\end{align*}
	where $F(s)=\zeta(s-1) M(s-1) \left[ (s-1) \zeta(s) \right]^{1/2}  \Gamma(s)$ and $F_{\chi_0}(s) =  \frac{1-q^{-(s-1)}}{s-1}   F(s)$,
with $M(s)$ as defined by~\eqref{def-M} and $A_q(s) =  \frac{1-q^{-(s-1)}}{s-1}  $.
	Assuming the Riemann Hypothesis, we can replace the error terms above by
	$O \left( H^{-1/2+\varepsilon} \right)$.
	\end{proposition}

Observe that close to $\frac12$ one has $F'(\sigma) \asymp (\sigma - \tfrac12)^{-\frac54}$, so the error term in the formula for $S(q, 0, H)$ depends strongly on $\varepsilon$.

\section{Proof of Theorem~\ref{Intro prop2.1} } \label{proof-prop-section}

\begin{proof} As in \cite{smilansky}, we define
 $a(h) = 2 K^2 \mathfrak{S}(\{ 0,h \})$. Then using~\eqref{connors-keating-normalized}, we see that $a(h)$ is a multiplicative function of $h$ with
 \begin{equation*}
    a(p^k) = 
    \begin{cases}
        1 &\quad \text{ for } p \equiv 1 \mod 4, \\ 
       \displaystyle 2-\frac{3}{2^k} &\quad \text{ for } p = 2, k \geq 1,\\ 
      \displaystyle  \frac{1-p^{-(k+1)}}{1-p^{-1}} &\quad \text{ for } p \equiv 3 \mod 4.
    \end{cases}
\end{equation*}
Using Mellin Inversion, we have
$$2 K^2 S(H) 
= \sum_{\substack{h \geq 1}}  a(h)
e^{-h/H} = 
\frac{1}{2 \pi i}\int_{(2)} D(s)H^s\Gamma(s) ds,$$
where $D(s)=\sum_{h \geq 1} {a(h)}{h^{-s}}.$
Similarly, for $\chi$ a character modulo $q$, 
\begin{align} \label{after-mellin} 2 K^2 S(H, \chi) 
:= \sum_{\substack{h \geq 1}}  a(h) \, \chi(h) 
e^{-h/H} = 
\frac{1}{2 \pi i}\int_{(2)} D_{ \chi}(s)H^s\Gamma(s) ds, \end{align} 
where $D_{ \chi}(s)=\sum_{h \geq 1} {a(h) \chi(h)}{h^{-s}}.$

In order to compute $S(H)$ and $S(H, \chi)$, we move the contour integral and pick up the contributions of the singularities of the integrand. So, first, we need to understand the analytic properties of the generating series $D(s)$ and $D_{\chi} (s)$.
Using the formulas for $a(p^k)$ above, we have
$$\hspace*{-2cm}
     D_{\chi} (s) = R_\chi(s) P_\chi(s) Q_\chi(s),   
$$
where
\begin{align*}
 R_\chi(s) &= 1 + 2\bigg( \frac{\chi(2) 2^{-s}}{1- \chi(2) 2^{-s}} \bigg) - 3 \bigg( \frac{\chi(2) 2^{-(s+1)}}{1- \chi(2) 2^{-(s+1)}} \bigg)\\
P_\chi(s) &= \prod_{p \equiv 1 \: (4)} ( 1- \chi(p)p^{-s} )^{-1}\\
Q_\chi(s) &= \prod_{p \equiv 3 \: (4)} (1-\chi(p)p^{-s})^{-1}(1-\chi(p)p^{-(s+1)})^{-1}. \end{align*}
This can be rewritten as 
$$D_{\chi}(s) = L(s,\chi) (1- \chi(2) 2^{-s}) R_\chi(s)Q_{1,\chi}(s) = L(s,\chi) L(s+1,\chi)^{\frac{1}{2}} M_{\chi}(s) $$
where
\begin{align*} 
Q_{1,\chi}(s) &= \prod_{p \equiv 3 \: (4)}  (1-\chi(p)p^{-(s+1)})^{-1} \\ 
&= \Big(\frac{L(s+1,\chi)}{L(s+1,\chi\cdot\chi_4)}\Big)^{\frac12} (1-\chi(2)2^{-(s+1)})^{\frac12}  \prod_{p \equiv 3 \pmod 4} (1 - \chi(p)^2 p^{-2(s+1)})^{-\frac12}, \\
M_{\chi}(s) &= \Bigg(1+ \chi(2) 2^{-s} - 3 \bigg( \frac{\chi(2) 2^{-(s+1)}(1- \chi(2) 2^{-s})}{1- \chi(2) 2^{-(s+1)}} \bigg) \Bigg) L(s+1,\chi\cdot\chi_4)^{-\frac12} \\ & \quad \times(1-\chi(2)2^{-(s+1)})^{\frac12} \prod_{p \equiv 3 \pmod 4} (1 - \chi(p)^2 p^{-2(s+1)})^{-\frac12}\\
&= (1 - \chi(2) 2^{-s} + \chi(4)2^{-2s} ) L(s+1,\chi\cdot\chi_4)^{-\frac12} \\ & \quad \times(1-\chi(2)2^{-(s+1)})^{-\frac12} \prod_{p \equiv 3 \pmod 4} (1 - \chi(p)^2 p^{-2(s+1)})^{-\frac12},
\end{align*}
where $\chi_4$ is the primitive character modulo $4$.
The formula for $Q_{1,\chi}(s)$ follows from developing the identity $(1 - \chi(p)^2 p^{-2(s+1)}) = (1 - \chi(p) p^{-(s+1)}) (1 + \chi(p) p^{-(s+1)})$.
Since $\chi \neq \chi_4$, the function $M_{\chi}$ is holomorphic in the half plane $\re(s) \geq 0$, and we can push this limit a bit further to the left depending on the zero-free region of $L(s+1,\chi\cdot\chi_4)$ (up to $\re(s) > - \frac12$ under Riemann Hypothesis).

In the case $\chi$ is a non-principal character, $L(s,\chi)$ is entire on the complex plane, and $L(s+1,\chi)^{\frac{1}{2}}$ is holomorphic in a region containing the half-plane $\re(s)\geq 0$, where $L(s+1,\chi)$ does not vanish. 
As such, there is no pole or singularity in the integrand at $s=1$. 
 If we shift the line of integration of~\eqref{after-mellin} to the left of the line $\re(s) = 0$ using the standard zero-free region and estimates for $L$-functions, we obtain (for some $c>0$ and any $\varepsilon > 0$)
\begin{align} \label{non-trivial} S(H, \chi) = \frac{C_{q,\chi}}{2K^2} + \begin{cases} 
{O}(H^{-1/2 + \varepsilon}) & \text{under GRH} \\ {O}\left(\exp \left( - c \sqrt{\log{H}} \right) \right)& \text{otherwise,} \end{cases}
\end{align}
where the constant term comes from the contribution of pole of order 1 from $\Gamma(s)$ at $s=0$, and 
\begin{align}\label{def-C}
C_{q,\chi} &= D_{ \chi}(0) = L(0,\chi)L(1,\chi)^{\frac12}M_{\chi}(0) \\ 
&= L(0,\chi)L(1,\chi)^{\frac12}L(1,\chi\cdot\chi_4)^{-\frac12} (1 - \chi(2) + \chi(4) ) (1-\chi(2)2^{-1})^{-\frac12} \prod_{p \equiv 3 \pmod 4} (1 - \chi(p)^2 p^{-2})^{-\frac12} . \nonumber
\end{align}
Note that $C_{q,\chi} \neq 0$ only when $\chi(-1) = -1$.

 If $\chi=\chi_0$, $D_{ \chi_0}(s)$ has a simple pole at $s=1$ with residue $2 K^2 \phi(q)/q$, and no other singularities for $\re(s) > 0$, and we move the integral to $\re(s) = \varepsilon > 0$. This gives
\begin{align} \nonumber
\sum_{\substack{h \geq 1}} a(h) \chi_0(q) e^{-h/H} &= \frac{1}{2 \pi i}\int_{(2)} D_{ \chi_0} (s) H^s\Gamma(s) ds  \\ 
\label{before-hankel-D0} &= 2 \frac{\phi(q)}{q} K^2 H+ \frac{1}{2 \pi i}\int_{(\varepsilon)} D_{ \chi_0}(s) H^s\Gamma(s) ds. 
\end{align}
Similarly, we have that 
\begin{align} \label{before-hankel-D}
\sum_{\substack{h \geq 1}} a(h)e^{-h/H} &= \frac{1}{2 \pi i}\int_{(2)} D(s) H^s\Gamma(s) ds  = 2 K^2 H+ \frac{1}{2 \pi i}\int_{(\varepsilon)} D(s) H^s\Gamma(s) ds, 
\end{align}
where
$$D(s) = \zeta(s)(1-  2^{-s}) R(s)Q_{1}(s)  = \zeta(s) \zeta(s+1)^{\frac{1}{2}} M(s),$$
and the functions $R, Q_1, M$ are obtained by taking $\chi \equiv 1$ in the previous definitions.

To account for the contribution of the singularity of $D_{ \chi_0}(s)$ and $D(s)$ as $s=0$ to the integrals~\eqref{before-hankel-D0} and~\eqref{before-hankel-D}, we use  again the Selberg--Delange method. Since we are evaluating a Mellin transform, we cannot use directly \cite[Theorem~13.2]{koukoulopoulos} as in Section~\ref{section-SOS-AP}, but we are following the same standard steps. We first approximate the  line of integration $\Re(s)=\varepsilon$ by the truncated segment  from $\varepsilon - iT$ to $\varepsilon + iT$, which we then deform to a truncated 
Hankel's contour. This is possible since there are no residue inside this contour. We then replace this contour by the infinite Hankel's countour $\mathcal{H}$ of Figure~\ref{fig-hankel} with a very good error term, which allows us to use Theorem~\ref{hankel} to compute the contribution of the singularity of the generating functions (for each term of the Taylor series). We refer the reader to \cite[Chapter~13]{koukoulopoulos} and \cite[Chapter~5]{tenenbaum} for more details. The contributions to $S(H)$ and $S(H, \chi_0)$
will be different in magnitude, because the singularites of $\zeta(s) \zeta(s+1)^{1/2}$ and $L(s,\chi_0) L(s+1,\chi_0)^{1/2}$ at $s=0$ are different, since $L(0, \chi_0)=0$, but $\zeta(0) \neq 0$.

\begin{theorem}[Hankel's formula \cite{tenenbaum} Theorem 0.17 p.179] \label{hankel} Fix any $r>0$, and 
let $\mathcal{H}$ be the Hankel's countour, which is the path consisting of the circle $|s|=r$ excluding the point $s=-r$, and of the half-line $(-\infty, -r]$ covered twice, 
with respective arguments $\pi$ and $-\pi$. Then for any complex number $z$, we have
$${\frac{1}{2\pi i}}\int_{\mathcal{H}} s^{-z} e^s \; ds = \frac{1}{\Gamma(z)}.$$
\end{theorem}
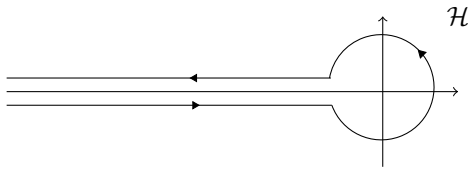
\begin{figure}[] 
\centering
\begin{tikzpicture}

\draw [->](-5,0) -- (1.0,0);
\draw [<-] (0,1) -- (0,-1.0);
\node[scale = 0.09, isosceles triangle,
	isosceles triangle apex angle=60,
	draw,
	rotate=135,
	fill=black,
	minimum size =1cm] (T2)at (0.518,0.505){};
\node[scale = 0.08, isosceles triangle,
	isosceles triangle apex angle=60,
	draw,
	rotate=180,
	fill=black,
	minimum size =1cm] (T2)at (-2.5,0.18116){};
	\node[scale = 0.08, isosceles triangle,
	isosceles triangle apex angle=60,
	draw,
	rotate=0,
	fill=black,
	minimum size =1cm] (T2)at (-2.5,-0.18116){};
\draw (-0.67613,-0.18116) arc
	[
		start angle=200,
		end angle=530,
		x radius=0.7cm,
		y radius =0.7cm
	] ;
\draw []  (-0.67613,-0.18116) -- (-5,-0.18116);
\draw []  (-5,0.18116) -- (-0.69,0.18116);

\node at (1,1) {$\mathcal{H}$};

\end{tikzpicture}
  \caption{Hankel's Contour}
  \label{fig-hankel}
\end{figure}

We first work with $D(s) = \zeta(s) \zeta(s+1)^{\frac{1}{2}} M(s)$.
The function $M(s)$  is analytic around $s=0$, and
\begin{equation*} 
	M(0) =  2K L(1, \chi_4)^{-\frac{1}{2}} = \frac{4K}{\sqrt{\pi}}. 
\end{equation*}
Then, $s^{\frac32}D(s) \Gamma(s) $ is analytic and non-zero around $s=0$ with Taylor series
$\sum_{n\geq 0} c_n s^n$,
and we write
\begin{align*} 
D(s) \Gamma(s)  = \frac{a{(3/2)}}{s^{3/2}} + \frac{a{(1/2)}}{s^{1/2}} + a{(-1/2)} s^{1/2} + \dots. \end{align*}
We now compute the contribution to the integral~\eqref{before-hankel-D} for each term of the series above using Theorem~\ref{hankel}.
For every term $a(z)/s^z$ of the Taylor series above (where $z = 3/2, 1/2, -1/2, \dots$), we have
\begin{align*}
\frac{1}{2 \pi i}\int_{\mathcal{H}} a(z) s^{-z} H^s ds &= \frac{a(z)}{2 \pi i}\int_{\mathcal{H}} s^{-z} e^{s \log{H}} ds \\ &= \frac{ a(z) {(\log H})^{z-1}}{{2 \pi i}} \int_{\mathcal{H}} t^{-z} e^{t} dt =
\frac{ a(z) {(\log H})^{z-1}}{\Gamma(z)} , \end{align*}
where we used the change of variables $t=s \log H$. 
This gives, for any integer $N\geq 1$,
\begin{align*} 
\frac{1}{2 \pi i}\int_{(\varepsilon)} D(s) H^s\Gamma(s) ds =
 \sum_{n=0}^N \frac{a(3/2-n) (\log H)^{1/2-n}}{\Gamma(3/2-n)} + O \left( (\log H)^{1/2-N-1} \right).
\end{align*}
Replacing in~\eqref{before-hankel-D}, this gives
\begin{equation} \label{SH}
S(H) = H +   \sum_{j=0}^J {c(j)(\log H)^{1/2-j}} + O \left( (\log H)^{-1/2-J} \right),
\end{equation}
with  $$c(j) = \displaystyle \frac{a(3/2-j)}{2K^2 \Gamma(3/2-j)}, \;\;j \geq 0.$$
To complete the proof of Theorem~\ref{Intro prop2.1}, we now compute the values of $c(0)$ and $c(1)$. 
Using the expansions\footnote{
	We used the Laurent expansion of $\zeta(s)$ at $s = 1$ 
	\begin{align*}
		\zeta(s) = \frac{1}{s-1} + \sum_{n = 0}^{\infty}(-1)^n\frac{\gamma_n}{n!} (s-1)^n, \hspace{0.3cm} \text{with } \gamma_n := \lim_{N \rightarrow \infty}\left(\sum_{1 \ge k \ge N}\frac{\log^n k}{k} - \int_1^N \frac{\log^n t}{t} dt\right), \hspace{0.3cm} \text{and } \gamma := \gamma_0.
\end{align*}} around $s=0$ 
\begin{align*}
\sqrt{ \zeta(s+1)} = \frac{1}{s^{1/2}} \left( 1 + \frac{\gamma}{2} s + O(s^2) \right), \qquad   & \qquad
\Gamma(s) = \frac{1}{s} - \gamma  + O( s ) ,
\end{align*}
we have
\begin{align} \label{TS1}
D(s) \Gamma(s)  =  \frac{1}{s^{\frac{3}{2}}} \zeta(s) M(s) \left( 1 -  \frac{ \gamma}{2}s + O(s^2) \right),
\end{align}
which gives
$$a(3/2) =  \zeta(0) M(0)  = -\frac{2K}{\sqrt{\pi}} \;\;\Longrightarrow\;\; c(0) = \frac{a(3/2)}{2 K^2\Gamma(3/2)} = -\frac{2}{K\pi}.
$$
To get the value of $c(1)$, we need the first 2 terms of the Taylor series around $s=0$ of the analytic function
\begin{align} \label{def-G}
Z(s) = \zeta(s) M(s) = Z(0) + Z'(0) s + O \left( s^2 \right).
\end{align}
Replacing in~\eqref{TS1}, and using Lemma~\ref{lemma-derivatives-G} for the value of $Z'(0)$, we have
\begin{align}\label{Value c}
a(1/2)  =  \left( Z'(0) - \frac{\gamma}{2} Z(0)   \right) =   \frac{K (\omega+\gamma)}{\sqrt{\pi}} \;\;\Longrightarrow\;\; c(1) = \frac{a(1/2)}{2 K^2\Gamma(1/2)} 
= \frac{\omega+\gamma}{2 \pi K},
\end{align}
where $\gamma$ is the Euler-Mascheroni constant and $\omega$ is defined in Lemma~\ref{lemma-derivatives-G}.

We now turn to the secondary term for the sum
\begin{align} \nonumber
\sum_{\substack{h \geq 1}} a(h) \chi_0(h) e^{-h/H} &= \frac{1}{2 \pi i}\int_{(2)} D_{ \chi_0} (s) H^s\Gamma(s) ds  \\ 
\label{before-hankel-D0-second} &= 2 \frac{\phi(q)}{q} K^2 H+ \frac{1}{2 \pi i}\int_{(\varepsilon)} D_{ \chi_0}(s) H^s\Gamma(s) ds 
\end{align}
 which is similar to the above replacing 
$D(s)$ with $D_{ \chi_0},$ where $\chi_0$ is the principal character modulo $q$.
We have
$$D_{\chi_0}(s) = L(s,\chi_0) L(s+1,\chi_0)^{\frac{1}{2}} M_{\chi_0}(s),$$
where 
\begin{multline*}
M_{\chi_0}(s) = (1- 2^{-s} + 2^{-2s}) L(s+1,\chi_0\cdot\chi_4)^{-\frac12} (1-2^{-(s+1)})^{-\frac12} \prod_{p \equiv 3 \pmod 4} (1 - p^{-2(s+1)})^{-\frac12}
\end{multline*}
since $\chi_0(p)=1$ for each $p \nmid q$ and $q \equiv 1 \pmod 4$. 
We remark that  $M_{\chi_0} (s) = (1-q^{-(s+1)})^{-\frac12} M(s)$, which implies that 
$$D_{\chi_0}(s) = L(s,\chi_0) \zeta(s+1)^{\frac{1}{2}} M(s) = (1-q^{-s}) D(s).$$
Then writing $(1-q^{-s}) = (\log q)s +O(s^2)$, we notice that $s^{1/2} D_{\chi_0}(s)\Gamma(s)$ is analytic and non-zero around $s=0$. 
Indeed, $L(s, \chi_0)$ has a simple zero at $s=0$ which cancels the pole of~$\Gamma(s)$.
Around $s=0$, we write
\begin{align*} 
D_{\chi_0}(s)\Gamma(s) = \frac{b(1/2)}{s^{1/2}} + b(-1/2)s^{1/2} + b(-3/2) s^{3/2}  + \dots, \end{align*} 
and working as above this gives
\begin{align*} 
\frac{1}{2 \pi i}\int_{(\varepsilon)} D_{\chi_0}(s) H^s\Gamma(s) ds =
 \sum_{n=0}^N \frac{b(1/2-n) (\log H)^{-1/2-n}}{\Gamma(1/2-n)} + O \left( (\log H)^{-1/2-N-1} \right),
\end{align*}
replacing in~\eqref{before-hankel-D0-second}, we have
\begin{align} \label{trivial-char-1}
\begin{split}
S(H, \chi_0) 
&=  \frac{\phi(q)}{q} H + \sum_{j=1}^J  c(j, \chi_0) (\log H)^{1/2-j}  + O \left( (\log{H})^{-1/2-J} \right).
\end{split}
\end{align}
Using the expansion of $D(s)\Gamma(s)$  above, we have 
$$b(1/2) = a(3/2)\log{q}  = -\frac{2K}{\sqrt{\pi}}\log{q} \;\;\Longrightarrow\;\; c(1, \chi_0) = \frac{b(1/2)}{2 K^2\Gamma(1/2)} = -\frac{1}{K {\pi}}\log{q}.$$

We now  complete the proof of Theorem~\ref{Intro prop2.1}.
Using~\eqref{non-trivial} and~\eqref{trivial-char-1}  and the orthogonality relations, we have for $v \neq 0$,
\begin{align*}
   S(q,v,H) &= \sum_{\substack{h\geq1 \\ h \equiv v \mod q}} \mathfrak{S}(\{ 0,h \}) e^{-h/H} 
    = \frac{1}{\phi(q)} \sum_\chi \chi(v)^{-1} S(H, \chi) \\
    &= \frac{1}{\phi(q)} S(H, \chi_0)   + {\frac{1}{2 K^2 \phi(q)}} \sum_{\chi\neq\chi_0} \chi(v)^{-1} C_{q,\chi}  + O(\exp(-c\sqrt{\log H})) \\
    &= \frac{H}{q} 
     + {\frac{1}{2 K^2 \phi(q)}} \sum_{\chi\neq\chi_0} \chi(v)^{-1} C_{q,\chi}   + \sum_{j=1}^J \frac{c(j, \chi_0)}{\phi(q)} (\log{H})^{1/2-j} + O \left( (\log{H})^{-1/2-J}  \right).
\end{align*}
For $v=0$, we use~\eqref{SH} and the above to get
\begin{eqnarray*}
S(q,0,H) &=& S(H) - \sum_{v \in ({\Z}/{q\Z})^*} S(q,v,H)\\
&=&  \frac{H}{q} - \frac{2}{K\pi} \sqrt{\log \: H} + \sum_{j=1}^J \left( c(j) - c(j, \chi_0) \right) (\log{H})^{1/2-j} 
+ O \left( (\log{H})^{-1/2-J}  \right),
\end{eqnarray*}
where we used the fact that
\begin{eqnarray} \label{def-M}
 \sum_{v \in ({\Z}/{q\Z})^*}\sum_{\chi\neq\chi_0} \chi(v)^{-1} C_{q,\chi} = 0\end{eqnarray}
by the orthogonality relations. 
This completes the proof of the proposition, with $c_1(j) = c(j, \chi_0)/\phi(q)$ and $c_0(j) = c(j) - c(j, \chi_0)$ for $j \geq 1$, from which the relation
$c_0(j)+\phi(q) c_1(j)=c(j)$ easily follows. From the values $c(1)$ and $c(1, \chi_0)$ computed above, we have
\begin{align}\label{Values c_1 c_0}
	c_1(1) = - \frac{ \log{q}}{K\phi(q) \pi} \;\; \mbox{and} \;\; c_0(1)  = \frac{1}{K\pi}\left(  \frac{\omega+\gamma}{2} + \log{q} \right),
\end{align}
where $\gamma$ is the Euler-Mascheroni constant and $\omega$ is defined in Lemma~\ref{lemma-derivatives-G}.
\end{proof}

\begin{lem} \label{lemma-derivatives-G} \label{lemma-Z} Let $Z(s)$ be the function defined by~\eqref{def-G}. Then,
\begin{align*}
Z'(0) & =  \frac{K}{\sqrt{\pi}} \omega \approx -0.3851314513\ldots
\end{align*}
where 
$$\omega = \log\frac{2}{\pi^2} + \frac{L'(1, \chi_4)}{L(1, \chi_4)}+ 2\sum_{p \equiv 3 (4)}\frac{\log{p}}{p^{2}-1}.$$ 
\end{lem}

\begin{proof}
Firstly, observe $M(s)$ in~\eqref{def-G} and rewrite it as
$$
Z(s) = \zeta(s)M(s) = \zeta(s)A(s)B(s)
$$
where 
\begin{equation*}
\begin{split}
A(s) &= 1-2^{-s}+2^{-2s},\\
B(s) &= \bigg( L(s+1, \chi_4)\big(1-2^{-(s+1)}\big)\prod_{p \equiv 3 (4)}\bigg(1-\frac{1}{p^{2(s+1)}}\bigg)\bigg)^{-1/2}.
\end{split}
\end{equation*}
Then, we have
\[
\begin{split}
\frac{Z'}{Z}(0) &= \frac{\zeta'}{\zeta}(0) + \frac{A'}{A}(0) + \frac{B'}{B}(0),\\
\end{split}
\]
Hence, we need to compute $\zeta, A, A', B$ and $B'$ at $s = 0$. Indeed, the following special values for $\zeta$ are well-known:
\[
\begin{split}
\zeta(0) &= -\frac{1}{2}, \quad \zeta'(0) = -\frac{\log 2\pi}{2}.
\end{split}
\]
Moreover, for $A(s)$, we have $A(0) = 1, A'(0) = -\log 2.$
We may use the recursive formula for $B(s)$ to obtain $B(0) = {4K}/{\sqrt{\pi}}$ and
\[
\frac{B'}{B}(0) =  -\frac{1}{2}\left(\frac{L'(1, \chi_4)}{L(1, \chi_4)} + \log{2} + 2\sum_{p \equiv 3 (4)}\frac{\log{p}}{p^{2}-1}\right) = -\frac{1}{2}\left(\log{2} + \alpha_1 + \beta_1\right) 
\]
where we denote
\begin{align*} 
\alpha_1 = \frac{L'(1, \chi_4)}{L(1, \chi_4)},  = 0.2456096036\ldots, 
\quad \beta_1 = 2\sum_{p \equiv 3 (4)}\frac{\log{p}}{p^{2}-1} = 0.4574727064\ldots.
\end{align*}
One can compute the value of $\alpha_1$ by $L(1, \chi_4) = {\pi}/{4} \text{ and } L'(1, \chi_4) = 0.192901331574902\ldots$.
\end{proof}

\section{Proof of Theorem~\ref{Prop bound for D with large sets}}\label{Section proof prop bound for D with large sets}

In the heuristic leading to Conjecture~\ref{main-conjecture-intro}, we used Theorem~\ref{Prop bound for D with large sets} to justify  that the terms involving a sum of singular series for sets with three or more elements contribute to the error term. 
Theorem~\ref{Prop bound for D with large sets}  is an analogue of~\cite[Theorem~2]{MS} of Montgomery and Soundararajan adapted from primes to sums of two squares.
We now prove Theorem~\ref{Prop bound for D with large sets}, following closely the argument developed in~\cite{MS}, without giving all the details but insisting on the points that are different in the case of the sums of two squares.
To help with the comparison, we stay close to the notation used in loc.~cit.\! so we may use notation that differs from the rest of the paper, which should not cause trouble to the reader as this section is relatively independent from the rest of the paper.
We use the standard notation $e(x)= e^{2i\pi x}$.

\subsection{The singular series}
The first step in the proof is to write the singular series as an actual series (and not a Euler product), the way it was introduced by Hardy and Littlewood (see \cite[Lemma~3]{MS}).
We begin with giving a new expression for the local factors of the singular series. Let $\mathcal{D} = \left\{ d_1, \dots, d_k \right\} \subseteq {\Z}.$
We recall that
	for any $p \not\equiv 1 \mod 4$, we have
\begin{align*}
	\delta_{\mathcal{D}}(p) = \lim_{\alpha\rightarrow\infty} \frac{\# \lbrace 0 \leq a < p^{\alpha} : \forall d \in \mathcal{D}, a + d \equiv \square + \square \pmod{p^{\alpha}} \rbrace}{p^{\alpha}},
\end{align*}
and the singular series 
is defined by
\begin{align*}
	\mathfrak{S}( \mathcal{D} ) := \prod_{p \not\equiv 1 \mod 4} \frac{\delta_{\mathcal{D}}(p)}{(\delta_{ \{ 0 \}}(p))^k}.
\end{align*}
\begin{lem} \label{lemma-HL} 
	Let $\mathcal{D} = \lbrace d_1, \dots, d_k\rbrace\subseteq \Z$ 
	be a set with $k$ elements. 
	For any prime number $p \not\equiv 1 \pmod{4}$,
	one has
	$$ 
	\frac{\delta_{\mathcal{D}}(p)}{\delta_{\lbrace 0\rbrace}(p)^k} = \sum_{q_1, \dots, q_k \mid p^{\infty}} \prod_{i =1}^{k} \frac{\lambda_2(q_i)}{q_i} A_{\mathcal{D}}(q_1,\dots, q_k),
	$$
	where for any $q_1, \dots, q_k \in \N$, 
	$$ A_{\mathcal{D}}(q_1,\dots, q_k) = \sum_{\substack{a_1, \dots, a_k \\ 1\leq a_i \leq q_i,  (q_i,a_i) =1 \\ \sum_{i =1}^{k} \frac{a_i}{q_i} \in \Z }} e\left(\sum_{i =1}^{k} \frac{a_i d_i}{q_i} \right)\prod_{i =1}^k C(q_i,a_i),
	$$
		with $$C(q,a) = \begin{cases}
		1 &\text{ if } q \text{ is odd,} \\
        0 &\text{ if } 2\mid q \text{ but } 4 \nmid q, \\
	    2 e(-a/4)  &\text{ if } 4 \mid q.
	\end{cases}$$ 
	and $\lambda_2$ is the multiplicative function defined on the prime powers by $$
	\lambda_2(p^m)=\begin{cases}
		(-1)^m &\text{ if } p \text{ is odd,}\\
		1 &\text{ if } p=2.
	\end{cases}$$ 
\end{lem}

\begin{proof}
	Let $p \equiv 3 \pmod{4}$ be a prime number. 
	For $\mathcal{D} = \lbrace d_1, \dots, d_k\rbrace\subseteq \Z$ a set with $k$ elements, we deduce from \cite[Proposition~5.1, Proposition~5.3(a) and (5.4)]{FKR} that
	\begin{align*}
		\delta_{\mathcal{D}}(p) = \lim_{\alpha\rightarrow\infty} p^{-\alpha} \sum_{x=1}^{p^{\alpha}} \prod_{i=1}^{k} \mathbf{1}_{S_{p,\alpha}}(x+d_i)
	\end{align*}
	where $\mathbf{1}_{S_{p,\alpha}}$ is the characteristic function of the set $S_{p,\alpha} =  \lbrace p^{2\beta}m  : 0\leq \beta< \tfrac{\alpha}{2}, m \not\equiv 0 \pmod{p} \rbrace$.
	In particular, for $\alpha$ even, following the idea of the proof of \cite[Lemma~2]{MV}, we write that
	\begin{align*}
		\mathbf{1}_{S_{p,\alpha}}(x)   &=  \sum_{\beta =0}^{\frac{\alpha}{2} -1}\sum_{s \mid p} \frac{\mu(s)}{p^{2\beta}s} \sum_{a =1}^{p^{2\beta}s} e\left(\frac{ax}{p^{2\beta}s}\right) \\
		&=  \sum_{\beta =0}^{\frac{\alpha}{2} -1} \frac{1}{p^{2\beta}}\left\lbrace\left(1 - \frac1p\right) \sum_{r \mid p^{2\beta}}\sum_{a \in (\Z/r\Z)^*} e\left(\frac{ax}{r}\right) -  \frac1p\sum_{a \in (\Z/p^{2\beta+1}\Z)^*} e\left(\frac{ax}{p^{2\beta+1}}\right) \right\rbrace\\
		&=  \sum_{\gamma=0}^{\alpha-2}\left\lbrace \sum_{\beta = \lceil\frac{\gamma}{2}\rceil}^{\frac{\alpha}{2} -1} \frac{1}{p^{2\beta}}\left(1 - \frac1p\right) \sum_{a \in (\Z/p^{\gamma}\Z)^*} e\left(\frac{ax}{p^{\gamma}}\right)\right\rbrace -  \sum_{\beta = 0}^{\frac{\alpha}{2} -1}\frac{1}{p^{2\beta+1}}\sum_{a \in (\Z/p^{2\beta +1}\Z)^*} e\left(\frac{ax}{p^{2\beta+1}}\right)\\
		&=  \sum_{\gamma=0}^{\alpha-1}\frac{(-p)^{-\gamma} - p^{-\alpha}}{1 + \frac1p} \sum_{a \in (\Z/p^{\gamma}\Z)^*} e\left(\frac{ax}{p^{\gamma}}\right),\\
	\end{align*}
	where we used the fact that $\alpha$ is even in the last line. 	Thus, we have,
	\begin{align*}
		\frac{\delta_{\mathcal{D}}(p)}{\delta_{\lbrace 0\rbrace}(p)^k}  &=  \lim_{\substack{\alpha\rightarrow\infty \\ \alpha \text{ even}}} p^{-\alpha} \sum_{x=1}^{p^{\alpha}} \prod_{i=1}^{k} \left(  \sum_{\gamma_i=0}^{\alpha-1}((-p)^{-\gamma_i} - p^{-\alpha}) \sum_{a_i \in (\Z/p^{\gamma_i}\Z)^*} e\left(\frac{(x+d_i)a_i}{p^{\gamma_i}}\right) \right),
	\end{align*}
	since  $\delta_{\lbrace 0\rbrace}(p) = (1+\tfrac{1}{p})^{-1}$ \cite[Proposition~5.3(c)]{FKR}.
	We swap the sums and begin with the sum over~$x$, as $\gamma_1, \dots, \gamma_k \leq \alpha -1$, we have
	\begin{align*}
		p^{-\alpha}\sum_{x=1}^{p^{\alpha}}	e\left(x \sum_{i=1}^{k}\frac{a_i}{p^{\gamma_i}}\right) = \begin{cases}
			1 &\text{ if }  \sum_{i=1}^{k}\frac{a_i}{p^{\gamma_i}} \in {\Z} \\
			0 &\text{ otherwise.}
		\end{cases}
	\end{align*}
	This yields
	\begin{align*}
		\frac{\delta_{\mathcal{D}}(p)}{\delta_{\lbrace 0\rbrace}(p)^k}  &=  \lim_{\substack{\alpha\rightarrow\infty \\ \alpha \text{ even}}}   \sum_{\gamma_1=0}^{\alpha-1} \dots  \sum_{\gamma_k=0}^{\alpha-1} \prod_{i=1}^{k} ((-p)^{-\gamma_i} - p^{-\alpha})A_{\mathcal{D}}(p^{\gamma_1},\dots,p^{\gamma_k}).
	\end{align*}
	We obtain the formula announced in the Lemma for $p \equiv 3 \pmod 4$ by taking the limit $\alpha \rightarrow \infty$, and using the bound $ \lvert A_{\mathcal{D}}(q_1,\dots, q_k) \rvert \leq \frac{q_1\dots q_k}{[q_1,\dots,q_k]}$ (see~\eqref{eq bound A}).

	The proof is similar for $p=2$.	
	By \cite[Proposition~5.2(a) and (5.3)]{FKR}, for $\alpha\geq 2$ we can take $S_{2,\alpha} =  \lbrace 2^{\beta}m  : 0\leq \beta< \alpha-1, m \equiv 1 \pmod{4} \rbrace$,
	and \cite[Proposition~5.2(c)]{FKR} gives $\delta_{\lbrace 0\rbrace}(2) = \frac12$.
	
	We write that
	\begin{align*}
		\mathbf{1}_{S_{2,\alpha}}(x) &=  \sum_{\beta =0}^{\alpha -2}2^{-\beta - 2} \sum_{a =1}^{2^{\beta}} e\left(\frac{ax}{2^{\beta}}\right) \sum_{t =1}^{4} e\left((\tfrac{x}{2^{\beta}} -1) \frac{t}{4}\right)  \\
		&=  \sum_{\beta =0}^{\alpha -2}2^{-\beta - 2} \sum_{r \mid2^{\beta +2}} \sum_{b \in (\Z/r\Z)^*} e\left(\frac{(x - 2^{\beta})b}{r}\right)   \\
		&= \sum_{\gamma =0}^{\alpha}  \sum_{b \in (\Z/2^{\gamma}\Z)^*} e\left(\frac{xb}{2^{\gamma}}\right)  \sum_{\beta = \max\lbrace 0, \gamma -2 \rbrace}^{\alpha -2}2^{-\beta - 2} e\left( - 2^{\beta - \gamma}b\right) \\
		&= \sum_{\gamma =0}^{\alpha}  \sum_{b \in (\Z/2^{\gamma}\Z)^*} e\left(\frac{xb}{2^{\gamma}}\right)  \left( C(2^{\gamma},b)2^{-\gamma -1} - 2^{-\alpha}  \right), 
	\end{align*}
	note that in the sum we always have $(b,2)=1$ so $1 + e(\tfrac{-b}{2})  =0$. 
	Thus, we have
	\begin{align*}
		\frac{\delta_{\mathcal{D}}(2)}{\delta_{\lbrace 0\rbrace}(2)^k}  &=
		 \lim_{\alpha\rightarrow\infty} 2^{-\alpha+k} \sum_{x=1}^{2^{\alpha}} \prod_{i=1}^{k} \mathbf{1}_{S_{2,\alpha}}(x+d_i)
		\\
		&=  \lim_{\alpha\rightarrow\infty} 2^{-\alpha} \sum_{x=1}^{2^{\alpha}} \prod_{i=1}^{k} \left( \sum_{\gamma_i =0}^{\alpha}  \sum_{b_i \in (\Z/2^{\gamma_i}\Z)^*} e\left(\frac{(x+d_i)b_i}{2^{\gamma_i}}\right)  \left( C(2^{\gamma_i},b_i)2^{-\gamma_i} - 2^{-\alpha +1}  \right) \right)
	\end{align*}
	Exchanging the sums and computing the sum over $x$ first, this yields
	\begin{align*}
		\frac{\delta_{\mathcal{D}}(2)}{\delta_{\lbrace 0\rbrace}(2)^k}  &=  \lim_{\alpha\rightarrow\infty}   \sum_{\gamma_1=0}^{\alpha} \dots  \sum_{\gamma_k=0}^{\alpha} \prod_{i=1}^{k} 2^{-\gamma_i} A_{\mathcal{D}}(2^{\gamma_1},\dots,2^{\gamma_k}) ,
	\end{align*}
	which gives the formula announced in the Lemma for $p=2$.
\end{proof}

We now give the analogue of \cite[(44) and (45)]{MS}. {The main difference between the case of primes and the case of sum of two squares is that the local probabilities $\delta_{\mathcal{D}}(p)$ at each prime $p$ involve all powers of $p$, and then the sum over $q_1, \dots, q_k$ in Lemma~\ref{lemma-HL} runs over all integers  (and not only square-free integers). We then approximate $\mathfrak{S}(\mathcal{D})$ by taking all integers supported on primes $p \leq y$ and appearing with power at most $N$, for the appropriate values of $y$ and $N$.}

\begin{lem}\label{Eq HL approx S-lemma} 	
Let  $\mathcal{D} \subseteq\N\cap[1,h]$ be a set with $k$ elements.
	Let $y >h$, $N\geq 4\log y$, and $P_y := \prod\limits_{\substack{p\leq y\\ p\not \equiv 1 \pmod{4}}}p$.
	Then,
	\begin{align} \label{first-EQ}
		&\mathfrak{S}(\mathcal{D}) = \sum_{q_1, \dots, q_k \mid P_y^{N}}	\prod_{i =1}^{k}\frac{\lambda_2(q_i)}{q_i}  A_{\mathcal{D}}(q_1,\dots, q_k) + O_k\big(y^{-1}(\log y)^{k-1}\big) \\
	\text{and } \qquad&
	 \label{Eq HL approx S0} 
		\mathfrak{S}_0(\mathcal{D}) = \sum_{\substack{q_1, \dots, q_k \mid P_y^{N} \\ q_i >1}}	\prod_{i =1}^{k}\frac{\lambda_2(q_i)}{q_i}  A_{\mathcal{D}}(q_1,\dots, q_k) + O_k\big(y^{-1}(\log y)^{k-1}\big),
	\end{align}
	where $A_{\mathcal{D}}(q_1,\dots, q_k)$ is defined in Lemma~\ref{lemma-HL}.
\end{lem}

\begin{proof}
	First, it follows from the Chinese Remainder Theorem that for $q_1,\dots, q_k, q'_1,\dots, q'_k \in {\N}$ satisfying $\left(\prod_{i =1}^{k}q_i, \prod_{i =1}^{k}q'_i\right)=1$, one has $A_{\mathcal{D}}(q_1,\dots, q_k)A_{\mathcal{D}}(q'_1,\dots, q'_k) = A_{\mathcal{D}}(q_1q'_1,\dots, q_1q'_1)$. 
	Since $y > h \geq \max \mathcal{D}$, from \cite[Proposition~5.3.(c)]{FKR} we deduce 
	\begin{align*}
		\prod_{p >y} \frac{\delta_{\mathcal{D}}(p)}{\delta_{\lbrace 0\rbrace}(p)^k}  = \prod_{p >y} (1 + \tfrac{1}{p})^{k-1}(1 - \tfrac{k-1}{p}) = \prod_{p >y} (1 + O_k(p^{-2})) = 1 + O_k((y \log y)^{-1}). 
	\end{align*}
	By definition we have $\delta_{\mathcal{D}}(p) \leq 1$ for all prime number $p$, thus
	\begin{align*}
		\prod_{p \leq y} \frac{\delta_{\mathcal{D}}(p)}{\delta_{\lbrace 0\rbrace}(p)^k}  \leq 2^k \prod_{\substack{p \leq y \\ p \equiv 3 \pmod{4}}} (1 + \tfrac{1}{p})^{k} \ll_k (\log y)^k, 
	\end{align*}
	which gives
	\begin{align*}
		\mathfrak{S}(\mathcal{D}) = \prod_{p \leq y} \frac{\delta_{\mathcal{D}}(p)}{\delta_{\lbrace 0\rbrace}(p)^k} + O(y^{-1}(\log y)^{k-1}).
	\end{align*}
	Using Lemma~\ref{lemma-HL} and the bound $ \lvert A_{\mathcal{D}}(q_1,\dots, q_k) \rvert \leq 2^{\frac{k}{2}}\frac{q_1\dots q_k}{[q_1,\dots,q_k]}$ (see~\eqref{eq bound A}), we have 
	\begin{align*}
		\frac{\delta_{\mathcal{D}}(p)}{\delta_{\lbrace 0\rbrace}(p)^k}  
		&= \sum_{q_1, \dots, q_k \mid p^{N}}	\prod_{i =1}^{k}\frac{\lambda_2(q_i)}{q_i}  A_{\mathcal{D}}(q_1,\dots, q_k) + O_k\big( \sum_{p^{N+1} \mid q_1 \mid p^{\infty}}\sum_{q_2, \dots, q_k \mid p^{\infty}} \frac{1}{[q_1,\dots,q_k]} \big) \\
		&= \sum_{q_1, \dots, q_k \mid p^{N}}	\prod_{i =1}^{k}\frac{\lambda_2(q_i)}{q_i}  A_{\mathcal{D}}(q_1,\dots, q_k) + O_k\big( \sum_{n=N+1}^{\infty}\frac{(n+1)^{k-1}}{p^n}\big) \\
		&= \sum_{q_1, \dots, q_k \mid p^{N}}	\prod_{i =1}^{k}\frac{\lambda_2(q_i)}{q_i}  A_{\mathcal{D}}(q_1,\dots, q_k) + O_k\big( p^{-N-1}(N+2)^{k-1} \big) .
	\end{align*}
	Moreover, using again the bound~\eqref{eq bound A}, we have
	\begin{align}\label{eq bound on sum with lcm}
		\prod_{p \leq y}\sum_{q_1, \dots, q_k \mid p^{N}}	\prod_{i =1}^{k}\frac{\lambda_2(q_i)}{q_i}  A_{\mathcal{D}}(q_1,\dots, q_k)
		&\ll_k
		\sum_{q_1, \dots, q_k \mid P_y^{N}}	\prod_{i =1}^{k}\frac{\lvert \lambda_2(q_i)\rvert}{q_i} \frac{ q_1\cdot\dots\cdot q_k}{[q_1,\dots,q_k]}  \nonumber \\ 
		&\leq \sum_{q_1, \dots, q_k \mid P_y^{N} } \frac{1}{[q_1,\dots,q_k]} \nonumber \\
		&\ll_k \prod_{p\leq y} \sum_{n =0}^{N} \frac{(n+1)^{k}}{p^n}  \leq  \prod_{p\leq y}\big(1 + \frac{C_k}p\big) \ll_\varepsilon y^{\varepsilon} 
	\end{align}
	for some constant $C_k >0$, for any $\varepsilon >0$.
	Finally,
	\begin{align*}
		\prod_{p \leq y} \frac{\delta_{\mathcal{D}}(p)}{\delta_{\lbrace 0\rbrace}(p)^k} &= \sum_{q_1, \dots, q_k \mid P_y^{N}}	\prod_{i =1}^{k}\frac{\lambda_2(q_i)}{q_i}  A_{\mathcal{D}}(q_1,\dots, q_k) + O_{k,\varepsilon}\big(y^{\varepsilon} \sum_{\substack{ q \mid P_y \\ q \neq 1}} q^{-N-1}N^{(k-1)\omega(q)} \big) \\
		&= \sum_{q_1, \dots, q_k \mid P_y^{N}}	\prod_{i =1}^{k}\frac{\lambda_2(q_i)}{q_i}  A_{\mathcal{D}}(q_1,\dots, q_k) + O_{k,\varepsilon}\big( y^{\varepsilon}  2^{-N}N^{k-1} \big).
	\end{align*}
	Choosing $\frac{\log y}{\log 2}(1+ \varepsilon) < N$ gives~\eqref{first-EQ}.
	We deduce~\eqref{Eq HL approx S0} 
	 from~\eqref{first-EQ} using the formula 
	 $$\mathfrak{S}_0(\mathcal{D}) = \sum_{\mathcal{T} \subseteq\mathcal{D}} (-1)^{\lvert \mathcal{D} \setminus \mathcal{T}\rvert} \mathfrak{S}(\mathcal{T})$$ and the relation
	$A_{\lbrace d_1, \dots, d_k\rbrace}(q_1,\dots, q_{k-1},1) = A_{\lbrace d_1, \dots, d_{k-1}\rbrace}(q_1,\dots, q_{k-1})$.
\end{proof}

In particular, 
taking $y= h^{k+1}$ in~\eqref{Eq HL approx S0}, one has
\begin{align} \label{constant-with-ET}
	\sum_{\substack{ \mathcal{D} \subseteq[1,h] \\ \lvert\mathcal{D}\rvert = k}}\mathfrak{S}_0(\mathcal{D}) = \sum_{\substack{q_1, \dots, q_k \mid P_y^{N} \\ q_i >1}}	\prod_{i =1}^{k}\frac{\lambda_2(q_i)}{q_i}  \sum_{\substack{ \mathcal{D} \subseteq[1,h] \\ \lvert\mathcal{D}\rvert = k}}A_{\mathcal{D}}(q_1,\dots, q_k) + o_k(1).
\end{align}

\subsection{An easier version of the main term}
To continue with notation similar to \cite{MS}, we define
\begin{align}\label{def V}
	V_k(y,N,h) = \sum_{\substack{q_1, \dots, q_k \mid P_y^{N} \\ q_i >1}}	\prod_{i =1}^{k}\frac{\lambda_2(q_i)}{q_i}  \sum_{1\leq d_1, \dots, d_k \leq h}A_{( d_1, \dots, d_k)}(q_1,\dots, q_k),
\end{align}
where we remark that the difference with the main term above is that $d_1, \dots, d_k$ do not have to be distinct.
Let us introduce some other useful notations and results from \cite{MS}.
For $\alpha \in \R$, we denote
\begin{align}\label{def E F}
	E_h(\alpha) = \sum_{d =1}^{h} e(\alpha d)  \quad \text{ and } \quad F_h(\alpha)= \min(h,\lVert \alpha\rVert^{-1}),
\end{align}
where $\lVert \cdot \rVert$ is the distance to the nearest integer, so that we have $\lvert E_h(\alpha)\rvert \leq F_h(\alpha)$.
We have (see \cite[(54)]{MS})
\begin{align}\label{bound sum F squared}
	\sum_{a=1}^{q-1} F_h(\tfrac{a}{q})^2 \ll q \min(q,h). 
\end{align}
We will also use the following result from the work of Montgomery and Vaughan \cite{MVbas} and which is an analogue of \cite[Lemma~1]{MS} that applies to the case of non necessarily square-free numbers.
\begin{lem}[Theorem~1 of \cite{MVbas}]\label{Lem MV basic ineq}
	Let $k\geq 2$ be an integer and for $1\leq i\leq k$, let $q_i \in \N$ and $G_i$ be a $1$-periodic complex valued function. 
	Then, we have
	\begin{align*}
		\Big\lvert \sum_{\substack{a_{1}, \dots, a_k \\ 1\leq a_i \leq q_i,  (q_i,a_i) =1 \\  \sum_{i =1}^{k} \frac{a_i}{q_i} \in \Z }} \prod_{i=1}^{k} G_i(\tfrac{a_i}{q_i}) \Big\rvert
		\leq \frac{1}{[q_1,\dots,q_k]}\prod_{i =1}^{k} \big( q_i \sum_{\substack{ 1\leq a_i \leq q_i \\  (q_i,a_i) =1  }} \lvert G_i(\tfrac{a_i}{q_i})\rvert^2 \big)^{\frac12}.
	\end{align*}
\end{lem}
In particular we deduce the bound for $A_{\mathcal{D}}(q_1,\dots,q_k)$ that we used in the proofs of Lemma~\ref{lemma-HL} and~\ref{Eq HL approx S-lemma}:
\begin{align}\label{eq bound A}
	\lvert A_{\mathcal{D}}(q_1,\dots,q_k) \rvert \leq \frac{1}{[q_1,\dots,q_k]}\prod_{i =1}^{k} \big( q_i \sum_{\substack{ 1\leq a_i \leq q_i \\  (q_i,a_i) =1  }} \lvert C(q_i,a_i)\rvert^2 \big)^{\frac12}.
\end{align} 
We also have a bound for $V_{k}(y,N,h)$.
\begin{cor}\label{Lem bound V}
	For any $h, y, N >0$,
	one has $V_{k}(y,N,h) \ll_{k,\varepsilon} h^{\frac{k}{2}}y^{\varepsilon}$.
\end{cor}

\begin{proof}
	Recall that we defined
	\begin{align*}
		V_k(y,N,h) = \sum_{\substack{q_1, \dots, q_k \mid P_y^{N} \\ q_i >1}}	\prod_{i =1}^{k}\frac{\lambda_2(q_i)}{q_i}  \sum_{\substack{a_1, \dots, a_k \\ 1\leq a_i \leq q_i,  (q_i,a_i) =1 \\ \sum_{i =1}^{k} \frac{a_i}{q_i} \in \Z }}  \prod_{i =1}^k E_h\Big(\frac{a_i d_i}{q_i}\Big)C(q_i,a_i),
	\end{align*}
	where $C(q,a)=1$ for odd $q$ and $\lvert C(q,a)\rvert \leq 2$ in general.
	We use~\eqref{def E F} to write
	\begin{align*}
		\lvert V_k(y,N,h) \rvert \leq \sum_{\substack{q_1, \dots, q_k \mid P_y^{N} \\ q_i >1}}	\prod_{i =1}^{k}\frac{2}{q_i} \sum_{\substack{a_1, \dots, a_k \\ 1\leq a_i \leq q_i,  (q_i,a_i) =1 \\ \sum_{i =1}^{k} \frac{a_i}{q_i} \in \Z }}  \prod_{i =1}^k F_h(\tfrac{a_i}{q_i}).
	\end{align*}  
	Then Lemma~\ref{Lem MV basic ineq},~\eqref{bound sum F squared} and the bound in~\eqref{eq bound on sum with lcm} yield
	\begin{align*}
		\lvert V_k(y,N,h) \rvert &\leq \sum_{\substack{q_1, \dots, q_k \mid P_y^{N} \\ q_i >1}}	\frac{2^k}{[q_1,\dots,q_k]}\prod_{i =1}^{k}\frac{1}{q_i^{\frac12}} \Big(  \sum_{\substack{ 1\leq a_i \leq q_i \\ (q_i,a_i) =1 }}   F_h(\tfrac{a_i}{q_i})^2 \Big)^{\frac12} \\
        &\leq \sum_{\substack{q_1, \dots, q_k \mid P_y^{N} \\ q_i >1}}	\frac{2^k}{[q_1,\dots,q_k]}h^{\frac{k}{2}}
		\ll_{k,\varepsilon} h^{\frac{k}{2}}y^{\varepsilon} ,
	\end{align*} 
which is the bound announced.
\end{proof}

\subsection{The main estimate}
We now prove the analogue of \cite[(60)]{MS}, writing $\sum_{\substack{ \mathcal{D} \subseteq[1,h] \\ \lvert\mathcal{D}\rvert = k}} \mathfrak{S}_0(\mathcal{D}) $ in terms of~$V_k(y,N, h)$. 
Again, the idea of the proof is very similar to the work of Montgomery and Soundararajan, except that we deal with a wider summation (namely a sum over all integers instead of a sum over square-free integers).

\begin{lem}\label{Lem analogue of MS60}
	For any $h>k \in \N$,
	let $y= h^{k+1}$ and $N\geq 4\log y$.
	One has
	\begin{align*}
		\sum_{\substack{ \mathcal{D} \subseteq[1,h] \\ \lvert\mathcal{D}\rvert = k}} \mathfrak{S}_0(\mathcal{D}) 
		&= \sum_{j=0}^{k/2}  \binom{k}{2j} \frac{(2j)!}{j! 2^j}\left(- h \sum_{1< d \mid P_y^{N} } \frac{C(d) \phi(d)}{d^2} \right)^j  V_{k-2j}(y,N,h)    + O_{k,\varepsilon}(h^{\frac{k-1}{2}}y^{\varepsilon}),
	\end{align*}
	where $	V_k(y,N,h)$ is defined in~\eqref{def V}, and 	$$C(d) = \begin{cases}
		1 \text{ if } d \text{ is odd} \\
		0 \text{ if } 2 \mid q, 4 \nmid d \\
		4  \text{ if } 4 \mid d.
	\end{cases}$$
\end{lem}

\begin{proof} Following the arguments of \cite{MS}, we can prove the analogue of \cite[(52)]{MS} in our context, which is
	\begin{align} \label{analogue-52}
		\sum_{\substack{ \mathcal{D} \subseteq[1,h] \\ \lvert\mathcal{D}\rvert = k}}A_{\mathcal{D}}(q_1,\dots, q_k) = \sum_{\mathcal{P}= \lbrace \mathcal{S}_1, \dots, \mathcal{S}_M\rbrace } w(\mathcal{P})\sum_{\substack{a_1, \dots, a_k \\ 1\leq a_i \leq q_i,  (q_i,a_i) =1 \\ \sum_{i =1}^{k} \frac{a_i}{q_i} \in \Z }} \prod_{i =1}^k C(q_i,a_i)  \prod_{m=1}^M \sum_{d_m =1}^{h}  e\left(\sum_{i\in \mathcal{S}_m } \frac{a_i}{q_i} d_m \right), 
	\end{align}
	where the first sum is over partitions $\mathcal{P}= \lbrace \mathcal{S}_1, \dots, \mathcal{S}_M\rbrace$ of $\lbrace 1, \dots , k\rbrace$, and $w(\mathcal{P})$ is defined in~\cite[p.~17]{MS}.
	
	In the case of a partition $\mathcal{P}$ containing at least one part of size $\geq3$, 
	write $\mathcal{N}_1 = \bigcup_{\lvert \mathcal{S}_m\rvert =1} \mathcal{S}_m$, $\mathcal{N}_2=\lbrace 1, \dots , k\rbrace\smallsetminus \mathcal{N}_1$    and $m_2= \lvert \lbrace 1\leq m\leq M : \lvert \mathcal{S}_m\rvert \geq 2  \rbrace \rvert$.
	Using~\eqref{def E F} and $\lvert C(q,x)\rvert \leq 2$, we have
	\begin{align*}
		\Big\lvert \sum_{\substack{a_1, \dots, a_k \\ 1\leq a_i \leq q_i,  (q_i,a_i) =1 \\ \sum_{i =1}^{k} \frac{a_i}{q_i} \in \Z }} \prod_{i =1}^k C(q_i,a_i)  \prod_{m=1}^M \sum_{d_m =1}^{h}  e\left(\sum_{i\in \mathcal{S}_m } \frac{a_i}{q_i} d_m \right) \Big\rvert \leq 2^k h^{m_2}  \sum_{\substack{a_1, \dots, a_k \\ 1\leq a_i \leq q_i,  (q_i,a_i) =1 \\ \sum_{i =1}^{k} \frac{a_i}{q_i} \in \Z }} \prod_{i \in \mathcal{N}_1} F_h(\tfrac{a_i}{q_i}).
	\end{align*}
	Then we apply Lemma~\ref{Lem MV basic ineq} and the bound~\eqref{bound sum F squared} to obtain that the sum above is
	\begin{align*}
		&\leq \frac{2^k h^{m_2}}{[q_1,\dots,q_k]} \prod_{i \in \mathcal{N}_1}\Big( q_i \sum_{\substack{ 1\leq a_i \leq q_i \\  (q_i,a_i) =1  }}  \big(F_h(\tfrac{a_i}{q_i})\big)^2 \Big)^{\frac12}  \prod_{i \in \mathcal{N}_2}\Big( q_i \sum_{\substack{ 1\leq a_i \leq q_i \\  (q_i,a_i) =1  }} 1^2 \Big)^{\frac12} \\
		&\leq  2^k h^{\frac{k-1}2}\frac{ q_1\cdot\dots\cdot q_k}{[q_1,\dots,q_k]},
	\end{align*}
	where we used $\tfrac{1}{2}\lvert \mathcal{N}_1\rvert + m_2 \leq \tfrac{k-1}{2}$ when 
	 the partition $\mathcal{P}$ contains at least one part of size $\geq3$. 
Replacing this bound in~\eqref{analogue-52} and then in \eqref{constant-with-ET}, we sum over $1< q_1,\dots, q_k \mid P_y^N$ as in~\eqref{constant-with-ET} and  use the bound~\eqref{eq bound on sum with lcm} to obtain that the contribution of the partitions containing at least one part of size $\geq 3$ 
 in $\sum_{\substack{ \mathcal{D} \subseteq[1,h] \\ \lvert\mathcal{D}\rvert = k}}\mathfrak{S}_0(\mathcal{D})$  is at most $O_{k,\varepsilon}(h^{\frac{k-1}2 + \varepsilon})$.
 
We now turn our attention to partitions of $\lbrace 1, \dots , k\rbrace$ with sets of size at most $2$. 
	The combinatorics leading to \cite[(56)]{MS} work similarly and give	
	\begin{equation} \label{def-main-sum}
\begin{split}
		\sum_{\substack{ \mathcal{D} \subseteq[1,h] \\ \lvert\mathcal{D}\rvert = k}}\mathfrak{S}_0(\mathcal{D}) 
		=& \sum_{0 \leq j \leq \frac{k}{2}} (-1)^{j}\binom{k}{2j}\frac{(2j)!}{j! 2^j}
		\sum_{r_1, \dots, r_j \mid P_y^{N}}
		\sum_{\substack{b_1, \dots, b_j \\ 1\leq b_i \leq r_i,  (r_i,b_i) =1  }} \prod_{i =1}^j H\big(\tfrac{b_i}{r_i}\big) \\
		&\times	\sum_{\substack{q_{2j+1}, \dots, q_k \mid P_y^{N} \\ q_i >1}}
		\sum_{\substack{a_{2j+1}, \dots, a_k \\ 1\leq a_i \leq q_i,  (q_i,a_i) =1 \\ \sum_{i =1}^{j} \frac{b_i}{r_i}+ \sum_{i =2j+1}^{k} \frac{a_i}{q_i} \in \Z }} \prod_{i =2j+1}^k \frac{\lambda_2(q_i)C(q_i,a_i)}{q_i}  \sum_{d_i =1}^{h}  e\left( \frac{a_i}{q_i}  d_i \right)
		\\	&+ O_k(h^{\frac{k-1}2 + \varepsilon})
	\end{split} \end{equation}
	where $$H\big(\tfrac{b}{r}\big) = \sum_{\substack{q_1, q_2 \mid P_y^{N} \\ q_i >1}}
	\sum_{\substack{a_1, a_2 \\ 1\leq a_i \leq q_i,  (q_i,a_i) =1 \\  \frac{a_1}{q_1} + \frac{a_2}{q_2} \in \frac{b}{r} +\Z }}  \frac{\lambda_2(q_1)C(q_1,a_1)\lambda_2(q_2)C(q_2,a_2)}{q_1 q_2}    \sum_{d =1}^{h}  e\left(\frac{b}{r} d \right).$$
	In particular, we have
	\begin{align*}
		H(1) = \sum_{1<q\mid P_y^{N}}
		\sum_{\substack{ 1\leq a \leq q \\ (q,a) =1  }}  \frac{C(q,a)C(q,q-a)}{q^2}    h  
		= h  \sum_{1<q\mid P_y^{N}}
		\frac{C(q)\phi(q)}{q^2} ,
	\end{align*}
	and the contribution of the terms with all $r_i = 1$ in the sum above is 
	\begin{align*}
		\sum_{0 \leq j \leq \frac{k}{2}} \binom{k}{2j}\frac{(2j)!}{j! 2^j}
		\Big(- h  \sum_{1<q\mid P_y^{N}}	\frac{C(q)\phi(q)}{q^2}  \Big)^j  V_{k-2j}(y,N,h).
	\end{align*}
	We now show that the contribution to \eqref{def-main-sum} of the terms
	where not all $r_i$ are 1 can be absorbed in the error term.
	Let $\ell$ be the number of $i$'s for which $r_i>1$.  
	For any~$\ell >0$, and any $r_1, \dots , r_{\ell}, q_{2j+1}, \dots,  q_k>1$ up to re-ordering and applying Lemma~\ref{Lem MV basic ineq}, we have
	\begin{equation}  \label{bound HF with basic ineq} \begin{split}
			\sum_{\substack{b_1, \dots, b_{\ell} \\ 1\leq b_i \leq r_i,  (r_i,b_i) =1  }} \prod_{i =1}^{\ell} H\big(\frac{b_i}{r_i}\big) 
		\sum_{\substack{a_{2j+1}, \dots, a_k \\ 1\leq a_i \leq q_i,  (q_i,a_i) =1 \\ \sum_{i =1}^{j} \frac{b_i}{r_i}+ \sum_{i =2j+1}^{k} \frac{a_i}{q_i} \in \Z }} \prod_{i =2j+1}^k \frac{\lambda_2(q_i)C(q_i,a_i)}{q_i}  \sum_{d_i =1}^{h}  e\left( \frac{a_i}{q_i}  d_i \right)  \\
		\ll_k 
		\frac{1}{[r_1, \dots, r_{\ell},q_{2j+1}, \dots, q_k]}\prod_{i=1}^{\ell} \Big( r_i \sum_{ 1\leq b\leq r_i,  (r_i,b) =1  } \left\lvert H\big(\frac{b}{r_i}\big) \right\rvert^2 \Big)^{\frac12}  \\
		 \times\prod_{i=2j+1}^{k} \Big( \frac{1}{q_i} \sum_{ 1\leq a\leq q_i,  (q_i,a) =1  } \left\lvert F_h\big(\frac{a}{q_i}\big) \right\rvert^2 \Big)^{\frac12}.  
	\end{split} \end{equation}
	To obtain a bound for $\left\lvert H\big(\frac{b}{r}\big) \right\rvert$ we proceed similarly to \cite{MS} which gives
	\begin{align*}
		H\big(\tfrac{b}{r}\big) &\ll F_h\big(\tfrac{b}{r}\big) \sum_{\substack{s_1, s_2 \mid P_y^{N} \\ [s_1,s_2] = r}}
		\sum_{\substack{c_1, c_2 \\ 1\leq c_i \leq s_i,  (s_i,c_i) =1 \\  \frac{c_1}{s_1} + \frac{c_2}{s_2} \in \frac{b}{r} +\Z }} \frac{1}{s_1 s_2} \sum_{\substack{t \mid P_y^{N} \\ (t,r)=1}}  \frac{\phi(t)}{t^2} \\
		&\ll  F_h\big(\tfrac{b}{r}\big) \frac{1}{\phi(r)}\sum_{\substack{s_1, s_2 \mid P_y^{N} \\ [s_1,s_2] = r}}
		\frac{\phi(s_1)\phi(s_2)}{s_1 s_2} \prod_{\substack{p\leq y \\ p\nmid r}} \Big(1 + \frac{1}{1-p^{-1}} \sum_{n = 1}^{N} p^{-n}\Big) \\
		& \ll F_h\big(\tfrac{b}{r}\big) \frac{1}{r} \prod_{p\mid r}(1-p^{-1})^2(1+ v_p(r) - v_p(r)p^{-1})   \prod_{\substack{p\leq y \\ p\nmid r}} \Big(1 + p^{-1}\Big) \\
		& \ll F_h\big(\tfrac{b}{r}\big) \frac{d(r)}{r} \log y,
	\end{align*}
	where $d(r)$ is the number of divisors of $r$.
	Using ~\eqref{bound sum F squared}, this gives
	$$\sum_{\substack{ 1\leq b \leq r \\  (r,b) =1  }}  \big\lvert H(\tfrac{b}{r})\big\rvert^2 \leq \min(r,h) \frac{d(r)^2}{r} (\log y)^2.$$
	Using this bound and~\eqref{bound sum F squared} in~\eqref{bound HF with basic ineq} summed over all $r_1, \dots , r_{\ell}, q_{2j+1}, \dots,  q_k>1$ divisors of $P_y^{N}$, we obtain
	\begin{align*}
		&\sum_{\substack{r_1, \dots, r_{\ell} \mid P_y^{N} \\ r_i >1}}
		\sum_{\substack{b_1, \dots, b_{\ell} \\ 1\leq b_i \leq r_i,  (r_i,b_i) =1  }} \prod_{i =1}^{\ell} H\big(\tfrac{b_i}{r_i}\big) 
		\sum_{\substack{q_{2j+1}, \dots, q_k \mid P_y^{N} \\ q_i >1}}
		\sum_{\substack{a_{2j+1}, \dots, a_k \\ 1\leq a_i \leq q_i,  (q_i,a_i) =1 \\ \sum_{i =1}^{j} \frac{y_i}{r_i}+ \sum_{i =2j+1}^{k} \frac{a_i}{q_i} \in \Z }} \prod_{i =2j+1}^k \frac{\lambda_2(q_i)C(q_i,a_i)}{q_i}  \sum_{d_i =1}^{h}  e\left( \frac{a_i}{q_i}  d_i \right)  \\
		\ll_k& \sum_{\substack{r_1, \dots, r_{\ell} \mid P_y^{N} \\ r_i >1}}\sum_{\substack{q_{2j+1}, \dots, q_k \mid P_y^{N} \\ q_i >1}}
		\frac{1}{[r_1, \dots, r_{\ell},q_{2j+1}, \dots, q_k]}\prod_{i=1}^{\ell} \Big(h^{\frac12}    d(r_i) \log y \Big)    h^{\frac{k-2j}{2}}  \\
		\ll_k& h^{\frac{k+ \ell -2j}{2}} (\log y)^{\ell} \sum_{m \mid P_y^{N}}\frac{1}{m} \Big(\sum_{r\mid m}d(r)\Big)^{\ell} \Big(\sum_{q\mid m}1\Big)^{k-2j} \\
		\ll_k& h^{\frac{k+ \ell -2j}{2}} (\log y)^{\ell} \prod_{p\mid P_y}\Big( 1 + \sum_{n=1}^{N}(n+1)^{k+\ell -2j}\big(\frac{n+2}{2}\big)^{\ell} p^{-n} \Big) 
		\ll_{k,\ell,j,\varepsilon} h^{\frac{k+ \ell -2j}{2}} y^{\varepsilon}.
	\end{align*}
	Finally, summing the contribution for each $\ell\geq 0$ yields
	\begin{align*}
		\sum_{\substack{ \mathcal{D} \subseteq[1,h] \\ \lvert\mathcal{D}\rvert = k}} \mathfrak{S}_0(\mathcal{D}) 
		&= \sum_{j=0}^{k/2}  \binom{k}{2j} \frac{(2j)!}{j! 2^j}\left(- h \sum_{1< d \mid P_y^{N} } \frac{C(d) \phi(d)}{d^2} \right)^j  V_{k-2j}(y,N,h)  \\
		&+	\sum_{j=0}^{k/2}  \binom{k}{2j} \frac{(2j)!}{j! 2^j}\sum_{\ell = 1}^j\binom{j}{\ell}\left( h \sum_{1< d \mid P_y^{N} } \frac{C(d) \phi(d)}{d^2} \right)^{j-\ell}  O\big(h^{\frac{k+ \ell -2j}{2}} y^{\varepsilon}\big)  + O_k(h^{(k-1 +\varepsilon)/2}),
	\end{align*}
	and using
$
		\sum_{1< d \mid P_y^{N} } \frac{C(d) \phi(d)}{d^2} \ll_{\varepsilon} y^{\varepsilon},
$
 we deduce
	\begin{align*}
		\sum_{\substack{ \mathcal{D} \subseteq[1,h] \\ \lvert\mathcal{D}\rvert = k}} \mathfrak{S}_0(\mathcal{D}) 
		&= \sum_{j=0}^{k/2}  \binom{k}{2j} \frac{(2j)!}{j! 2^j}\left(- h \sum_{1< d \mid P_y^{N} } \frac{C(d) \phi(d)}{d^2} \right)^j  V_{k-2j}(y,N,h)    + O_{k,\varepsilon}(h^{\frac{k-1}{2}}y^{\varepsilon}).
	\end{align*}
	which completes the proof of Lemma~\ref{Lem analogue of MS60}.
\end{proof}

The proof of Theorem~\ref{Prop bound for D with large sets} is now relatively straightforward.
Lemma~\ref{Lem analogue of MS60} gives
	for any $h>k \in {\N}$, and $N\geq 4(k+1)\log h$
that
\begin{align*}
	\sum_{\substack{ \mathcal{D} \subseteq[1,h] \\ \lvert\mathcal{D}\rvert = k}} \mathfrak{S}_0(\mathcal{D}) 
	&= \sum_{j=0}^{k/2}  \binom{k}{2j} \frac{(2j)!}{j! 2^j}\left(- h \sum_{1< d \mid P_y^{N} } \frac{C(d) \phi(d)}{d^2} \right)^j  V_{k-2j}(h^{k+1},N,h)    + O_{k,\varepsilon}(h^{\frac{k-1}{2}+\varepsilon}).
\end{align*}
Then the bound from Corollary~\ref{Lem bound V} yields
\begin{align*}
	\sum_{\substack{ \mathcal{D} \subseteq[1,h] \\ \lvert\mathcal{D}\rvert = k}} \mathfrak{S}_0(\mathcal{D}) 
	&\ll_{k,\varepsilon} \sum_{j=0}^{k/2}  \binom{k}{2j} \frac{(2j)!}{j! 2^j}\left( h \sum_{1< d \mid P_y^{N} } \frac{C(d) \phi(d)}{d^2} \right)^j  h^{\frac{k-2j}{2} + \varepsilon}    + O_{k,\varepsilon}(h^{\frac{k-1}{2}+\varepsilon}) \\
	&\ll_k h^{\frac{k}{2} + \varepsilon} \Big( \prod_{p<y} 1-\frac{1}{p} \Big)^{\frac{-k}{2}}
	\ll_{k,\varepsilon} h^{\frac{k}{2} + 2\varepsilon},
\end{align*}
which finishes the proof of Theorem ~\ref{Prop bound for D with large sets}.

\section{Integral form and improved error terms}
\label{section-improved-averages}

Using the methods of  Section~\ref{proof-prop-section} and Theorem~\ref{Prop bound for D with large sets}, we can obtain a more precise form of the averages of the Hardy--Littlewood constants for \sos of \cite[Theorem~1.1]{smilansky} 
and \cite[Proposition~1.3]{FKR} (in a special case) by exhibiting a secondary term. In order to see the secondary term, we need to express the results of Section~\ref{proof-prop-section} differently, as a closed-form expression which contains implicitly all the descending powers of $\log{h}$.  
We first prove that we can write such an asymptotic for the number of sums of two squares, with a square-root cancellation error term under the Riemann Hypothesis (Theorem~\ref{thm-S0S-integral}). The argument for the proof of Theorem~\ref{thm-S0S-integral} is essentially due to Selberg and known to experts, it appeared as a mathoverflow post \cite{lucia}, and an exercise in the book of Koukoulopoulos \cite[Exercise~13.7]{koukoulopoulos}. Note also the independent analogue result of Gorodetsky and Rodgers \cite[Theorem~B.1]{GR20} inspired by \cite{Ram76}.
With the same techniques, we then  prove Proposition~\ref{prop-lucia}, which exhibits the secondary term for the average of the Hardy-Littlewood constants for 2-tuples of sums of two squares. The general case is Proposition~\ref{intro-FKR-improved} and it follows by using Theorem~\ref{Prop bound for D with large sets} to show that the average over $k$-tuples reduces to the average over 2-tuples.

\begin{proof}[Proof of Theorem~\ref{thm-S0S-integral}:]
We first assume the Riemann Hypothesis. Using Perron's formula, we have for any $\delta>0$
\begin{align}
\label{after-perron} 
\sum_{n \leq x} 1_\mathbf{E}(n) = \int_{1+\delta-iT}^{1+\delta+iT} F(s) \frac{x^s}{s}ds + {O}
\bigg(\frac{x^{1+\delta} \log{x}}{T }\bigg),
\end{align}
where  $F(s) = \zeta^{1/2}(s) L(s, \chi_4)^{1/2} (1 - 2^{-s})^{-1/2} \prod_{p \equiv 3 \mod 4} \left( 1 - p^{-2s} \right)^{-1/2}$ as seen in the proof of Theorem~\ref{thm-SOS-AP}.
The above path integral is part of a contour which enclose a region of analyticity of the integrand, which is the usual 
countour  going from $1+\delta - iT$ to $1+\delta + iT$ then to $1/2+\varepsilon +iT$ then to $1/2+\varepsilon -iT$ and then back to $1+\delta - iT$ with a slit along the real axis between $1/2+\varepsilon$ and 1, with a line just above the real axis from
	$1/2+\varepsilon$ to $1$, 
	and a line just below the real axis from 1 to $1/2 + \varepsilon$. More precisely, for any $\varepsilon, \eta > 0$ and for $0 < \kappa < \delta$, 
we define line segments $L_j, \, j=1,2,...,7$ as in Figure~\ref{fig:cont2}.

Together with the line segment $1+\delta +iT \to 1+\delta -iT$ of the integral~\eqref{after-perron}, this gives the closed contour of Figure~\ref{fig:cont2}, which enclose a region of analyticity of the function $F(s) = G(s)(s-1)^{-1/2}$ since we are assuming the Riemann Hypothesis and $\zeta^{1/2}(s)(s-1)^{1/2}$, $L(s, \chi_4)^{1/2} (1-2^{-s})^{-1/2} \prod_{p \equiv 3 \mod 4} (1-p^{-2s})^{-1/2}$ are analytic for $\re(s) > 1/2 + \varepsilon.$ 
Then, using Cauchy's theorem, we have
$$\int_{1+\delta-iT}^{1+\delta+iT} F(s) \frac{x^s}{s}ds = \sum_{j=1}^7 \int_{L_j} F(s) \frac{x^s}{s}ds .$$
The contribution coming from $L_1,L_2,L_4,L_6,L_7$ are bounded by the classical estimates, where we use the Lindel\"of Hypothesis to  bound 
$|\zeta^{1/2}(\sigma+it)|, |L^{1/2}(\sigma+it, \chi_4)| \ll_\sigma |t|^{\varepsilon_1}$ for $1/2 < \sigma < 1$ and $\varepsilon_1 >0$.
For the horizontal integral over $L_1$, we have
$$\int_{L_1} F(s) \frac{x^s}{s}ds \ll \int_{1/2+\varepsilon}^{1+\delta}  \frac{x^{\sigma}}{T^{1-2\varepsilon_1}}  d \sigma = {O}\bigg(\frac{x^{1+\delta}}{T^{1-2\varepsilon_1}} \bigg),$$
where we also used the fact that the Euler product $(1-2^{-s})^{-1/2} \prod_{p \equiv 3 \mod 4} (1-p^{-2s})^{-1/2}$ is absolutely bounded for $\re(s)> 1/2+\varepsilon$. We get the same bound for $\int_{L_7}$.

    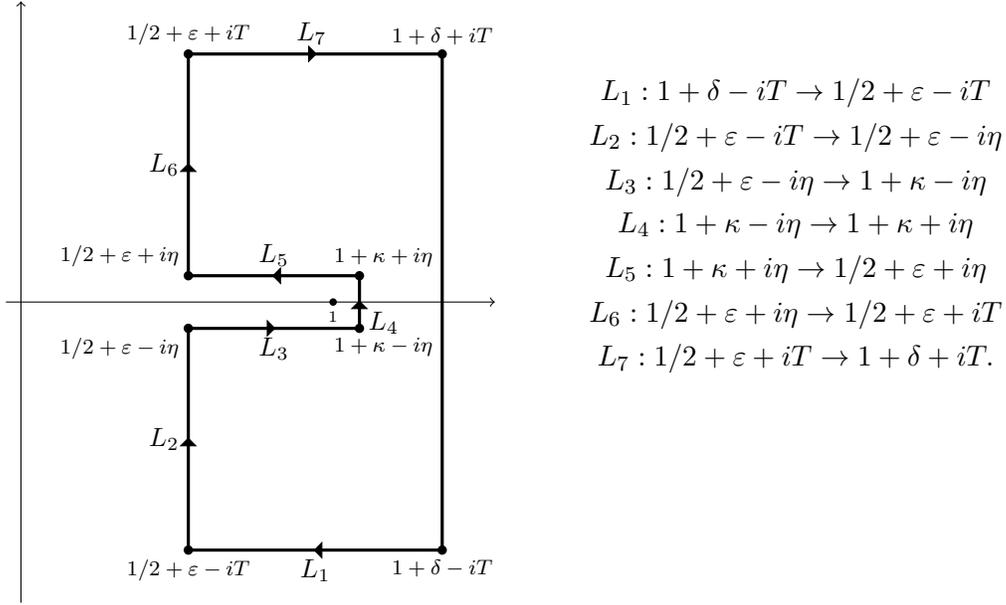
\begin{figure}
\centering
\hspace*{1cm}\begin{tikzpicture}

\draw [->](-0.5,0) -- (6.0,0);
\draw [<-] (-0.3,4.0) -- (-0.3,-4.0);
\begin{scope}[very thick, every node/.style={sloped,allow upside down}]
\draw  (5.3,3.3) -- node[scale=0.8] {} (5.3,-3.3)-- node[scale=0.8] {\midarrow} (1.925,-3.3) --node[scale=0.8] {\midarrow} (1.925,-0.35) --node[scale=0.8] {\midarrow} (4.2,-0.35) -- node[scale=0.8] {\midarrow}(4.2,0.35) --node[scale=0.8] {\midarrow} (1.925,0.35) --node[scale=0.8] {\midarrow} (1.925,3.3) --node[scale=0.8] {\midarrow} (5.3,3.3);
\end{scope}
\draw  (5.3,3.3) -- (5.3,-3.3) -- node[below] {$L_1$} (1.925,-3.3) --node[left] {$L_2$} (1.925,-0.35) --node[below] {$L_3$} (4.2,-0.35) -- node[below right] {$L_4$} (4.2,0.35) --node[above] {$L_5$} (1.925,0.35) --node[left] {$L_6$} (1.925,3.3) --node[above] {$L_7$} (5.2,3.3);

 \coordinate[label = above:\footnotesize$1+\delta +iT$] (A) at (5.3,3.3);
 \coordinate[label = below:\footnotesize$1+\delta -iT$] (B) at (5.3,-3.3);
 \coordinate[label = below:\footnotesize$1/2+\varepsilon -iT$] (C) at (1.925,-3.3);
 \coordinate[label = below left:\footnotesize$1/2 +\varepsilon -i \eta$] (D) at (1.925,-0.35);
 \coordinate[label = below:\footnotesize$\hspace{18pt} 1+\kappa -i\eta$] (E) at (4.2,-0.35);
 \coordinate[label = above:\footnotesize$\hspace{18pt}1+\kappa +i\eta$] (F) at (4.2,0.35);
 \coordinate[label = above left:\footnotesize$1/2 +\varepsilon +i \eta$] (G) at (1.925,0.35);
 \coordinate[label = above:\footnotesize$1/2 +\varepsilon +i T$] (H) at (1.925,3.3);
 \coordinate[label = below:\tiny$1$] (I) at (3.85,0);

 \node at (A)[circle,fill,inner sep=1.25pt]{};
  \node at (B)[circle,fill,inner sep=1.25pt]{};
   \node at (C)[circle,fill,inner sep=1.25pt]{};
   \node at (D)[circle,fill,inner sep=1.25pt]{};
   \node at (E)[circle,fill,inner sep=1.25pt]{};
\node at (F)[circle,fill,inner sep=1.25pt]{};
\node at (G)[circle,fill,inner sep=1.25pt]{};
\node at (H)[circle,fill,inner sep=1.25pt]{};
\node at (I)[circle,fill,inner sep=1pt]{};
\node at (10.0,1.0) {$\large\begin{minipage}{\textwidth}
            \begin{equation*}
            \begin{gathered}
             L_1 : 1+\delta -iT \to 1/2+\varepsilon -iT\\
    L_2 : 1/2+\varepsilon -iT \to 1/2+\varepsilon -i \eta\\
    L_3 : 1/2+\varepsilon -i \eta \to 1+\kappa -i \eta\\
    L_4 : 1+\kappa -i \eta \to 1+\kappa +i \eta\\
    L_5 : 1+\kappa +i \eta \to 1/2 + \varepsilon + i \eta\\
    L_6 : 1/2 + \varepsilon + i \eta \to 1/2 + \varepsilon + iT\\
    L_7 : 1/2 + \varepsilon + iT \to 1+\delta +iT .
    \end{gathered}
            \end{equation*}
       \end{minipage}$
  };
\end{tikzpicture}
 \caption{The contour used in the proof of Theorem 2.4.}
  \label{fig:cont2}
\end{figure}

For the vertical integral over $L_2$, 
we have
$$\int_{L_2} F(s)\frac{x^s}{s}ds \ll \int_{\eta}^T \frac{x^{1/2+\varepsilon}}{(t+ \frac12)^{1-2\varepsilon_1}} dt = {O}\bigg( x^{1/2+\varepsilon} {T^{2\varepsilon_1}} \bigg),$$
which also hold for $\int_{L_6}$.
Finally, we have 
$$\int_{L_4} F(s) \frac{x^s}{s}ds \ll  \eta \: x^{1+\kappa},$$
and choosing $T=x^{1/2}$ and $\eta< x^{-\frac12 - \kappa}$, this gives
$$\int_{1+\delta-iT}^{1+\delta+iT} F(s) \frac{x^s}{s}ds = \lim_{\eta \to 0^+} \bigg(\int_{1/2+\varepsilon-i \eta}^{1+\kappa-i \eta} - \int_{1/2+\varepsilon+i\eta}^{1+\kappa +i\eta} \bigg) F(s)\frac{x^s}{s}ds + {O}\bigg( x^{1/2+\varepsilon} \bigg).$$
Note that $\kappa$ can be arbitrarily small, and choosing for example $\kappa = x^{-2}$, we have
$$\lim_{\eta \to 0^+} \bigg(\int_{1/2+\varepsilon-i \eta}^{1+\kappa-i \eta} - \int_{1/2+\varepsilon+i\eta}^{1+\kappa +i\eta} \bigg) F(s) \frac{x^s}{s}ds = \lim_{\eta \to 0^+} \bigg(\int_{1/2+\varepsilon-i \eta}^{1-i \eta} - \int_{1/2+\varepsilon+i\eta}^{1+i\eta} \bigg) F(s) \frac{x^s}{s}ds +{O}(1).$$
Putting everything together, we have
$$\sum_{n \leq x} 1_\mathbf{E}(n) = \frac{1}{2 \pi i}\int_{1/2 + \varepsilon}^{1} G(\sigma)\frac{x^\sigma}{\sigma} \lim_{\eta \to 0^+}\Big( (\sigma - i \eta -1)^{-1/2}- (\sigma + i \eta -1)^{-1/2} \Big) d \sigma + {O}\bigg( x^{1/2+\varepsilon} \bigg),$$
where $G(s) = (s-1)^{1/2}F(s)$.

We use the fact that when $\sigma \in (0,1)$, $\log(\sigma \pm i \eta -1) \sim \log |\sigma -1| \pm i \pi$ as $\eta \to 0^+$. Writing $(\sigma \pm i \eta -1)^{-1/2} = \exp{(-\frac{1}{2}\log (\sigma \pm i \eta -1))}$, we see that $(\sigma \pm i \eta -1)^{-1/2} \sim \mp i |\sigma -1|^{-1/2}$, and we have 
$$\lim_{\eta \to 0^+}\Big( (\sigma - i \eta -1)^{-1/2}- (\sigma + i \eta -1)^{-1/2}\Big) = 2i |\sigma -1|^{-1/2}.$$
Replacing above, this proves the theorem under the Riemann Hypothesis.
Unconditionally, we start from~\eqref{after-perron}, and we use a similar coutour,  but with $1/2+\varepsilon$ replaced by $1- c / \sqrt{\log{x}}$, where $c$ is small enough to insure that the coutour does not contain any zeroes of $\zeta(s)$ or $L(s, \chi_4)$. Working as above, we get 
$$\sum_{n \leq x} 1_\mathbf{E}(n) = \frac{1}{\pi}\int_{1-\frac{c}{\sqrt{\log{x}}}}^{1} \frac{x^\sigma}{\sigma} G(\sigma)|\sigma-1|^{-1/2} d \sigma + O \left(
\frac{x^{1+\delta}}{T^{1-2\varepsilon_1}} + x \exp{\left( - c \sqrt{\log{x}} \right)} T^{2\varepsilon_1} \right),
$$ 
and choosing $\delta = 1/\log{x}$ and $T=\exp(c\sqrt{\log{x}})$, we get
$$\sum_{n \leq x} 1_\mathbf{E}(n) = \frac{1}{\pi}\int_{1-c/\sqrt{\log{x}}}^{1} \frac{x^\sigma}{\sigma} G(\sigma)|\sigma-1|^{-1/2} d \sigma + O \left(
x \exp(- c_0 \sqrt{\log{x}} ) \right),$$ for some $c_0 > 0$.
Finally, we have
	\begin{align*}
		\int_{\frac12 + \varepsilon}^{1-c/\sqrt{\log{x}}}  \frac{x^\sigma}{\sigma} G(\sigma)|\sigma-1|^{-1/2}
		 d\sigma
		&\ll 	\int_{\frac12 + \varepsilon}^{1-c/\sqrt{\log{x}}} \frac{ x^\sigma}{\sigma |\sigma-1|^{1/2}} d\sigma 
		\ll x^{1-c/\sqrt{\log{x}}}
	\end{align*}
	which shows the unconditional result.
\end{proof}

The following proposition is a more precise version of \cite[Theorem~1.1]{smilansky} who showed that
$$
\sum_{\substack{1\leq d_1, d_2 \leq H \\ \text{distinct}}} \mathfrak{S}(\lbrace d_1, d_k \rbrace) = H^2 + O ( H^{1+\varepsilon} ).
$$
We remark that our normalization differs from~\cite{smilansky} for the singular series.

\begin{proposition} \label{prop-lucia}
	Fix $\varepsilon > 0$. There exists $c>0$ such that
	$$\sum_{\substack{1\leq d_1, d_2 \leq H \\ \text{distinct}}} \mathfrak{S}(\lbrace d_1, d_k \rbrace) = H^2 +
	\frac{2}{\pi K^2}  \int_{1/2+\varepsilon}^1 \frac{F'(\sigma) H^\sigma + F(\sigma) H^\sigma \log{H}}{|\sigma-1|^{1/2}} d\sigma + 
	O \left(H \exp(-c \sqrt{\log{H}} ) \right)
	$$
	where $F(s)=\zeta(s-1) M(s-1) \left[ (s-1) \zeta(s) \right]^{1/2} s^{-1}$, with $M(s)$ as defined by~\eqref{def-M}.
		Assuming the Riemann Hypothesis, we can replace the error term by
	$O \left( H^{1/2+\varepsilon} \right)$.
\end{proposition}

\begin{proof}
	As in \cite[\S~2.3]{smilansky}, we have 
	\begin{align*}
		\sum_{\substack{1\leq d_1, d_2 \leq H \\ \text{distinct}}} \mathfrak{S}(\lbrace d_1, d_k \rbrace)
		&= 2\sum_{1\leq d < H} \mathfrak{S}(\lbrace 0,d \rbrace)(H-d)\\
		&= \frac{1}{K^2}\frac{1}{2i\pi}\int_{(2)} \frac{D(s)}{s(s+1)}H^{s+1} ds,
	\end{align*}
	where $D(s) =\zeta(s) \zeta(s+1)^{\frac12} M(s)$  as defined in the beginning of section~\ref{proof-prop-section}. 
	As~\cite{smilansky}, we compute the main term, coming from the pole of $D(s)$ at $s=1$, which gives
	\begin{align*}
		\sum_{\substack{1\leq d_1, d_2 \leq H \\ \text{distinct}}} \mathfrak{S}(\lbrace d_1, d_k \rbrace)
		&= H^2 + \frac{1}{K^2}\frac{1}{2i\pi}\int_{(\varepsilon)} \frac{D(s)}{s(s+1)}H^{s+1} ds .
	\end{align*}
	We first assume the Riemann hypothesis and we evaluate the integral
	$$
	\frac{1}{2i\pi}\int_{(\varepsilon)} \frac{D(s)}{s(s+1)}H^{s+1} ds  = \frac{1}{2i\pi}\int_{(1+\varepsilon)} \frac{F(s)}{(s-1)^{3/2}}H^{s} ds  
	$$
	where $F(s)=\zeta(s-1) M(s-1) \left[ (s-1) \zeta(s) \right]^{1/2} s^{-1}$ is analytic for $\re(s) > 1/2+\varepsilon$. 
	We begin with an integration by part to obtain
	\begin{align*}
		\int_{(1+\varepsilon)} \frac{F(s)}{(s-1)^{3/2}}H^{s} ds 
		&= \lim_{T\rightarrow\infty} [-2F(s) H^s (s-1)^{-1/2}]_{1+\varepsilon - iT}^{1+\varepsilon + iT}  + 2\int_{(1+\varepsilon)} \frac{F'(s) H^s + F(s) H^s \log{H}}{(s-1)^{1/2}} ds \\
		&= 2\int_{(1+\varepsilon)} \frac{F'(s) H^s + F(s) H^s \log{H}}{(s-1)^{1/2}} ds.
	\end{align*} 
	To evaluate the last integral, we first approximate the line integral by the segment from $1+\varepsilon - iT$ to $1+\varepsilon + iT$, and use the coutour of Figure~\ref{fig:cont2}. Working as in the proof of Theorem~\ref{thm-S0S-integral}, we get 
	\begin{align*} \nonumber
		&\frac{2}{2 \pi i} \int_{(1+\varepsilon)} \frac{F'(s) H^s + F(s) H^s \log{H}}{(s-1)^{1/2}} ds \\
		&=  \frac{1}{\pi i}  \int_{1/2+\varepsilon}^1 \left( F'(\sigma) H^\sigma + F(\sigma) H^\sigma \log{H} \right) 
		\lim_{\eta \rightarrow 0+} \left( (\sigma  - 1 - i \eta)^{-1/2} - (\sigma - 1 + i \eta)^{-1/2} \right) d\sigma + O \left( H^{1/2+\varepsilon} \right) \\
		&=  \frac{2}{\pi}  \int_{1/2+\varepsilon}^1 \left( F'(\sigma) H^\sigma + F(\sigma) H^\sigma \log{H} \right) 
	|\sigma-1|^{-1/2} d\sigma + O \left( H^{1/2+\varepsilon} \right).
	\end{align*}
	Replacing above, this gives (under the Riemann Hypothesis)
	\begin{align*}
		\frac{1}{K^2}\frac{1}{2\pi i}\int_{(\varepsilon)} \frac{D(s)}{s(s+1)}H^{s+1} ds = \frac{2}{\pi K^2}  \int_{1/2+\varepsilon}^1 \frac{F'(\sigma) H^\sigma + F(\sigma) H^\sigma \log{H}}{|\sigma-1|^{1/2}}  d\sigma + 
		O \left( H^{1/2+\varepsilon} \right).
	\end{align*}
	To do a proof without the Riemann Hypothesis, we proceed as in the proof of Theorem~\ref{thm-S0S-integral}, 
	 and we get
	\begin{align*}
		&\frac{2}{2 \pi i} \int_{(1+\varepsilon)} \frac{F'(s) H^s + F(s) H^s \log{H}}{(s-1)^{1/2}} ds
		&= \frac{2}{\pi}  \int_{1-c_1/\sqrt{\log{H}}}^1 \frac{F'(\sigma) H^\sigma + F(\sigma) H^\sigma \log{H}}{|\sigma-1|^{1/2}} 
		d\sigma + O \left( H \exp{\left( -c \sqrt{\log{H}}\right)} \right).
	\end{align*}
	To conclude the proof, we show that
	$$
	\int_{1-c_1/\sqrt{\log{H}}}^1 \frac{F'(\sigma) H^\sigma + F(\sigma) H^\sigma \log{H}}{|s-1|^{1/2}} 
	d\sigma =  \int_{1/2+\varepsilon}^1 \frac{F'(\sigma) H^\sigma + F(\sigma) H^\sigma \log{H}}{|s-1|^{1/2}} 
	d\sigma   + O \left( H \exp{\left( -c \sqrt{\log{H}}\right)} \right).
	$$
	This follows from the fact that $\zeta$ does not vanish on $[\tfrac{1}{2} +\varepsilon,1]$, so $F$ and $F'$ are defined and continuous on $[\tfrac{1}{2} +\varepsilon,1]$, in particular, they are uniformly bounded.
	We have
	\begin{align*}
		\int_{\frac12}^{1-c_1/\sqrt{\log{H}}} \frac{F'(\sigma) H^\sigma + F(\sigma) H^\sigma \log{H}}{|\sigma-1|^{1/2}} d\sigma
		&\ll_F  	\int_{\frac12}^{1-c_1/\sqrt{\log{H}}} \frac{ H^\sigma \log{H}}{|\sigma-1|^{1/2}} d\sigma \\
		&\ll H^{1-c_1/\sqrt{\log{H}}}(\log H)^{\frac54} \ll_c H \exp{\left( -c \sqrt{\log{H}}\right)}
	\end{align*}
	for any $c < c_1$.
\end{proof}

We can now prove Proposition~\ref{intro-FKR-improved}. We observe that
it is a more precise version of (a particular case of) \cite[Proposition~1.3]{FKR} who showed that
$$
\sum_{\substack{1\leq d_1, \dots, d_k \leq H \\ \text{distinct}}} \mathfrak{S}(\lbrace d_1, \dots, d_k \rbrace)
		= H^k + O \left( H^{k-2/3+o(1)} \right).
$$

\begin{proof}[Proof of Proposition~\ref{intro-FKR-improved}] 
Note that the cases $k=0$ or $1$ are easy.
We have $\mathfrak{S}(\emptyset) = \mathfrak{S}(\lbrace d\rbrace) =1$, so we obtain $1$ and $H$ respectively, without error term. The case $k=2$ is proven in Proposition~\ref{prop-lucia}.
	Similarly to \cite[(17)]{MS}, we have
	\begin{align*}
		\sum_{\substack{1\leq d_1, \dots, d_k \leq H \\ \text{distinct}}} \mathfrak{S}(\lbrace d_1, \dots, d_k \rbrace)
		&= \sum_{r=0}^{k}\binom{k}{r}\frac{(H-r)!}{(H-k)!}\sum_{\substack{1\leq d_1, \dots, d_r \leq H \\ \text{distinct}}} \mathfrak{S}_0(\lbrace d_1, \dots, d_r \rbrace) \\
		&= \frac{H!}{(H-k)!} + \binom{k}{2}\frac{(H-2)!}{(H-k)!}\sum_{\substack{1\leq d_1, d_2 \leq H \\ \text{distinct}}} \mathfrak{S}_0(\lbrace d_1, d_2 \rbrace) 
		+ O(H^{k - \frac32 + \varepsilon}),
	\end{align*}
	where we used the decomposition $\mathfrak{S}(\mathcal{H}) = \sum_{\mathcal{T} \subseteq\mathcal{H}} \mathfrak{S}_0(\mathcal{T})$, the fact that $\mathfrak{S}_0(\lbrace d \rbrace)=0$,  and the bound from Theorem~\ref{Prop bound for D with large sets} as soon as the size of the set is larger than $2$.
	Using the estimates
	\begin{align*}	\frac{H!}{(H-k)!} &= H(H-1)\dots(H-k+1) = H^k + H^{k-1}\sum_{i=1}^{k-1}(- i) + O_k(H^{k-2}) \\ & =  H^k - H^{k-1}\frac{k(k-1)}{2} + O_k(H^{k-2}),\\
		\frac{(H-2)!}{(H-k)!} &= H^{k-2}  + O_k(H^{k-3}),
	\end{align*}
	and Proposition~\ref{prop-lucia}, this gives
	\begin{align*}
		&\sum_{\substack{1\leq d_1, \dots, d_k \leq H \\ \text{distinct}}} \mathfrak{S}(\lbrace d_1, \dots, d_k \rbrace)\\
		&= H^k - \frac{k(k-1)}{2}{ H^{k-1}} +  \binom{k}{2}  H^{k-2}\sum_{\substack{1\leq d_1, d_2 \leq H \\ \text{distinct}}}\left(\mathfrak{S}(\lbrace d_1, d_2 \rbrace) - 1  \right) + O(H^{k - \frac32 + \varepsilon}) \\
		&= H^k - \frac{k(k-1)}{2}{ H^{k-1}} +  \binom{k}{2}  H^{k-2}\left(\frac{ 2 }{\pi K^2}   \int_{1/2+\varepsilon}^1 \frac{F'(\sigma) H^\sigma + F(\sigma) H^\sigma \log{H}}{|\sigma-1|^{1/2}} d\sigma +H \right) + O(H^{k - \frac32 + \varepsilon}) \\
		&= H^k +  k(k-1) \frac{{ H^{k-1}}}{\pi K^2}   \int_{1/2+\varepsilon}^1 \frac{F'(\sigma) H^{\sigma -1} + F(\sigma) H^{\sigma-1} \log{H}}{|\sigma-1|^{1/2}} d\sigma   + O(H^{k - \frac32 + \varepsilon}).
	\end{align*}
\end{proof}

We conclude this section by proving  Proposition~\ref{prop-lucia-2}. The proof is similar to the other proofs of this section, and we skip the details.

\begin{proof}[Proof of Proposition~\ref{prop-lucia-2}:] 
	Starting from~\eqref{before-hankel-D},
	we write
	$$S(H) =  H+ \frac{1}{2 K^2}\frac{1}{2 \pi i}\int_{(1+\varepsilon)} \frac{F(s)}{(s-1)^{\frac{3}{2}}}  H^{s-1} ds ,$$
	where $F(s)=\zeta(s-1) M(s-1) \left[ (s-1) \zeta(s) \right]^{1/2}  \Gamma(s)$, with $M(s)$ as defined by~\eqref{def-M}. 
Proceeding as in the proof of Proposition~\ref{prop-lucia}, with an integration by part before moving the contour of integration gives the following
	$$S(H) = H + 
	\frac{1}{ \pi  K^2}\int_{1/2+\varepsilon}^1 \frac{F'(\sigma) H^{\sigma-1} + F(\sigma) H^{\sigma-1} \log{H}}{|\sigma-1|^{1/2}} d\sigma + 
	O \left( \exp(-c \sqrt{\log{H}} ) \right).
	$$
	Similarly, using~\eqref{before-hankel-D0} and without the integration by part, we have
	\begin{align*}
	S(H, \chi_0) &=  \frac{\phi(q)}{q} H+ \frac{1}{2 K^2}\frac{1}{2 \pi i}\int_{(1+\varepsilon)} \frac{F_{\chi_0}(s)}{(s-1)^{\frac{1}{2}}}  H^{s-1}ds \\
	&= \frac{\phi(q)}{q} H + 
	\frac{1}{2  \pi K^2}\int_{1/2+\varepsilon}^1 \frac{F_{\chi_0}(\sigma) H^{\sigma-1}}{|\sigma-1|^{1/2}} d\sigma + 
	O \left(\exp(-c \sqrt{\log{H}} ) \right)	,
	\end{align*}
	where 
	\begin{align*}
	F_{\chi_0}(s) &=L(s-1, \chi_0) \Gamma(s-1) M_{\chi_0}(s-1) \left[ (s-1) L(s,\chi_0) \right]^{1/2} \\
	&=  \frac{1-q^{-(s-1)}}{s-1}   F(s) =: A_q(s) F(s)
	\end{align*} where we used $M_{\chi_0}(s)=(1-q^{-(s+1)})^{-1/2} M(s)$. 
	Assuming the Riemann Hypothesis, we can replace the error term by
	$O \left( H^{-1/2+\varepsilon} \right)$.
Then, we obtain the expressions in Proposition~\ref{prop-lucia-2} by using the orthogonality of characters and expression~\eqref{non-trivial} for the contribution of non-trivial characters as in the proof of Theorem~\ref{Intro prop2.1}.  
For $v \neq 0 \mod q$, we have
\begin{align*}
S(q, v, H) &\sim \frac{1}{2K^2 \phi(q)} \sum_{\chi \neq \chi_0} \chi(v)^{-1} C_{q, \chi} + \frac{1}{\phi(q)} S(H, \chi_0) \\
&\sim \frac{H}{q} + \frac{1}{2K^2 \phi(q)} \sum_{\chi \neq \chi_0} \chi(v)^{-1} C_{q, \chi} + \frac{1}{2  \pi K^2 \phi(q)}\int_{1/2+\varepsilon}^1 \frac{F_{\chi_0}(\sigma) H^{\sigma-1}}{|\sigma-1|^{1/2}} d\sigma
\end{align*}
and
\begin{align*}
S(q, 0, H) &\sim S(H) - \frac{\phi(q)}{q} H -  \frac{1}{2  \pi K^2}\int_{1/2+\varepsilon}^1 \frac{F_{\chi_0}(\sigma) H^{\sigma-1}}{|\sigma-1|^{1/2}} d\sigma \\
&\sim \frac{H}{q} + 
	\frac{1}{ \pi  K^2}\int_{1/2+\varepsilon}^1 \frac{F'(\sigma) H^{\sigma-1} + F(\sigma) H^{\sigma-1} \log{H}}{|\sigma-1|^{1/2}} d\sigma - \frac{1}{2  \pi K^2}\int_{1/2+\varepsilon}^1 \frac{A_q(\sigma) F(\sigma) H^{\sigma-1}}{|\sigma-1|^{1/2}} d\sigma \\
	&\sim \frac{H}{q} + 
	\frac{1}{ \pi  K^2}\int_{1/2+\varepsilon}^1 \frac{F'(\sigma) H^{\sigma-1} + F(\sigma) H^{\sigma-1} (\log{H} - A_q(\sigma)/2)}{|\sigma-1|^{1/2}} d\sigma .	\end{align*}\end{proof}

\section{Heuristic in the case of $r$-uplets}\label{section-general-conj}

As in \cite{LOS}, the essence for the general conjecture in the case of the distribution of $r$ consecutive \sos is really in  the particular case $r=2$ that we explained in more details. In this section we present the heuristic for Conjecture~\ref{general-conjecture-intro} with hightlights on the differences from the case $r=2$, for this we follow again the exposition of \cite{LOS}.
Let $r \geq 3$, $q\equiv 1 \pmod{4}$  and $\mathbf{a} = (a_1,\dots,a_r) \in \N^{r}$ be fixed.
We write
\begin{align*}
	N(x;q,\mathbf{a}) = \sum_{\substack{n \leq x \\ n \equiv a_1 \pmod{q}}}  &\sum_{\substack{h_2,\dots,h_r>0 \\ h_\ell \equiv a_\ell-a_{\ell-1} \pmod{q}}} \mathbf{1}_{\mathbf{E}}(n)\prod_{i=2}^{r} \mathbf{1}_{\mathbf{E}}(n+h_2+\dots+h_i) \\
	&\quad \times \prod_{t = 1}^{h_i-1} (1 - \mathbf{1}_{\mathbf{E}}(n+h_2 + \dots + h_{i-1} + t)).
\end{align*}
As in Section~\ref{section-heuristic-1}, we use the notation $\widetilde{ \mathbf{1}}_{\mathbf{E}}(n)= \mathbf{1}_{\mathbf{E}}(n)-\frac{K}{\sqrt{\log n}}$, approximate all the $\log(n+t)$ by $\log x$, expand out the products and apply the Hardy--Littlewood Conjecture~\eqref{def-S0q} in our context, neglecting the terms corresponding to products over more than $3$ terms thanks to Theorem~\ref{Prop bound for D with large sets}.
Thus, heuristically, up to error of size $x(\log x)^{-\frac{r}2 - \frac{1}{4} + \varepsilon}$, we obtain
\begin{align*}
	N(x;q,\mathbf{a}) \sim \frac{x}{q}\Big(\frac{K}{\sqrt{\log x}} \Big)^{r}\alpha(x)^{-r+1}\big(\mathcal{D}_0(\mathbf{a},x) + \mathcal{D}_1(\mathbf{a},x) + \mathcal{D}_2(\mathbf{a},x)\big),
\end{align*}
where $\alpha(x) = 1- \frac{K}{\sqrt{\log x}}$ and
\begin{align*}
	\mathcal{D}_0(\mathbf{a},x)	 &=  \sum_{\substack{h_2,\dots,h_r>0 \\ h_{\ell} \equiv a_{\ell}-a_{\ell-1} \pmod{q}}}\Big( 1 + \sum_{1\leq i <j\leq r}\mathfrak{S}_{0}(\lbrace 0, h_{i+1}+ \dots + h_j\rbrace) \Big) \alpha(x)^{h_2+\dots+h_r} \\
	\mathcal{D}_1(\mathbf{a},x)	 &=  - \frac{K}{\alpha(x)\sqrt{\log x}} \sum_{\substack{h_2,\dots,h_r>0 \\ h_{\ell} \equiv a_{\ell}-a_{\ell-1}\pmod{q}}} \sum_{i=1}^{r}\sum_{j=2}^{r}\sum_{t=1}^{h_j -1}\mathfrak{S}_{0}(\lbrace h_{2}+ \dots + h_i, h_{2}+ \dots + h_{j-1} + t\rbrace)  \alpha(x)^{h_2+\dots+h_r} \\
	\mathcal{D}_2(\mathbf{a},x)	 &= \frac{K^2}{\alpha(x)^2 \log x} \sum_{\substack{h_2,\dots,h_r>0 \\ h_{\ell} \equiv a_{\ell}-a_{\ell-1} \pmod{q}}} \sum_{2\leq i \leq j\leq r}\sum_{t_1=1}^{h_i -1}\sum_{\substack{t_2=1 \\ t_2 >t_1 \text{ if } i =j}}^{h_j -1}\mathfrak{S}_{0}(\lbrace t_1, h_{i}+ \dots + h_{j-1} + t_2\rbrace)  \alpha(x)^{h_2+\dots+h_r}.
\end{align*}
Let us begin with studying $\mathcal{D}_0(\mathbf{a},x)$ in more details.
As in Section~\ref{section-heuristic-1}, we write $H = -\frac{1}{\log \alpha(x)} \iff \alpha(x)^h = e(-h/H)$.
The contribution of  the $1$  to $\mathcal{D}_0(\mathbf{a},x)$ gives 
\begin{align} \nonumber
	\sum_{\substack{h_2,\dots,h_r>0 \\ h_{\ell} \equiv a_{\ell}-a_{\ell-1} \pmod{q}}} e^{-(h_2+\dots+h_r)/H} &= \prod_{\ell=2}^{r}\Big( \frac{H}{q} + f(a_{\ell} - a_{\ell-1};q) + O(H^{-1})\Big) \\
	\label{first-D0}
	&= \Big(\frac{H}{q}\Big)^{r-1} + \Big(\frac{H}{q}\Big)^{r-2}\sum_{\ell=2}^{r}f(a_{\ell} - a_{\ell-1};q) + O(H^{r-3}).
\end{align}

For the contribution of $\sum_{1 \leq i < j \leq r}$ to $\mathcal{D}_0(\mathbf{a},x)$, we first make a change of variables by writing $j= i + k$, and we exchange the order of summation, which gives
\begin{align*}
	\sum_{\substack{1 \leq i \leq r-1 \\ 1 \leq k \leq r-i}}  & \left( \sum_{\substack{h_{2},\dots,h_{i}, h_{i+k+1}, \dots, h_r >0 \\ h_{\ell} \equiv a_{\ell}-a_{\ell-1} \pmod{q}}}\
	e^{-(h_{2}+\dots+h_{i} + h_{i+k+1} + h_r)/H} \right) \\ & \times \left( \sum_{\substack{h_{i+1},\dots,h_{i+k}>0 \\ h_{\ell} \equiv a_{\ell}-a_{\ell-1} \pmod{q}}}\mathfrak{S}_{0}(\lbrace 0, h_{i+1}+ \dots + h_{i+k}\rbrace)  e^{-(h_{i+1}+\dots+h_{i+k})/H} \right).
\end{align*}

For each fixed $i,k$, the second factor in the inner sum above is
\begin{align} \nonumber
	& \sum_{\substack{h>0 \\ h \equiv a_{i+k} - a_i \pmod{q}}}\mathfrak{S}_{0}(\lbrace 0, h\rbrace)  e^{-h/H}\sum_{\substack{h_{i+1},\dots,h_{i+k}>0 \\ h_{\ell} \equiv a_{\ell}-a_{\ell-1} \pmod{q} \\ h_{i+1} +\dots +h_{i+k} =h}}1 \\
	\label{second-D0}
	&\quad = \frac{1}{(k-1)! q^{k-1}}\sum_{\substack{h>0 \\ h \equiv a_{i+k} - a_i \pmod{q}}}\mathfrak{S}_{0}(\lbrace 0, h\rbrace)  e^{-h/H}(h^{k-1} + O(h^{k-2})). 
\end{align}

We need some notation, generalizing the functions defined in Section~\ref{section-heuristic-2}.
For $v,k \in \N$, let
\begin{align*}
	S^{(k)}(q,v,H) &:= \sum_{\substack{h \geq 1 \\ h \equiv v \mod q}} \mathfrak{S}(\{ 0,h \})h^ke^{-h/H} \\
	S_0^{(k)}(q,v,H) &:= \sum_{\substack{h \geq 1 \\ h \equiv v \mod q}} \mathfrak{S}_0(\{ 0,h \})h^ke^{-h/H} \\
	S^{(k)}(H) &:= \sum_{h \geq 1 } \mathfrak{S}(\{ 0,h \})h^ke^{-h/H} \\
	S_0^{(k)}(H) &:= \sum_{h \geq 1} \mathfrak{S}_0(\{ 0,h \})h^ke^{-h/H} .
\end{align*}
Note that $S^{(0)}(q,v,H) = S(q,v,H)$ as defined in~\eqref{eq def S(q,v;H)-S0(q,v;H)}.
Moreover, we have
\begin{align*}
	S_0^{(k)}(H) &=  S^{(k)}(H) - \sum_{h \geq 1 } h^ke^{-h/H}  =  S^{(k)}(H) - k! H^{k+1} + O(H^{k-1}) \\
	\text{and } \qquad
	S_0^{(k)}(q,v,H) &=  S^{(k)}(q,v,H) - \sum_{\substack{h \geq 1 \\ h \equiv v \mod q}} h^ke^{-h/H}  =  S^{(k)}(q,v,H) - \frac{k!}{q} H^{k+1} + O(H^{k-1}).
\end{align*}

\begin{proposition} \label{Lem Sk}
	Let $q\equiv 1 \pmod{4}$ be a prime. For any $k \geq 1$ , we have
	\begin{align*}
		S^{(k)}(H) = k! H^{k+1} - \frac{(k-1)!}{K\sqrt{\pi}}H^{k}(\log H)^{-\frac12} + O(H^{k}(\log H)^{-\frac32}).
	\end{align*}
	and
	\begin{align*}
		S^{(k)}(q,v,H) = \begin{cases}
			\frac{k!}{q}H^{k+1} + O(H^{k}(\log H)^{-\frac32}) & \text{ if } v \not\equiv 0 \pmod{q} \\
			\frac{k!}{q}H^{k+1} - \frac{(k-1)!}{K\sqrt{\pi}}H^{k}(\log H)^{-\frac12} + O(H^{k}(\log H)^{-\frac32}) & \text{ if } v \equiv 0 \pmod{q}.
		\end{cases}
	\end{align*}
\end{proposition}

We observe that the secondary term is relatively smaller in the case $k\geq 1$ than in the case $k=0$ (which is Theorem~\ref{Intro prop2.1}). This is due to the fact that the order of the singularity is smaller when $k \geq 1$ as the poles of the functions $\zeta(k+1+s)$ and $\Gamma(s)$ do not coincide. 
Note also that, similarly to Theorem~\ref{Intro prop2.1}, one could develop the secondary term using a sum of descending powers of $\log H$ with explicit coefficients. We chose not to do so in this statement as we are mostly interested in the direction of the bias in the distribution of consecutive \sos in arithmetic progressions.

\begin{proof}
	The proof is similar to the proof of Theorem~\ref{Intro prop2.1}, and we just give a sketch.
	The main idea is to approximate the sums $S^{(k)}(H)$ and $S^{(k)}(q,v,H)$  via contour integration of the shifted functions $D(s)$ and $D_{\chi}(s-k)$ (for $\chi$ a character modulo $q$) respectively, where the functions $D(s)$ and $D_{\chi}$ are as defined in Section~\ref{proof-prop-section}. 
	For $k\geq 1$ and $\chi \neq \chi_0$, the function $\Gamma(s) D_{\chi}(s-k)$ is analytic on a zero free region containing the line $\re(s)=k$, thus, we have 
	\begin{align*}
		\sum_{h \geq 1 } 2K^2\mathfrak{S}(\{ 0,h \})\chi(h)h^ke^{-h/H} = O(H^{k}e^{-c\sqrt{\log H}}).
	\end{align*}
	For $S^{(k)}(H)$, the function $\Gamma(s) D(s-k)$  has a simple pole at $s=k+1$ with residue $2K^2\Gamma(k+1)$ and an essential singularity at $s= k$ of the shape $(s-k)^{-\frac12}$. 
	We deduce that
	\begin{align*}
		\sum_{h \geq 1 } 2K^2\mathfrak{S}(\{ 0,h \})h^ke^{-h/H} = H^{k+1}2K^2\Gamma(k+1) - 2\Gamma(k)\frac{K}{\sqrt{\pi}}H^{k}(\log H)^{-\frac12} + O(H^{k}(\log H)^{-\frac32}),
	\end{align*}
	which gives
	\begin{align*}
		S^{(k)}(H) =  H^{k+1}\Gamma(k+1) - \Gamma(k)\frac{1}{K\sqrt{\pi}}H^{k}(\log H)^{-\frac12} + O(H^{k}(\log H)^{-\frac32}).
	\end{align*}
	In the case $\chi = \chi_0$,  the function $\Gamma(s) D_{\chi_0}(s-k)$ 
	has a simple pole at $s=k+1$ with residue $2K^2\Gamma(k+1)\frac{\phi(q)}{q}$ and an essential singularity at $s= k$ of the shape $(s-k)^{\frac12}$. 
	We deduce
	\begin{align*}
		\sum_{h \geq 1 } 2K^2\mathfrak{S}(\{ 0,h \})\chi_0(h)h^ke^{-h/H} = H^{k+1}2K^2\Gamma(k+1)\frac{\phi(q)}{q} + O(H^{k}(\log H)^{-\frac32}).
	\end{align*}
	Finally, we obtain the expressions in the statement of Lemma~\ref{Lem Sk} using the orthogonality relations in the case~$v\not\equiv 0 \pmod{q}$, and the case~$v \equiv 0 \pmod{q}$ is then deduced by substracting the contributions of all non-zero~$v$'s to~$S^{(k)}(H)$.
\end{proof}

Using Lemma~\ref{Lem Sk} and~\eqref{first-D0},~\eqref{second-D0}, we get
\begin{align*}
	\mathcal{D}_0(\mathbf{a},x)	
	=& \Big(\frac{H}{q}\Big)^{r-1} + \Big(\frac{H}{q}\Big)^{r-2}\sum_{i=2}^{r}f(a_i - a_{i-1};q)  \\
	& \quad + \sum_{i=1}^{r-1}\sum_{k=1}^{r-i} \Big(\frac{H}{q}\Big)^{r-k-1} \frac{1}{(k-1)! q^{k-1}}S_0^{(k-1)}(q,a_{i+k}-a_i,H)
	+ O(H^{r-3}) \\ 
	=& \Big(\frac{H}{q}\Big)^{r-1}  + \Big(\frac{H}{q}\Big)^{r-2}\sum_{i=1}^{r-1}\Big( S_0(q,a_{i+1}-a_i,H) + f(a_{i+1} - a_{i};q)  \Big) \\
	&\quad + \Big(\frac{H}{q}\Big)^{r-2}\frac{(\log H)^{-\frac12}}{K\sqrt{\pi}(k-1)}\sum_{\substack{1\leq i , j \leq r\\ j>i+1}} \delta(a_{j}\equiv a_i)
	+ O(H^{r-2}(\log H)^{-\frac32}). 
\end{align*}

Let us now study $\mathcal{D}_1(\mathbf{a},x)$.
We first write
\begin{align} \nonumber
	\mathcal{D}_1(\mathbf{a},x)	 = - \frac{K}{\alpha(x)\sqrt{\log x}} \sum_{\substack{h_2,\dots,h_r>0 \\ h_{\ell} \equiv a_{\ell}-a_{\ell-1} \pmod{q}}} 
	\Bigg(	&\sum_{\substack{2\leq j\leq r\\ 2 \leq i \leq j}}
	\sum_{t=1}^{h_j -1}   \mathfrak{S}_{0}(\lbrace 0, h_{i}+ \dots + h_{j-1} + t\rbrace)  
	e^{-(h_2+\dots+h_r)/H} \\ \label{first-step-D1}
	&+ \sum_{\substack{2\leq j\leq r\\ j \leq i \leq r}}
	\sum_{t=1}^{h_j -1}   \mathfrak{S}_{0}(\lbrace h_{j}+ \dots + h_{i}, t \rbrace) e^{-(h_2+\dots+h_r)/H} \Bigg)
\end{align}
We focus on the first inner sum of~\eqref{first-step-D1}. Exchanging the order of summation, for each fixed $i$ and $j = i+k \geq i$,  we have
\begin{align*}
	&\sum_{\substack{h_{i},\dots,h_{i+k}>0 \\ h_{\ell} \equiv a_{\ell}-a_{\ell-1} \pmod{q}}}\sum_{t=1}^{h_{i+k} -1}\mathfrak{S}_{0}(\lbrace 0, h_{i}+ \dots + h_{i+k-1} + t\rbrace) e^{-(h_{i}+\dots+h_{i+k})/H}\\
	& \times  \sum_{\substack{h_{2},\dots,h_{i-1}, h_{i+k+1}, \dots, h_r >0 \\ h_{\ell} \equiv a_{\ell}-a_{\ell-1} \pmod{q}}} e^{-(h_{2}+\dots+h_{i-1}
		+ h_{i+k+1} + \dots + h_r )/H} 
\end{align*}
The second sum of the above is evaluated by~\eqref{first-D0}, and
\begin{align*}
	&\sum_{\substack{h_{i},\dots,h_{i+k}>0 \\ h_{\ell} \equiv a_{\ell}-a_{\ell-1} \pmod{q}}}\sum_{t=1}^{h_{i+k} -1}\mathfrak{S}_{0}(\lbrace 0, h_{i}+ \dots + h_{i+k-1} + t\rbrace) e^{-(h_{i}+\dots+h_{i+k})/H} \\
	&\quad = 
	\sum_{u>0}\mathfrak{S}_{0}(\lbrace 0, u\rbrace)
	\sum_{\substack{h>u \\ h \equiv a_{i+k} - a_{i-1} \pmod{q}}}e^{-h/H}
	\sum_{\substack{h_{i},\dots,h_{i+k-1}>0 \\ h_{\ell} \equiv a_{\ell}-a_{\ell-1} \pmod{q} \\ h_{i} +\dots +h_{i+k-1} <u}}
	\sum_{\substack{h_{i+k}> 0 \\ h_{i+k} \equiv a_{i+k}-a_{i+k-1} \pmod{q} \\ h_{i} +\dots +h_{i+k} = h}}  1 \\
	&\quad = 
	\sum_{u>0}\mathfrak{S}_{0}(\lbrace 0, u\rbrace)e^{-u/H}
	\sum_{\substack{h'>0 \\ h' \equiv a_{i+k} - a_{i-1} -u \pmod{q}}}e^{-h'/H}
	\sum_{\substack{h_{i},\dots,h_{i+k-1}>0 \\ h_{\ell} \equiv a_{\ell}-a_{\ell-1} \pmod{q} \\ h_{i} +\dots +h_{i+k-1} <u}}1 \\
	&\quad = 
	\sum_{u>0}\mathfrak{S}_{0}(\lbrace 0, u\rbrace)e^{-u/H}\Big(\frac{1}{k!}\big(\frac{u}{q}\big)^{k} + O(u^{k-1})  \Big) \Big(\frac{H}{q} + O(1)\Big) \\
	&\quad = 
	\frac{H}{k!q^{k+1}} S^{(k)}_0(H) + O(H^{k+\varepsilon})
\end{align*}
We get a similar estimate for the second inner sum of~\eqref{first-step-D1} involving $\mathfrak{S}_{0} (\lbrace h_j + \dots + h_i, t \rbrace)$ by making a change of variable to replace it by $\mathfrak{S}_{0} (\lbrace 0, r + h_{j+1} \dots + h_i \rbrace)$ with $r=h_j-t$, 
and we obtain
\begin{align*}
	\mathcal{D}_1(\mathbf{a},x)	 &= - 2\frac{K}{\alpha(x)\sqrt{\log x}} \sum_{i=2}^{r}\sum_{k=0}^{r-i} \displaystyle \Big(\frac{H}{q}\Big)^{r-2-k} \frac{H}{q^{k+1}k!}S^{(k)}_0(H)  + O(H^{r-3+\varepsilon})  \\
	&= - 2\frac{K}{\alpha(x)\sqrt{\log x}} \frac{H^{r-1}}{q^{r-1}}\Big((r-1)S_0(H) -  \frac{(\log H)^{-\frac12}}{K\sqrt{\pi}} \sum_{k=1}^{r-2} \frac{(r-1-k)}{k} \Big) + O(H^{r-2}(\log H)^{-\frac32}) \\
\end{align*}

The same ideas are used to estimate $\mathcal{D}_2(\mathbf{a},x)$.
Let $i$ and $j=i+k \geq i$ be fixed, and let us study the sum in $\mathcal{D}_2(\mathbf{a},x)$.
In the case $i=j$, this is
\begin{align*}
	\sum_{\substack{h_i>0 \\ h_i \equiv a_i-a_{i-1} \pmod{q}}} \sum_{1 \leq t_1 < t_2 \leq h_i -1}\mathfrak{S}_{0}(\lbrace t_1, t_2\rbrace)  e^{-h_i /H} = \Big(\frac{H^2}{q} + O(H)\Big)S_0(H)
\end{align*}
as we already saw in the case $r=2$.
In the case $k\geq 1$,
we have
\begin{align*}
	&\sum_{\substack{h_i,\dots,h_{i+k}>0 \\ h_{\ell} \equiv a_{\ell}-a_{\ell-1} \pmod{q}}} \sum_{t_1=1}^{h_i -1}\sum_{t_2=1}^{h_{i+k} -1}\mathfrak{S}_{0}(\lbrace t_1, h_{i}+ \dots + h_{i+k-1} + t_2\rbrace)  e^{-(h_i+\dots+h_{i+k})/H} \\
	& \quad = \sum_{1\leq t_1 < t'_2} \mathfrak{S}_{0}(\lbrace 0, t'_2 - t_1\rbrace) \sum_{h > t'_2}e^{-h/H}
	\sum_{\substack{h_{i},\dots,h_{i+k-1}>0 \\ h_{\ell} \equiv a_{\ell}-a_{\ell-1} \pmod{q} \\ t_1 < h_{i} +\dots +h_{i+k-1} <t'_2}}
	\sum_{\substack{h_{i+k}> 0 \\ h_{i+k} \equiv a_{i+k}-a_{i+k-1} \pmod{q} \\ h_{i} +\dots +h_{i+k} = h}}  1 \\
	& \quad = \sum_{u>0} \mathfrak{S}_{0}(\lbrace 0, u\rbrace) \sum_{t'_2 >u} \sum_{h > t'_2}e^{-h/H}
	\Big(\frac{1}{k!}\big(\frac{u}{q}\big)^{k} + O(u^{k-1})  \Big)\\	
	& \quad = \sum_{u>0} \mathfrak{S}_{0}(\lbrace 0, u\rbrace)e^{-u/H} 
	\Big(\frac{1}{k!}\big(\frac{u}{q}\big)^{k} + O(u^{k-1})  \Big) \Big(\frac{H^2}{q} + O(H)\Big)\\	
	& \quad = 	\frac{H^2}{k!q^{k+1}} S^{(k)}_0(H) + O(H^{k+1+\varepsilon}).
\end{align*}
We deduce that
\begin{align*}
	\mathcal{D}_2(\mathbf{a},x)	 
	&= \frac{K^2}{\alpha(x)^2 \log x} \sum_{i =2}^{r}\sum_{k=0}^{r-i}\Big(\tfrac{H}{q}\Big)^{r-2-k} \frac{H^2}{k!q^{k+1}} S^{(k)}_0(H) + O(H^{r-3+\varepsilon}) \\
	&= \frac{K^2}{\alpha(x)^2 \log x} \frac{H^r}{q^{r-1}}\Big((r-1)S_0(H) -  \frac{(\log H)^{-\frac12}}{K\sqrt{\pi}} \sum_{k=1}^{r-2} \frac{(r-1-k)}{k}  \Big) + O(H^{r-2}(\log H)^{-\frac32}).
\end{align*}

Wrapping up, we obtain
\begin{align*}
	N(x;q,\mathbf{a}) &= \frac{x}{q}\Big(\frac{K}{\sqrt{\log x}} \Big)^{r}\alpha(x)^{-r+1}
	\Big(\big(\tfrac{H}{q}\big)^{r-1}  +  \big(\tfrac{H}{q}\big)^{r-2}\sum_{i=1}^{r-1}\big(\mathcal{D}_0(a_i,a_{i+1},x) - \tfrac{H}{q} + \mathcal{D}_1(a_i,a_{i+1},x) + \mathcal{D}_2(a_i,a_{i+1},x)\big) \\
	&\quad -\big(\tfrac{H}{q}\big)^{r-2}\frac{(\log H)^{-\frac12}}{K\sqrt{\pi}}\big( \sum_{k=1}^{r-2}  \sum_{i=1}^{r-1-k} \frac{\delta(a_{i+k+1}\equiv a_i) - \frac{1}{q}}{k}   \big)
	+ O(H^{r-2}(\log H)^{-\frac32}) \Big).
\end{align*}
Then using $H = \frac{\sqrt{\log x}}{K} - \frac12 + O((\log x)^{-\frac{1}{2}})$, and the estimates for $\mathcal{D}_i(a,b,x)$, $i = 0,1,2$ from Theorem~\ref{Intro prop2.1} we obtain Conjecture~\ref{general-conjecture-intro}.

\section{Numerical data}  \label{section-numerical}

We present in this section some numerical data testing the approximation of Conjecture~\ref{main-conjecture-intro} for $N(x; q, (a,b))$. One of the challenges of the numerical testing is the change of scale introduced by the change of variable~\eqref{changeofvariables}, which gives $H = \sqrt{\log{x}}/K$.
The actual value of $N(x; q,(a,b))$ were obtained by using SageMath~\cite{Sage} on about 20 CPU cores in a Linux cluster for a couple of months, which allow us to take $x=10^{12}$. But then, $H  \approx 6.356$ in Theorem~\ref{Intro prop2.1}, which is very small even for this large value of $x$.

 There are some technical methods for computing the Euler products, whenever they converge, and their derivatives with enough precision,
and we used the following equality, which gives us a faster convergence: 
$$
\prod_{p\equiv 3\pmod{4}}\left(1-p^{-2s}\right) = \prod_{1\le j\le J}\left(\frac{L(2^js,\chi_4)}{\zeta(2^js)(1-2^{-2^js})}\right)^{1/2^j}\prod_{p\equiv 3\pmod{4}}\left(1-p^{-2^{J+1}s}\right)^{1/2^J}.
$$
Note that the rightmost hand side product converges much faster than the left hand side one. Also, its derivatives can be computed by taking the derivatives of the right hand side instead so that one might obtain some recursive formula.

We present in Table \ref{table:Prop D012 numerical} some numerical data for Conjecture~\ref{main-conjecture-intro}, for $q=5$ and $x=10^{12}$.  
There are 25 cases for $N(x; q, (a,b))$ in Table \ref{table:Prop D012 numerical},
but the conjectural asymptotic of Conjecture ~\ref{main-conjecture-intro} only depends on $b-a \mod q$, and there are then unavoidable fluctuations in the data for various pairs $(a,b)$ with the same value of $b-a \mod 5$. 
The fit between the numerical data and the conjecture is slightly better when $b-a \not\equiv 0 \mod 5$. The numerical data is also influenced by the bias of Theorem \ref{thm-SOS-AP}, which is of smaller magnitude that the bias of Conjecture~\ref{main-conjecture-intro} but in the opposite direction, and the data when $a=b=0$ in particular
shows the influence of both biases.  
We have used several asymptotic approximations of our conjecture in Table  \ref{table:Prop D012 numerical}.  
We used Conjecture~\ref{main-conjecture-intro}  as such with $J=1$ (the column labelled ``Conjecture~\ref{main-conjecture-intro}''), and we also used 
 the more complicated expression of Proposition~\ref{Prop D012 numerical} for $\mathcal{D}_0(a,b;x)+\mathcal{D}_1(a,b;x)+\mathcal{D}_2(a,b;x) $  in  ~\eqref{total-count-D012}, where we evaluate the exponential sums $E(q,v;H)$ exactly for each residue class (recall that $H = \sqrt{\log{x}}/K \approx 6.356$ when $x=10^{12}$).  We then replaced  $S_0(q,v;H)$  in that expression by the approximation of Theorem~\ref{Intro prop2.1} with $J=1$ (the column labelled ``Theorem~\ref{Intro prop2.1}''), and by the actual numerical value of $S_0(q,v;H)$ (the column labelled ``$S_0(q,v;H)$'').

We also present some numerical data for Theorem~\ref{Intro prop2.1} in Table~\ref{table:S(q,0,H)} and~\ref{table:S(q,3,H)} for larger values of $H$. We tested the asymptotic of Theorem~\ref{Intro prop2.1} for $J=1,2,3$ and
the  integral formula of Proposition  \ref{prop-lucia-2} for various values of $H$. For $H \approx 6.356$, larger values of $J$ or the integral formula of Proposition  \ref{prop-lucia-2} 
are not approximating well $S(q,v;H)$, but one can see the fit for larger values of $H$. The values of the constants $c_0(2), c_0(3), c_1(2), c_1(3)$ can be computed by taking more terms in the Taylor expansions of the proof of Theorem~\ref{Intro prop2.1}, similarly to the computations of  $c_0(1), c(1)$ in Section  \ref{proof-prop-section}.
We did not include those computations (which are lengthy but straighforward and not very interesting) in the paper. The numerical values are
\begin{align*}
&c_0(1) \approx 0.604541230, \;\; c_0(2) \approx 0.696827721, \;\;c_0(3) \approx 1.185903185\\
&c_1(1) \approx -0.167588374, \;\; c_1(2) \approx -0.054190676, \;\;c_1(3) \approx -0.328019051.
\end{align*}
\vspace{0.5cm}

\vfill\eject

\begin{table}[htb!]
\begin{tabular}{cc|cccc|ccc}
\hline 
\makecell{$a$} & \makecell{$b$} & \makecell{$N(x;q,(a,b))$} & 
\makecell{$S_0(q,v;H)$}
& \makecell{Theorem \ref{Intro prop2.1}}
& \makecell{Conjecture~\ref{main-conjecture-intro}}
& \makecell{Error1} & \makecell{Error2}  & \makecell{Error3} \\
\hline
\hline
  & $0$ & 4 108 $\cdot 10^6$ & 3 585 $\cdot 10^6$ & 3 219 $\cdot 10^6$  & 3 919 $\cdot 10^6$ & 1.1461 & 1.2763  & 1.0483\\
  & $1$ & 7 153 $\cdot 10^6$ & 6 949 $\cdot 10^6$ & 6 904 $\cdot 10^6$  & 6 841 $\cdot 10^6$ & 1.0294 & 1.0360  & 1.0457\\
$0$ & $2$ & 5 604 $\cdot 10^6$ & 5 430 $\cdot 10^6$ & 5 493 $\cdot 10^6$  & 5 426 $\cdot 10^6$ & 1.0320 & 1.0203  & 1.0329\\
  & $3$ & 8 055 $\cdot 10^6$ & 7 487 $\cdot 10^6$ & 7 858 $\cdot 10^6$  & 7 153 $\cdot 10^6$ & 1.0759 & 1.0250  & 1.1261\\
  & $4$ & 5 780 $\cdot 10^6$ & 5 626 $\cdot 10^6$ & 5 603 $\cdot 10^6$  & 5 738 $\cdot 10^6$ & 1.0274 & 1.0317  & 1.0073\\
\hline
  & $0$ & 5 777 $\cdot 10^6$ & 5 626 $\cdot 10^6$ & 5 603 $\cdot 10^6$ & 5 738 $\cdot 10^6$ &  1.0269 & 1.0312  & 1.0068\\
  & $1$ & 3 765 $\cdot 10^6$ & 3 585 $\cdot 10^6$ & 3 219 $\cdot 10^6$ & 3 919 $\cdot 10^6$ & 1.0503 & 1.1697  & 0.9607\\
$1$ & $2$ & 6 870 $\cdot 10^6$ & 6 949 $\cdot 10^6$ & 6 904 $\cdot 10^6$ & 6 841 $\cdot 10^6$ & 0.9886 & 0.9950  & 1.0043\\
  & $3$ & 5 354 $\cdot 10^6$ & 5 430 $\cdot 10^6$ & 5 493 $\cdot 10^6$ & 5 426 $\cdot 10^6$ &  0.9860 & 0.9747  & 0.9868\\
  & $4$ & 7 742 $\cdot 10^6$ & 7 487 $\cdot 10^6$ & 7 858 $\cdot 10^6$ & 7 153 $\cdot 10^6$ &  1.0341 & 0.9853  & 1.0824\\
\hline
  & $0$ & 8 050 $\cdot 10^6$ & 7 487 $\cdot 10^6$ & 7 858 $\cdot 10^6$ & 7 153 $\cdot 10^6$ & 1.0752 & 1.0244  & 1.1254\\
  & $1$ & 5 516 $\cdot 10^6$ & 5 626 $\cdot 10^6$ & 5 603 $\cdot 10^6$ & 5 738 $\cdot 10^6$ & 0.9804 & 0.9845  & 0.9613\\
$2$ & $2$ & 3 755 $\cdot 10^6$ & 3 585 $\cdot 10^6$ & 3 219 $\cdot 10^6$  & 3 919 $\cdot 10^6$ & 1.0474 & 1.1664  & 0.9580\\
  & $3$ & 6 838 $\cdot 10^6$ & 6 949 $\cdot 10^6$ & 6 904 $\cdot 10^6$  & 6 841 $\cdot 10^6$ & 0.9840 & 0.9903  & 0.9996\\
  & $4$ & 5 351 $\cdot 10^6$ & 5 430 $\cdot 10^6$ & 5 493 $\cdot 10^6$  & 5 426 $\cdot 10^6$ & 0.9853 & 0.9741  & 0.9861\\
\hline
  & $0$ & 5 609 $\cdot 10^6$ & 5 430 $\cdot 10^6$ & 5 493 $\cdot 10^6$  & 5 426 $\cdot 10^6$ & 1.0330 & 1.0212  & 1.0338\\
  & $1$ & 7 718 $\cdot 10^6$ & 7 487 $\cdot 10^6$ & 7 858 $\cdot 10^6$  & 7 153 $\cdot 10^6$ & 1.0309 & 0.9822  & 1.0790\\
$3$ & $2$ & 5 549 $\cdot 10^6$ & 5 626 $\cdot 10^6$ & 5 603 $\cdot 10^6$  & 5 738 $\cdot 10^6$ & 0.9863 & 0.9904  & 0.9670\\
  & $3$ & 3 765 $\cdot 10^6$ & 3 585 $\cdot 10^6$ & 3 219 $\cdot 10^6$  & 3 919 $\cdot 10^6$ & 1.0503 & 1.1697  & 0.9607\\
  & $4$ & 6 867 $\cdot 10^6$ & 6 949 $\cdot 10^6$ & 6 904 $\cdot 10^6$  & 6 841 $\cdot 10^6$ & 0.9882 & 0.9946  & 1.0039\\
\hline
  & $0$ & 7 156 $\cdot 10^6$ & 6 949 $\cdot 10^6$ & 6 904 $\cdot 10^6$  & 6 841 $\cdot 10^6$ & 1.0298 & 1.0364  & 1.0461\\
  & $1$ & 5 357 $\cdot 10^6$ & 5 430 $\cdot 10^6$ & 5 493 $\cdot 10^6$  & 5 426 $\cdot 10^6$ & 0.9864 & 0.9752  & 0.9872\\
$4$ & $2$ & 7 731 $\cdot 10^6$ & 7 487 $\cdot 10^6$ & 7 858 $\cdot 10^6$  & 7 153 $\cdot 10^6$ & 1.0326 & 0.9838  & 1.0808\\
  & $3$ & 5 497 $\cdot 10^6$ & 5 626 $\cdot 10^6$ & 5 603 $\cdot 10^6$  & 5 738 $\cdot 10^6$ & 0.9771 & 0.9812  & 0.9580\\
  & $4$ & 3 769 $\cdot 10^6$ & 3 585 $\cdot 10^6$ & 3 219 $\cdot 10^6$  & 3 919 $\cdot 10^6$ & 1.0512 & 1.1707  & 0.9615\\
\hline
\end{tabular}
\caption{The experimental value of $N(x;q,(a,b))$ versus several estimates for Conjecture~\ref{main-conjecture-intro} with $J=1$ for $q=5$ and $x=10^{12}$.
We used Conjecture~\ref{main-conjecture-intro}  as such with $J=1$ (the column labelled ``Conjecture~\ref{main-conjecture-intro}''), and we also used 
 the more complicated expression of Proposition~\ref{Prop D012 numerical} for $\mathcal{D}_0(a,b;x)+\mathcal{D}_1(a,b;x)+\mathcal{D}_2(a,b;x) $  in  ~\eqref{total-count-D012}, where we evaluate the exponential sums $E(q,v;H)$ exactly for each residue class (recall that $H = \sqrt{\log{x}}/K \approx 6.356$ when $x=10^{12}$).  We then replaced  $S_0(q,v;H)$  in that expression by the approximation of Theorem~\ref{Intro prop2.1} with $J=1$ (the column labelled ``Theorem~\ref{Intro prop2.1}''), and by the actual numerical value of $S_0(q,v;H)$ (the column labelled ``$S_0(q,v;H)$''). Error1, Error2, Error3 are the percentage errors for the 4th, 5th and 6th columns, respectively.}
\label{table:Prop D012 numerical}
\end{table}

\begin{table}[h!]
\hspace*{-.5cm}
\centering
\begin{tabular}{c|ccccc||cccc}
\hline
\makecell{$H$} & \makecell{$S(q,0;H)-H/q$} & \makecell{Prop.~\ref{prop-lucia-2}} & \makecell{$J=1$} & \makecell{$J=2$} & \makecell{$J=3$} & \makecell{Prop.~\ref{prop-lucia-2}} & \makecell{$J=1$} & \makecell{$J=2$} & \makecell{$J=3$} \\
\hline
\hline
6.356 & $-0.6093$ & -0.0087 & -0.6889 & -0.4122 & -0.1577 & 70.3362 &  0.8843 & 1.4779 & 3.8630\\
16 & $-0.8852$ & -0.5540 & -1.0240 & -0.8731 & -0.7804 & 1.5980 & 0.8645 & 1.0139 & 1.1343 \\
$10^2$ & $-1.3968$ & 1.2847 &  -1.5059 & -1.4354 & -1.4094 & 1.0862 & 0.9275 & 0.9731 & 0.9910 \\
$10^4$ & $-2.2932$ & -2.2839 &  -2.3289 & -2.3040 & -2.2994 & 1.0041 & 0.9846 & 0.9953 & 0.9973 \\
$10^6$ & $-2.9169$ & -2.9162 & -2.9337 & -2.9201 & -2.9184 & 1.0002 & 0.9943 & 0.9989 & 0.9995 \\
\hline
\end{tabular}
\caption{The numerical value of $S(q,0;H)-H/q$ for $q=5$ and various values of $H$ versus the asymptotic of Proposition~\ref{prop-lucia-2} and Theorem \ref{Intro prop2.1} for $J = 1, 2,3$.
The last 4 columns are the percentage errors.}
\label{table:S(q,0,H)}
\end{table}

\begin{table}[h!]
\hspace*{-.5cm}
\centering
\begin{tabular}{c|ccccc||cccc}
\hline
\makecell{$H$} & \makecell{$S(q,3;H)-H/q$} & \makecell{Prop.~\ref{prop-lucia-2}} & \makecell{$J=1$} & \makecell{$J=2$} & \makecell{$J=3$} & \makecell{Prop.~\ref{prop-lucia-2}} & \makecell{$J=1$} & \makecell{$J=2$} & \makecell{$J=3$} \\
\hline
\hline
6.356 & 0.0327 & 0.0728 & 0.0811 & 0.0596 & -0.0108 & 0.4485 & 0.4029 & 0.5485 & -3.0166 \\
16 & 0.0788 & 0.0919 & 0.1036 & 0.0919 & 0.0663 & 0.8575 & 0.7609 & 0.8581 & 1.1900 \\
$10^2$ & 0.1120 & 0.1171  & 0.1262 & 0.1207 & 0.1135 & 0.9565 & 0.8875 & 0.9278 & 0.9868 \\
$10^4$ & 0.1456 & 0.1461  & 0.1490 & 0.1471 & 0.1458 & 0.9966 & 0.9770 & 0.9899 & 0.9986 \\
$10^6$ & 0.15813 & 0.15819 & 0.1592 & 0.1581 & 0.1577 & 0.9997 & 0.9935 & 1.0001 & 1.0030 \\
\hline
\end{tabular}
\caption{The numerical value of $S(q,3;H)-H/q$ for $q=5$ and various values of $H$ versus the asymptotic of Proposition~\ref{prop-lucia-2} and Theorem \ref{Intro prop2.1} for $J = 1, 2,3$.
The last 4 columns are the percentage errors.}
\label{table:S(q,3,H)}
\end{table}


\newpage
 \begin{thebibliography}{99} 


\bibitem[BW00]{BW} A. Balog, and T. D. Wooley, Trevor, \textit{Sums of two squares in short intervals}. Canad. J. Math. 52 (2000), no. 4, 673--694.

\bibitem[BFT15]{BFT15} W. D. Banks, T. Freiberg and C. L. Turnage-Butterbaugh, \textit{Consecutive primes in tuples}. Acta Arith. 167 (2015), no. 3, 261--266.

\bibitem[BF19]{BF19} L. Bary-Soroker and A. Fehm, \textit{Correlations of sums of two squares and other arithmetic functions in function fields}, Int. Math. Res. Not., 
Volume 2019, Issue 14 (2019), 4469--4515.

\bibitem[BBF18]{BBF18} E. Bank, L. Bary-Soroker, and A. Fehm, \textit{Sums of two squares in short intervals in polynomial rings over finite fields}.
Am. J. Math. 140 (2018), no. 4, 1113--1131.

\bibitem[BYW16]{BYW16} L. Bary-Soroker, Y. Smilansky, A. Wolf, \textit{On the Function Field Analogue of Landau's Theorem on Sums of Squares}.
Finite Fields Appl. 39 (2016), 195--215. 


\bibitem[CK97]{CK} R. D. Connors and J. P. Keating, \textit{Two-Point Spectral Correlations for the Square Billiard}.  J. Phys. A 30 (1997), no. 6, 1817--1830. 

\bibitem[EG18]{EG18} I. Eriksson and L. Gustafsson,  \textit{On the Asymptotic Behaviour of Sums of Two Squares}, Degree project, KTH Royal Institute of Technology, School of Engineering Sciences (2018), 31 pp.


 \bibitem[FKR17]{FKR} T. Freiberg, P. Kurlberg and L. Rosenzweig, 
\textit{Poisson distribution for gaps between sums of two squares and level spacings for toral point scatterers}.
Commun. Number Theory Phys. 11 (2017), no. 4, 837--877. 

\bibitem[Gal76]{Ga1976} P. X. Gallagher, 
\textit{On the distribution of primes in short intervals}.
Mathematika 23 (1976), no. 1, 4--9. 

\bibitem[GR20]{GR20} O. Gorodetsky and B. Rodgers, \textit{The variance of the number of sums of two squares in $\F_q[T]$ in short intervals}, Journal of the AMS, to appear.
\url{https://arxiv.org/pdf/1810.06002.pdf}.



\bibitem[Hoo71]{Hooley1} C. Hooley, \textit{On the intervals between numbers that are sums of two squares}. Acta Math. 127 (1971), 279--297.

\bibitem[Hoo73]{Hooley2} C. Hooley, \textit{On the intervals between numbers that are sums of two squares. II}. J. Number Theory 5 (1973),  215--217.  



\bibitem[Kou19]{koukoulopoulos} D. Koukoulopoulos. \textit{The distribution of prime numbers}.
Graduate Studies in Mathematics, 203. American Mathematical Society, Providence, RI, [2019], xii + 356 pp. 

\bibitem[Lan08]{Lau} E. Landau \textit{\"{U}ber die einteilung der positiven ganzen zahlen in vier klassen nach der mindestzahl der zu ihrer additiven zusammensetzung erforderlichen quadrate}, Arch. Math u. Phys. (3) 13 (1908), 305--312.

\bibitem[LOS16]{LOS} R. J. Lemke Oliver and K. Soundararajan, 
\textit{Unexpected biases in the distribution of consecutive primes}. 
Proc. Natl. Acad. Sci. USA 113 (2016), no. 31, E4446--E4454. 

\bibitem[Luc13]{lucia} Lucia (https://mathoverflow.net/users/38624/lucia), Asymptotic density of k-almost primes, URL (version: 2013-10-31): \url{https://mathoverflow.net/q/146469}.

\bibitem[May16]{Ma16} J. Maynard. \textit{Dense clusters of primes in subsets}. 
Compos. Math. 152 (2016), no. 7, 1517--1554. 

\bibitem[Mai85]{Ma85}  H. Maier, \textit{Primes in short intervals}. Michigan Math. J. 32 (1985), no. 2, 221--225. 

\bibitem[Men17]{Men} I. Mennema,  \textit{The Distribution of Consecutive Square-Free Numbers}, Master Thesis, Department of Mathematics
Leiden University (2017), 31 pp.

\bibitem[Mir49]{Mir}  L. Mirsky,
\textit{On the frequency of pairs of square-free numbers with a given difference},
Bull. Amer. Math. Soc. 55 (1949), 936--939. 


\bibitem[MS04]{MS} H. L. Montgomery and K. Soundararajan, \textit{Primes in Short Intervals.} Comm. Math. Phys. 252 (2004), no. 1--3,  589--617.

\bibitem[MV84]{MVbas} H. L. Montgomery and R. C. Vaughan,  \textit{A basic inequality}. Proceedings of the Congress on Number Theory (Spanish) (Zarauz, 1984), 163--175, Univ. Pa\'i­s Vasco-Euskal Herriko Unib., Bilbao, 1989.

\bibitem[MV86]{MV} H. L. Montgomery and R. C. Vaughan, \textit{On the Distribution of Reduced Residues}.  Ann. of Math. (2) 123 (1986), no. 2, 311--333. 


\bibitem[PARI]{PARI2} The PARI~Group, PARI/GP version \texttt{2.11.1}, Univ. Bordeaux, 2019, \url{http://pari.math.u-bordeaux.fr/}.

\bibitem[Rie65]{Rie} Rieger, G. J.
		\textit{\"Uber die Anzahl der als Summe von zwei Quadraten darstellbaren und in einer primen Restklasse gelegenen Zahlen unterhalb einer positiven Schranke. II.} (German)
		J. Reine Angew. Math. 217 (1965), 200--216.
		
\bibitem[MATH]{Mathematica} Mathematica, Version 12.0, Wolfram Research Inc., 2020, \url{https://www.wolfram.com/mathematica/}.

\bibitem[Ram76]{Ram76}  K. Ramachandra, \textit{Some problems of analytic number theory}, Acta Arithmetica 31 (1976), 313--324. 

\bibitem[S{\etalchar{+}}09]{Sage} \emph{{S}ageMath, the {S}age {M}athematics {S}oftware {S}ystem ({V}ersion 8.6.rc0)}, The Sage Developers, 2019, \url{https://www.sagemath.org}.

\bibitem[Sha64]{Sha} D. Shanks, \textit{The Second-Order Term in the Asymptotic Expansion of ${B}(x)$},
Math. Comp. 18 (1964), 75--86.

\bibitem[Shi00]{Sh00} D. K. L. Shiu,  
\textit{Strings of congruent primes}.
J. London Math. Soc. (2) 61 (2000), no. 2, 359--373.

\bibitem[Smi13]{smilansky} Y. Smilansky. 
\textit{Sums of two squares --- pair correlation and distribution in short intervals}. 
Int. J. Number Theory 9 (2013), no. 7, 1687--1711. 


\bibitem[Sta28]{Sta28} G. K. Stanley, \textit{Two Assertions made by Ramanujan}. J. London Math. Soc. 3 (1928), no. 3, 232--237. 

\bibitem[Sta29]{Sta29} G. K. Stanley, \textit{Corrigenda: Two Assertions made by Ramanujan}. J. London Math. Soc. 4 (1929), no. 1, 32. 

\bibitem[Ten15]{tenenbaum} G. Tenenbaum,  \textit{Introduction to analytic and probabilistic number theory}. Third edition. Translated from the 2008 French edition by Patrick D. F. Ion. Graduate Studies in Mathematics, 163. American Mathematical Society, Providence, RI, 2015. xxiv+629 pp.

\end{thebibliography}
\end{document}